\documentclass[11pt]{amsart}
\usepackage{setspace}
\usepackage{amsmath}
\usepackage{amsfonts}
\usepackage{bm}
\usepackage{amssymb}
\usepackage{amsthm}
\usepackage{geometry}
\newcommand{\mres}{\mathbin{\vrule height 1.6ex depth 0pt width
0.13ex\vrule height 0.13ex depth 0pt width 1.3ex}}
\newtheorem{thm}{Theorem}[section]
\newtheorem{lem}[thm]{Lemma}
\newtheorem{cor}[thm]{Corollary}

\newtheorem{prop}[thm]{Proposition}
\newtheorem{rmk}[thm]{Remark}
\newtheorem{qe}[thm]{Question}
\begin{document}
\title[Minimal hypersurfaces on Spheres and cohomogeneity two min-max theory]{Regularity of Cohomogeneity two equivariant isotopy minimization problems and minimal hypersurfaces with large first Betti number on spheres}
    \author{Dongyeong Ko}
    \address{Department of Mathematics, Massachusetts Institute of Technology, Cambridge, MA 02139, USA}
    \email{dyko@mit.edu}
\maketitle

\begin{abstract}
    We prove the regularity of cohomogeneity two equivariant isotopy minimization problems. Based on this, we develop cohomogeneity two equivariant min-max theory for minimal hypersurfaces proposed by Pitts and Rubinstein in 1987. As an application, for $g \ge 1$ and $4 \le n+1 \le 7$, we construct minimal hypersurfaces $\Sigma_{g}^{n}$ on round spheres $\mathbb{S}^{n+1}$ with $(SO(n-1) \times \mathbb{D}_{g+1})$-symmetry. For sufficiently large $g$, $\Sigma_{g}^{n}$ is a sequence of minimal hypersurfaces with arbitrarily large Betti numbers of topological type $\#^{2g} (S^{1} \times S^{n-1})$ or $\#^{2g+2} (S^{1} \times S^{n-1})$. For $4 \le n+1 \le 6$, $\Sigma^{n}_{g}$ converges to a union of $\mathbb{S}^{n}$ and a Clifford hypersurface $\sqrt{\frac{1}{n}}\mathbb{S}^{1} \times \sqrt{\frac{n-1}{n}} \mathbb{S}^{n-1}$ or $\sqrt{\frac{2}{n}}\mathbb{S}^{2} \times \sqrt{\frac{n-2}{n}} \mathbb{S}^{n-2}$. In particular, for dimensions $5$ and $6$, $\Sigma_{g}^{n}$ has a topological type $\#^{2g} (S^{1} \times S^{n-1})$.
\end{abstract}
\section{Introduction}
It is a fundamental problem in differential geometry to construct and understand minimal hypersurfaces in a given ambient manifold. Among these, minimal hypersurfaces in the round sphere $\mathbb{S}^{n+1}$ have been extensively studied, yet many fundamental questions remain open, particularly in dimensions $n+1 \ge 4$. Embedded minimal hypersurfaces in $\mathbb{S}^{n+1}$ are the links of minimal cones with an isolated singularity at the origin, and therefore have their own significance.

There have been a number of constructions of minimal surfaces on a three dimensional round sphere $\mathbb{S}^{3}$. Starting from simplest examples such as an equatorial $2$-sphere $\mathbb{S}^{2}$ and Clifford torus $(1/\sqrt{2}) \mathbb{S}^{1} \times (1/\sqrt{2}) \mathbb{S}^{1}$ (up to congruence), various theories have yielded results on many different minimal surfaces. Lawson introduced classes of minimal surfaces by developing the theory of equivariant Plateau problem in \cite{lawson1972equivariant} (See related construction in \cite{choe2016new} and \cite{karcher1988new}). Brendle \cite{brendle2013embedded} proved the uniqueness of minimal torus on $\mathbb{S}^{3}$ and Marques-Neves \cite{marques2014min} proved Willmore conjecture and the area bound. Recently, there have been constructions by PDE gluing methods as in \cite{kapouleas2017minimal}, \cite{kapouleas2019minimal}, \cite{kapouleas2010minimal}, and \cite{kapouleas2022minimal}. The min-max method has produced examples of minimal surfaces in $\mathbb{S}^{3}$ as in \cite{ketover2020catenoid}, \cite{ketover2016equivariant}, and \cite{ketover2022flipping}. The eigenvalue optimization approach has produced new examples as in \cite{karpukhin2024embedded} and \cite{karpukhin2025large}. Here we note that many of such examples are obtained by `desingularizing' or `doubling' equatorial spheres or Clifford tori.

In higher dimensional round spheres $\mathbb{S}^{n+1}$ $(n+1 \ge 4)$, fewer examples have been constructed. Provided a closed manifold $(M^{n+1},g)$ with an isometry group $G$ and when a dimension of an orbit space $M/G$ is dimension $2$, Hsiang-Lawson \cite{hsiang1971minimal} constructed classes of minimal hypersurface by developing ODE method by finding geodesics on $M/G$ endowed with a weighted metric by volume of orbits (See \cite{otsuki1970minimal} for the method). In particular, Hsiang and Hsiang-Sterling constructed infinite sequences of minimal hyperspheres on round spheres $\mathbb{S}^{n+1}$ are constructed in \cite{hsiang1983minimal, hsiang1983minimal2, hsiang1986minimal}, which prove nonuniqueness of minimal hyperspheres on high dimensional spheres and hence give a negative answer on the spherical Bernstein problem proposed by Chern. There are constructions by isoparametric hypersurfaces in \cite{cecil2007isoparametric}, \cite{chi2020isoparametric}, \cite{ki1987characterization}, and \cite{solomon1990harmonic}. The Cartan isoparametric hypersurfaces in \cite{ki1987characterization} and \cite{solomon1990harmonic} have a topological type $SO(3)/(\mathbb{Z}_{2} \times \mathbb{Z}_{2})$ in $S^{4}$. More recently, Carlotto and Schulz \cite{carlotto2021minimal} constructed infinitely many distinct minimal hypertori on $\mathbb{S}^{4}$ by refining Hsiang's ODE analysis. Kapouleas-Zou \cite{kapouleas2024minimal} constructed a family of infinitely many minimal hypersurfaces on $\mathbb{S}^{4}$ whose first Betti number goes to infinity by doubling equator $\mathbb{S}^{3}$ with a gluing method.  

There has been a recent remarkable progress on the existence of minimal hypersurfaces via min-max theory. In particular, by Almgren-Pitts min-max theory, it is proven that there are infinitely many closed and embedded minimal hypersurfaces on every closed smooth manifold $(M^{n+1},g)$ ($3 \le n+1 \le 7$) in a series of papers by Marques-Neves \cite{marques2017existence}, Song \cite{song2023existence}, and Zhou's multiplicity one theorem \cite{zhou2020multiplicity} (See \cite{irie2018density}, \cite{liokumovich2018weyl}, and \cite{marques2019equidistribution} for geometric properties). Moreover, recently, Liu \cite{liu2021existence} and Wang (\cite{wang2022min}, \cite{wang2025generic}) developed equivariant version of Almgren-Pitts min-max theory with a Lie group isometry.

However, in contrast to its existence theory, only very few results are known on minimal hypersurfaces with prescribed topology on a closed manifold. In particular, regarding the topology of minimal hypersurfaces, there are questions posed by Hsiang in \cite{hsiang1993closed} as follows.
\begin{qe} [\cite{hsiang1993closed}, Problem 3] \label{q1} How to construct embedded closed minimal hypersurfaces in $\mathbb{S}^{4}$ which are of topological types different from the four known types, namely, $S^{3}$, $S^{1} \times S^{2}$ , $T^{3}$ and $SO(3)/\mathbb{Z}_{2} \times \mathbb{Z}_{2}$?
\end{qe}
\begin{qe} [\cite{hsiang1993closed}, Problem 6] \label{q2} Can $\mathbb{S}^{n+1}$ accommodate embedded closed minimal hypersurfaces of the diffeomorphism types of arbitrary connected sums of the products of spheres?
\end{qe}
As mentioned previously, even though Kapouleas-Zou \cite{kapouleas2024minimal} constructed minimal hypersurfaces in dimension $4$ whose diffeomorphism types are large connected sums of $S^{1} \times S^{2}$, which answer the first question, there has been little progress in answering these questions, in particular, in higher dimensions. The construction in this paper answers Question \ref{q1} and provides a partial answer to Question \ref{q2} by producing minimal hypersurfaces which are connected sums of arbitrary many copies of $S^{1} \times S^{n-1}$ with min-max construction in dimension $4$ to $7$.

Although there are several works on the control of the topology of the hypersurface by Morse index, area, and ambient geometry, such as in \cite{savo2010index}, \cite{ambrozio2018comparing}, and \cite{song2023morse}, it remains obscure to obtain minimal hypersurfaces with sharply controlled topology (See recent examples of minimal hypersurfaces with a large Betti number and low Morse index in \cite{maximo2025ricci}). Moreover, the ODE method based on the cohomogeneity one theory may not be enough to provide examples of minimal hypersurfaces with complicated topology.

In dimension $3$, Simon-Smith min-max theory, for instance in Smith \cite{smith1983existence} and Colding-De Lellis \cite{colding2003min}, has provided constructions of minimal surfaces with controlled genus. The min-max theory is built on the regularity and the genus bound of the solution of isotopy minimization problem by Almgren-Simon \cite{almgren1979existence} and Meeks-Simon-Yau \cite{meeks1982embedded}. De Lellis-Pellandini \cite{de2010genus} and Ketover \cite{ketover2019genus} proved the upper bound of the genus of minimal surfaces obtained by Simon-Smith min-max construction. Notably, Wang-Zhou \cite{wang2023existence} developed the multiplicity one theorem for Simon-Smith min-max theory which gives a construction of four minimal spheres in bumpy metric on $S^{3}$ (See Sarnataro-Stryker \cite{sarnataro2023optimal} for the regularity result and Haslhofer-Ketover \cite{haslhofer2019minimal}). Moreover, Ketover developed the equivariant version of Simon-Smith min-max theory in \cite{ketover2016equivariant} and its free boundary version \cite{ketover2016free} in dimension $3$ when the isometry group is given by finite group, which is proposed by Pitts and Rubinstein in their monograph \cite{pitts1987applications} without publishing proofs.

However, due to lack of the regularity of solutions of isotopy minimization problem, Simon-Smith min-max theory in dimension $3$ cannot be directly generalized to higher dimensional settings. By the example of White in \cite{white1983existence} in dimension $4$, even with a relaxation of isotopy minimization to homotopy minimization, a certain solution of an area-minimization problem has a singularity with positive $2$-dimensional Hausdorff measure. This example shows the difficulty to obtain smooth embedded minimal hypersurfaces with controlled topology with a min-max theory based on isotopy minimization in higher dimensions.

Instead of directly generalizing Simon-Smith theory, Pitts and Rubinstein proposed the construction of the family of minimal hypersurfaces on manifolds with Lie group symmetry in the same monograph, and outlined the construction of minimal hypersurfaces with arbitrarily large Betti numbers on high-dimensional round spheres $\mathbb{S}^{n+1}$.

After we completed the first version of the manuscript, Wang–Wang–Zhou \cite{wang2026equivariant} developed an equivariant min-max theory for compact Lie group actions with three-dimensional orbit space, including regularity for equivariant isotopy area minimization and a reduced genus bound, and applied it to construct a non-equatorial minimal hypersphere in $\mathbb{S}^4$ (See also Wang \cite{wang2026embedded}).

\subsection{Main Results}
In this paper, we formulate and prove the regularity of cohomogeneity two equivariant isotopy minimization problems. Given a simply connected manifold $M^{n+1}$ for $4 \le n+1\le 7$ with a Lie group isometry $G$, where $G$ is a product of a compact connected Lie group $G_{c}$ and finite group $G_{f}$, we call $\Sigma^{n} \subset M^{n+1}$ as a $G$\textit{-equivariant minimal hypersurface} if $\Sigma$ is invariant under the group action $G$ on $M$. We restrict the dimension of orbit space $M/G$ as $\dim(M/G)=3$ for cohomogeneity $2$ setup, roughly speaking, where cohomogeneity is defined to be a dimension of $\Sigma/G$. Assuming that $M$ is simply connected and that the principal orbits by the group action $G_{c}$ are connected, it follows from topological arguments in \cite{bredon1972introduction} that the orbit space $M':=M/G_{c}$ is a (topological) $3$-manifold with boundary. We denote $\pi:M \rightarrow M/G_{c}$ to be a canonical projection map. 

We first prove the regularity of isotopy minimization problem when the isometry group is a compact connected Lie group $G_{c}$. Let $U' \subset M'$ be an admissible open set as defined in Section 3 with boundary in $M'$ and a set of isotopy $\phi :[0,1] \times M' \rightarrow M'$ whose support is in $U'$ and each $\phi(t,M')$ is a diffeomorphism in $\mathcal{I}(U')$ for $t \in [0,1]$. Here $\mathcal{I}(U')$ is an isotopy induced by a smooth $G_c$-equivariant ambient isotopy supported in $\pi^{-1}(U')$ in $M$. We denote $U= \pi^{-1}(U')$. Then we prove the regularity of the minimizer of
\[
I := \inf_{\phi \in \mathcal{I}(U')} \mathcal{H}^{n}(\pi^{-1}(\phi(1,\Sigma') \cap M'))
\]
in $M$ and we call this as \textit{a minimization problem} $(\Sigma', \mathcal{I}(U'))$. Moreover, we denote $\{(\Sigma')^{i}\}$ to be \textit{a minimizing sequence} if $(\Sigma')^{i} = \phi^{i}(1,\Sigma')$ and $\mathcal{H}^{n}(\pi^{-1}(\phi^{i}(1,\Sigma') \cap M'))$ converges to $I$, where $\phi^{i} \in \mathcal{I}(U')$.  Notice that we consider the isotopy in the quotient manifold $M'$ and prove the regularity in the original manifold $M$. More specifically, we prove the following regularity result.
\begin{thm} \label{mainregularitycontiLie} Suppose $\{(\Sigma')^{i}\}$ is a minimizing sequence for the minimization problem $(\Sigma', \mathcal{I}(U'))$ so that $(\Sigma')^{i}$ intersects $\partial M'$ transversally for each $i$, and $\Sigma^{i} : = \pi^{-1}((\Sigma')^{i})$ converges to a limit (in a varifold sense) to $V \in IV_{n}(M)$. Then the following holds:
\begin{enumerate}
    \item $V = \Gamma$ in $\pi^{-1}(U') \cap M$, where $\Gamma$ is a compact, smooth and embedded $G_{c}$-equivariant minimal hypersurface in $\pi^{-1}(U') \cap M$ whose boundary is contained in $\pi^{-1}(\partial_{rel} U' \cap M') \cap M$;
    \item $V \mres (M \setminus U) =  \pi^{-1}(\Sigma') \mres (M \setminus U)$;
    \item $\Gamma$ is $G_{c}$-stable with respect to the pullback isotopy of $\mathcal{I}(U')$ ($G_{c}$-stability).
\end{enumerate}
\end{thm}
We extend the regularity theory of isotopy minimization problem to the case of $G=G_{c} \times G_{f}$ and $G_{f} = \mathbb{Z}_{m}$. In dimension $3$, Ketover proved the regularity of minimization problem in the finite isometry group setting which was extended to corresponding min-max theory in \cite{ketover2016equivariant} and \cite{ketover2016free}, and this provides us the regularity outside of the singular orbits. We obtain the following regularity of minimizer of $\mathbb{Z}_{m}$-equivariant isotopy minimization problem as follows. We add an assumption that singular loci of $\mathbb{Z}_{m}$-isometry intersect $\partial M'$ on points $\{x_{i}' \}^{k}_{i=1} \in \partial M'$ only when $T_{x_{i}'}\partial M'$ is a smooth plane. This assumption enables us to obtain the regularity and is enough for our applications.

\begin{thm}
\label{mainregularitydiscLie} Suppose $\{(\Sigma')^{i}\}$ is a minimizing sequence for the $\mathbb{Z}_{n}$-equivariant minimization problem $(\Sigma', \mathcal{I}(U'))$ so that $(\Sigma')^{i}$ intersects $\partial M'$ transversally for each $i$, and $\Sigma^{i} : = \pi^{-1}((\Sigma')^{i})$ converges to a limit (in a varifold sense) to $V \in IV_{n}(M)$. Then the following holds:
\begin{enumerate}
    \item $V = \Gamma$, where $\Gamma$ is a compact smooth embedded $(G_{c} \times \mathbb{Z}_{n})$-equivariant minimal hypersurface in $\pi^{-1}(U') \cap M$ whose boundary is contained in $\pi^{-1}(\partial U' \cap \operatorname{Int}(M')) \cap M$;
    \item $V \mres (M \setminus U) =  \pi^{-1}(\Sigma') \mres (M \setminus U)$;
    \item $\Gamma$ is $(G_{c} \times \mathbb{Z}_{n})$-stable with respect to the pullback isotopy of $\mathcal{I}(U')$. ($(G_{c} \times \mathbb{Z}_{n})$-stability)
\end{enumerate}
\end{thm}

Based on the previous regularity results on isotopy minimization problems, we develop the cohomogeneity two equivariant min-max theory. We obtain the general min-max construction of $G$-equivariant minimal hypersurfaces as follows.

\begin{thm}[Cohomogeneity two min-max $G$-equivariant minimal hypersurfaces] \label{mainminmax}  On $M^{n+1}$ $(4 \le n+1 \le 7)$, suppose $G$ is a product of smooth Lie group $G_{c}$ whose orbits are connected and a finite group $G_{f}$, which acts on $M$ as an isometry whose principal orbits have a dimension $n-2$ where $M'=M/G_{c}$ is a $3$-manifold with boundary. Moreover, a singular axis $\tilde{\mathcal{S}}$ in $M'$ of a finite group action $G_{f}$ intersects $\partial M'$ only on smooth points and $G_{f}$ is an orientation preserving isometry on $M'$ and such that $M'/G_{f}$ is an orbifold with boundary. Let $\{ \Lambda_{t} \}_{t=0}^{1}$ be a $G_{f}$-equivariant sweepout of $M'$ whose sums of genus is bounded by $g \in \mathbb{N}$ and denote corresponding $G_{f}$-equivariant saturation by $\Pi$. Define the $G_{f}$-equivariant min-max width to be
\begin{equation*}
    W^{G_{f}}_{\Pi} = \inf_{\{ \Lambda_{t} \} \in \Pi} \sup_{t \in [0,1]} \mathcal{H}^{n}(\pi^{-1}(\Lambda_{t}))
\end{equation*}
with the nontriviality of the width:
\[
W^{G_{f}}_{\Pi} > \max \{ \mathcal{H}^{n}(\pi^{-1} \Lambda_{0}), \mathcal{H}^{n}(\pi^{-1} \Lambda_{1})\}.
\]
Then there exists a min-max sequence $\{ \Sigma_{i} \}:= \{\pi^{-1}(\Lambda_{t_{i}}^{i})\}$ converging as varifolds to $\Gamma = \sum_{j=1}^{k} n_{j} \Gamma_{j}$, satisfying the following conditions.
\begin{enumerate}
    \item  $\Gamma_{j}$ are smooth embedded pairwise disjoint $G$-equivariant minimal hypersurfaces for $1 \le j \le k$;
    \item  $W^{G_{f}}_{\Pi} = \sum_{j=1}^{k} n_{j} \mathcal{H}^{n}(\Gamma_{j})$;
    \item $\sum_{i=1}^{k} \operatorname{genus}(\pi(\Gamma_{i})) \le g$.
\end{enumerate}
\end{thm}
\begin{rmk} \label{indefinitebetti} We obtain a genus bound of the image in Theorem \ref{mainminmax}, however the min-max, even minimization procedure, may produce unexpected increase of the number of boundary component on $\partial M'$. This contrasts to the topological bound of Franz-Schulz \cite{franz2023topological} in the free boundary Simon-Smith setting. Since the metric on $\partial M'$ we are solving the isotopy minimization problem is degenerate, we cannot guarantee the bounded number of boundary components in the image in general. This may increase the first Betti number of resulting minimal hypersurfaces, and we only can produce minimal hypersurfaces with prescribed first Betti number on spheres for sufficiently large $g$.
\end{rmk}
\begin{rmk} In dimension $8$ or larger, by Simons' classification \cite{simons1968minimal} and Schoen-Simon \cite{schoen1981regularity}, we expect that the resulting minimal hypersurfaces may have codimension $7$ singularity sets. Since we focus on constructing minimal hypersurfaces with a topological control, We restrict the dimension of ambient manifolds to be $4$ to $7$. We expect that we still obtain equivariant minimal hypersurfaces with singular set whose codimension is larger than equal to $7$, and get topological control out of singular sets.
\end{rmk}
As an application of Theorem \ref{mainminmax}, we obtain minimal hypersurfaces with arbitrarily large first Betti number on round spheres $\mathbb{S}^{n+1}$ $(4 \le n+1 \le 7)$. We take $M= \mathbb{S}^{n+1}$, $G_{c} = SO(n-1)$, and $G_{f} = \mathbb{D}_{g+1}$ in Theorem \ref{mainminmax}, where $\mathbb{D}_{m}$ is a dihedral group with $2m$ elements. The following is a statement on the construction of min-max hypersurfaces with prescribed topology on round spheres for sufficiently large $g$.

\begin{thm} \label{mainminhypsphere} For dimension $4 \le n+1 \le 7$, there exists a sequence of closed embedded minimal hypersurfaces $\Sigma_{g}$ on $\mathbb{S}^{n+1}$ such that $\Sigma_{g}$ has a symmetry $SO(n-1) \times \mathbb{D}_{g+1}$. For sufficiently large $g$, $\Sigma_{g}$ has a diffeomorphism type of connected sums $\#^{2g}(S^{1} \times S^{n-1})$ or $\#^{2g+2}(S^{1} \times S^{n-1})$ for $n+1=4, 7$, and $\#^{2g}(S^{1} \times S^{n-1})$ for $n+1=5, 6$. Hence, the first Betti number is $2g$ or $2g+2$ for $n+1=4,7$, and $2g$ for $n+1=5,6$. Moreover, $\Sigma_{g}$ (after passing to a subsequence) converges to $S^{3} \cup (S^{2} \times S^{1})$ or $S^{3} \cup (S^{1} \times S^{2})$ for $n+1=4$, and in particular, $\Sigma_{g}$ converges to $S^{n} \cup (S^{n-2} \times S^{2})$ for $n+1=5, 6$ as $g \rightarrow \infty$.
\end{thm}
We obtain the following lower bound of Morse Index of min-max $G$-equivariant minimal hypersurfaces by applying the Morse index lower bound of minimal hypersurfaces on spheres by first Betti number by Savo \cite{savo2010index} (See Ambrozio-Carlotto-Sharp \cite{ambrozio2018comparing} for its generalization and recent Morse-theoretic characterizations on Hsiang hypersurfaces in dimension $4$ of Carlotto-Schulz-Wiygul \cite{carlotto2025index}). 
\begin{cor} \label{mainequivmorse} For dimension $4 \le n+1 \le 7$, there exists a sequence of minimal hypersurfaces $\Sigma_{g}$ on $\mathbb{S}^{n+1}$ such that $\Sigma_{g}$ is a $(SO(n-1) \times \mathbb{D}_{g+1})$-equivariant minimal hypersurface. Moreover, $\Sigma_{g}$ has the following Morse index lower bound.
\begin{align*}
\operatorname{Index}(\Sigma_{g}) \ge \frac{4g}{(n+1)(n+2)} + n+2. 
\end{align*}
    \end{cor}

\subsection{Outline of the proof}
We outline the proof of main results. For the regularity theorem of isotopy minimization problem with continuous Lie group symmetry (Theorem \ref{mainregularitycontiLie}), we focus on the regularity on singular orbits. On principal orbits, orbits with the maximal dimension, we can rely on the regularity theory of Almgren-Simon and Meeks-Simon-Yau. By Hsiang-Lawson's characterization on minimal surfaces on original manifold and on orbit spaces, we can consider an isotopy minimization problem on orbit space, which is a $3$-dimensional manifold with boundary with a degenerate metric. We develop the $G_{c}$-equivariant version of the topological reduction ($\gamma$-reduction) of Meeks-Simon-Yau on an orbit space to reduce the cases to minimization problem of genus zero surfaces on an orbit space. Then we develop the $G_{c}$-equivariant replacement lemma and filigree type lemma on the neighborhood of singular orbits to obtain a good convergence along the limit. Combining the topological arguments and tangent cone analysis based on the regularity on principal orbits, we obtain the regularity on singular orbits.

For the proof of $(G_{c} \times \mathbb{Z}_{n})$-equivariant isotopy minimization problem (Theorem \ref{mainregularitydiscLie}), the main part of proof is devoted to prove the regularity on the intersection of singular orbits and singular loci by finite group actions. As before, regularity near exceptional orbits follows by applying Ketover's discrete group equivariant regularity theory by Ketover \cite{ketover2016equivariant}, since the corresponding slice representations factor through finite groups. We analyze the tangent cone on points of intersection between singular orbits and singular loci and prove that the tangent cone is hyperplane. We use the characterization of stable cone by Simons \cite{simons1968minimal} and the Schoen-Simon regularity of stable minimal hypersurfaces \cite{schoen1981regularity} to prove the regularity in both minimization problems. 

Based on the regularity of an isotopy minimization problem, we build a cohomogeneity two min-max theory to construct min-max $(G_{c} \times G_{f})$-equivariant minimal hypersurfaces in Theorem \ref{mainminmax}. We set up $G_{f}$-equivariant Simon-Smith min-max by adopting Colding-De Lellis \cite{colding2003min} and Ketover \cite{ketover2016equivariant}. We sweep out the orbit space $M'= M/G_{c}$ by a one-parameter family of properly embedded $G_{f}$-equivariant surfaces. One ingredient we need to modify is the compactness of $G$-stable embedded minimal hypersurfaces. Since $G$-stability does not necessarily imply stability of minimal hypersurfaces, especially due to $\mathbb{Z}_{2}$-actions, we prove compactness theorem of $G$-stable embedded minimal hypersurfaces with uniform area bound and local genus bounds on the quotient space outside of the intersection of $\mathbb{Z}_{2}$-singular loci and singular orbits. This enables us to apply replacement argument in $G$-annuli to obtain the regularity. We obtain a genus control of images of minimal hypersurfaces in the orbit space manifold by applying \cite{de2010genus} and \cite{ketover2019genus}.

As an application, we construct minimal hypersurfaces with arbitrarily large first Betti number on round spheres $\mathbb{S}^{n+1}$ $(4 \le n+1 \le 7)$ (Theorem \ref{mainminhypsphere}). We take $M= \mathbb{S}^{n+1}$ and $G_{c} = SO(n-1)$, then $M/G_{c}$ is a $3$-dimensional ball $\mathbb{B}^{3}$ topologically. Note that the natural projection map $\pi : \mathbb{S}^{n+1} \rightarrow \mathbb{S}^{n+1}/SO(n-1)$ sends a hypersphere $S^{n}$, Clifford hypersurfaces $S^{n-1} \times S^{1}$ and $S^{n-2} \times S^{2}$ to a disk, a properly embedded cylinder and a sphere, respectively. We consider a one parameter family of $\mathbb{D}_{g+1}$-equivariant  genus $g$ surfaces which is a (topological) desingularization between a disk and a sphere. Then by applying Theorem \ref{mainminmax}, we obtain $(SO(n-1) \times \mathbb{D}_{g+1})$-equivariant min-max minimal hypersurfaces whose image over the projection map is bounded by $g$.

We apply a topological argument to prove that the genus is not collapsed along the limiting process, which is inspired by the two-piece property of minimal hypersurfaces on round spheres by Choe-Soret \cite{choe2009first}. Every equatorial hypersphere separates any other minimal hypersurfaces into two connected components. Moreover, recently Mcgrath-Zou proved free boundary minimal surfaces with genus zero which are not a flat disk are radial graphs and do not pass the origin in an Euclidean $3$-ball, based on the free boundary version of the two-piece property in \cite{mcgrath2024areas}. Knowing that an image of a minimal hypersurface is a free boundary minimal surface on a $3$-ball with a weighted metric, we can modify Mcgrath-Zou's argument to rule out genus zero case since the image of our minimal hypersurfaces contains the origin because the prescribed horizontal axes persist in the limit. The remaining possibilities of the genus $1 \le h <g$ are excluded by the dihedral symmetry $\mathbb{D}_{g+1}$.

As in Remark \ref{indefinitebetti}, the min-max procedure may increase the number of boundary components of the image of min-max minimal hypersurfaces, and this contributes to the increase of the first Betti number of minimal hypersurfaces. However, for sufficiently large $g$, we can obtain that a topological control by the symmetry of a limit minimal hypersurface $\Sigma_{\infty}$. $\Sigma_{\infty}$ has an $(SO(n-1) \times SO(2) \times \mathbb{Z}_{2})$-symmetry, and by applying Allard's regularity theorem, we deduce $2g$ or $2g+2$ as a first Betti number of minimal hypersurfaces. Moreover, we obtain $\Sigma_{\infty}$ is a union of an equatorial hypersphere $S^{n}$ and a Clifford hypersurface $\sqrt{n-2/n} \mathbb{S}^{n-2} \times \sqrt{2/n} \mathbb{S}^{2}$ for $4 \le n+1 \le 6$ by an area bound of a sweepout and Willmore type argument of Viana \cite{viana2023isoperimetry}. This gives that $\Sigma_{g}$ has $2g$ as a first Betti number when $g$ is sufficiently large for $n+1 = 5, 6$. Moreover, we obtain a lower bound of Morse index of our min-max hypersurfaces in terms of $g$ by applying Savo's work \cite{savo2010index}.
\subsection{Topology of minimal hypersurfaces and further discussion}
Continuing from the previous discussion, the limit hypersurface gives a restriction on topology of minimal hypersurfaces for sufficiently large $g$, and we obtain that $\Sigma_{g}$ has a diffeomorphism type of $\#^{2g} (S^{1} \times S^{n-1})$ for $n+1 = 5,6$. However, there are two configurations of the Clifford hypersurface relative to the fixed group action, namely $\sqrt{1/3}  \mathbb{S}^{1} \times \sqrt{2/3} \mathbb{S}^{2}$ and $\sqrt{2/3}  \mathbb{S}^{2} \times \sqrt{1/3} \mathbb{S}^{1}$, in dimension $4$. Also, for $n+1= 7$, the lack of information on $(SO(n-1) \times SO(2) \times \mathbb{Z}_{2})$-equivariant minimal hypersurfaces whose area is next to an equatorial $\mathbb{S}^{n}$ gives an indefiniteness of topology. We know that Clifford hypersurface $\sqrt{1/2} \mathbb{S}^{3} \times \sqrt{1/2} \mathbb{S}^{3}$ achieves area next to an equatorial sphere $\mathbb{S}^{6}$ among minimal hypersurfaces with an antipodal symmetry, but this does not give a clue for our equivariant case. Using an ODE argument, we obtain $\Sigma_{\infty}$ to be a union of $S^{n-2} \times S^{2}$ or $S^{n-1} \times S^{1}$ with $S^{n}$ topologically. This gives a dichotomy of the first Betti number to be $2g$ or $2g+2$ for sufficiently large $g$.

In dimension $7$, we conjecture that the Clifford hypersurface $\sqrt{2/3} \mathbb{S}^{4} \times \sqrt{1/3} \mathbb{S}^{2}$ achieves the second least area among minimal hypersurfaces with $(SO(n-1) \times SO(2) \times \mathbb{Z}_{2})$-symmetry. By intuitions from Morse theoretic idea on the relation between area and (equivariant) Morse index as in Mazet-Rosenberg \cite{mazet2017minimal} and Marques-Neves \cite{coda2016morse}, numerical simulations suggest that $\sqrt{2/3} \mathbb{S}^{4} \times \sqrt{1/3} \mathbb{S}^{2}$ and $\sqrt{5/6} \mathbb{S}^{5} \times \sqrt{1/6} \mathbb{S}^{1}$ are only such candidates, and thus $\sqrt{2/3} \mathbb{S}^{4} \times \sqrt{1/3} \mathbb{S}^{2}$ should achieve the least. However, we could not prove this mathematically, and the resulting Betti number is $2g$ for sufficiently large $g$ if the conjecture holds. 

We now discuss the comparison with the construction via doubling method, mainly with the recent construction by Kapouleas-Zou \cite{kapouleas2024minimal} in $\mathbb{S}^{4}$. Along the procedure to construct a sweepout in the quotient manifold $B^{3}$, we desingularize $D^{2} = \pi(S^{n})$ and $S^{2} = \pi(S^{n-2} \times S^{2})$ with $(g+1)$ cylindrical necks which are diffeomorphic to $S^{1} \times I$. Each neck corresponds to $S^{n-2} \times S^{1} \times {I}$ in the original manifold, and this contributes to the first Betti number by $2$. By $\mathbb{D}_{g+1}$-symmetry on the quotient manifold $B^{3}$, the increment of the boundary components also appears to be by an even number along the minimizing and min-max processes, and this gives us the first Betti number of resulting minimal hypersurfaces to be an even number.

On the other hand, the recent gluing construction of Kapouleas and Zou obtains new minimal hypersurfaces by doubling two equatorial hyperspheres $\mathbb{S}^{3}$ by adding hypercatenoidal necks $S^{2} \times I$. Each neck $S^{2} \times I$ contributes to the first Betti number by $1$ and achieves a simpler topology than $S^{1} \times S^{1} \times I$ in dimension $4$. This gluing construction is performed along $m \times m$ lattice contained in the Clifford torus $\sqrt{1/2} \mathbb{S}^{1} \times \sqrt{1/2} \mathbb{S}^{1}$, and they obtain a sequence of minimal hypersurfaces on $\mathbb{S}^{4}$ whose growth of the first Betti number is quadratic, more specifically, they obtain minimal hypersurfaces with the diffeomorphism type $\#^{m^{2}-1}(S^{2} \times S^{1})$ for large $m$ on $\mathbb{S}^{4}$.

We expect our regularity result of cohomogeneity two isotopy minimization problem and min-max construction to have further applications to construct new minimal hypersurfaces or singularity models of mean curvature flow, such as self-shrinkers and self-expanders. After the work of Buzano-Nguyen-Schulz \cite{buzano2024noncompact}, Ketover \cite{ketover2024self} constructed self-shrinkers whose asymptotic cones fatten in dimension $3$ by equivariant min-max construction with prisamtic symmetry $\mathbb{D}_{g+1} \times \mathbb{Z}_{2}$ (See the construction by Ilmanen-White \cite{ilmanen2024fattening} based on the resolution of multiplicity one conjecture by Bamler-Kleiner \cite{bamler2023multiplicity}). It would be interesting to find high dimensional analog of such self-shrinkers, and to classify self-shrinkers whose image by the projection map of symmetry $SO(n-1)$ acting on $\mathbb{R}^{n+1}$ having a low genus. Moreover, Shao and Zou constructed a sequence of self-expanders with unbounded genus in dimension $3$ by solving Meeks-Simon-Yau type equivariant isotopy minimization problem in \cite{shao2025self}. We expect there are further applications to construct new self-expanders in higher dimension and minimal hypersurfaces with Lie group symmetry on hyperbolic manifolds via isotopy minimization technique. Moreover, there will be more constructions of new minimal hypersurfaces on spheres other than minimal hypersurfaces in this paper.
\subsection{Organization} The organization of this paper is as follows. In Section $2$, we introduce the basic geometric and topological properties of orbit space of manifolds with an isometry group action. In Section $3$, we formulate and develop the regularity result for the isotopy minimization problem of an area functional of hypersurfaces over $3$-dimensional orbit manifolds. In Section $4$, we prove the regularity of isotopy minimization problem when the isometry group is a product of a connected Lie group and a finite isometry group. In Section $5$, we formulate and develop the cohomogeneity $2$ equivariant min-max theory. In Section $6$, we prove the genus bound of the quotient of $G$-equivariant min-max hypersurfaces. In Section $7$, we construct minimal hypersurfaces with arbitrarily large first Betti number on round spheres in dimension $4$ to $7$. In Appendix, we prove the compactness theorem of $G$-stable embedded minimal hypersurfaces to apply to develop cohomogeneity $2$ equivariant min-max theory.

\subsection{Acknowledgements}
The author would like to thank his advisor Daniel Ketover for suggesting this question and his helpful discussions and encouragements on this project. The author is grateful to Jacob Bernstein, Renato Bettiol, Jaigyoung Choe, Luis Atzin Franco, Niels Martin Møller, Guanhua Shao and Jonathan Zhu for inspiring discussions. The author appreciates to Mario Schulz for his interest and bringing Morse Index papers to the author's attention. The author is grateful to Tomasz Mrowka for explanations on topology. The author is grateful to Otis Chodosh, Nicos Kapouleas, Jared Marx-Kuo, and Jiahua Zou for their interests in this work. The author had partially supported by University and Louis Bevier fellowship, NSF grant DMS-1906385 and DMS-2405114.
\section{Quotient space with cohomogeneity two}
In this section, we introduce the basic notions on a quotient space of manifolds by an isometry action, and recall the construction of minimal submanifolds with a symmetry in Hsiang-Lawson \cite{hsiang1971minimal}. We consider a closed Riemannian manifold $(M^{n+1},g)$ and an isometry group $G$ which is a product of a connected Lie group $G_{c}$ and a finite isometry group $G_{f}$, acting as an isometry. We denote $M/G$ to be an \emph{orbit space} in this paper. We denote the $n$-dimensional Hausdorff measure on $M$ by $\mathcal{H}^{n}$. We also denote the $m$-th Betti number i.e. the rank of $m$-th homology group of topological manifold $X$ by $\beta_{m}(X)$.
\subsection{The orbit space and minimal submanifolds with symmetry}
We call that a submanifold $\Sigma \subset M$ is \emph{$G$-equivariant} if $\Sigma$ remains invariant along the action of $G$ i.e. $g(\Sigma)=\Sigma$ for $g \in G$. For each $x \in M$, let $G_{x} : = \{ g \in G \,|\, g \cdot x = x\}$ be an \emph{isotropy group} at $x$, and $(G_{x})$ be a conjugacy group of $G_{x}$ in $G$. Then we have a stratification of orbits by conjugacy groups and there exists a minimal conjugacy group $(H)$ over $\Sigma$. We define \emph{cohomogeneity} $k$ by $\dim \Sigma - (\dim(G)- \dim (H))$. Let us call $(H)$ by a \emph{principal orbit type} and the dimension of principal orbit is $\dim(G)- \dim (H)$. Note that the union of all orbits $M^{*} : = \{ x \in M \, | \, (G_{x}) = (H)\}$ is an open and dense submanifold of $M$. We call the orbit $(H')$ to be an \emph{exceptional orbit} if $(H)\neq (H')$ but $\dim H= \dim H'$. We call an orbit by a \emph{singular orbit} if the orbit is neither a principal orbit nor an exceptional orbit, and denote the set of points in a singular orbit by $M_{s} := \{ x \in M \, | \, x \text{ is in a singular orbit} \}$. Note that singular orbits have strictly lower dimension than principal orbit and exceptional orbits. 
\begin{rmk}
In case of absence of singular orbits in the sense of our definitions, Ketover \cite{ketover2016equivariant} developed the equivariant min-max theory to construct minimal surfaces on $3$-dimensional manifolds in Simon-Smith min-max setting with discrete isometry group.
\end{rmk}
Now we suppose the dimension of principal orbits to be $n-2$ in this paper and see the structure of the orbit space $M/G_{c}$. Hence, the cohomogeneity is $2$ in our cases. $M/G_{c}$ may not be a topological manifold (with boundary) and may have singular points (as a topological manifold) by \cite[Chapter IV]{bredon1972introduction}, while $M/G_{c}$ has a triangulation which also
 provides an equivariant triangulation of $M$ and we can stratify singular orbits (see Illman \cite{illman1978smooth} and \cite{illman1983equivariant} for more details). By Theorem 4.6 and Corollary 4.7 of \cite[Chapter IV]{bredon1972introduction}, if $M$ is simply connected and all orbits are connected, then $M/G_{c}$ is a topological $3$-manifold (with boundary). 
 
\noindent \textbf{Topological assumption.} We restrict $M$ to be simply connected and principal orbits by an isometry action by a continuous Lie group $G_{c}$ are connected. 
\begin{rmk} \label{scassumption}
In \cite[Chapter IV]{bredon1972introduction}, it is proven that the orbit space $M/G_{c}$ is a manifold with boundary outside a set of codimension $4$ if $H_{1}(M,\mathbb{Z}_{2})=0$ and a principle orbit is connected so that all orbits are connected. Hence from the assumption above, $M/G_{c}$ is a topological $3$-manifold (with boundary), which may have corners. In this paper, since this condition is sufficient for our applications, we only consider manifolds satisfying such conditions.
\end{rmk}
We formulate the metric structure on the orbit space $M/G_{c}$ to deal with minimal submanifolds. For convenience, let us put $M' = M/G_{c}$ and $\Sigma' = \Sigma/G_{c}$ for a $G_{c}$-equivariant hypersurface $\Sigma \subset M$. Let us first consider the induced metric $\overline{g}$ on $M^{*}/G_{c}$ from a given metric $g$ on $M^{*}$ to be a pushforward metric of a projection map $\pi : M \rightarrow M'$. Let us define it in more detail. Fix $x' \in M^{*}/G_{c}$ and take one point $x= \pi^{-1}({x'})$ in the preimage $M^{*}$. For $X',Y' \in T_{x'}(M^{*}/G_{c})$, there exist $X, Y \in T_{x}M^{*}$ satisfying $\pi_{\sharp}X= X'$ and $\pi_{\sharp}Y= Y'$. We define the metric $\overline{g}$ at $x'$ by the metric satisfying
\begin{equation} \label{inducedmetric}
    \overline{g}(X',Y') = g(X,Y),
\end{equation}
which also gives the distance between the orbits on $M^{*}/G_{c}$. Since $M^{*}$ is an open and dense submanifold in $M$, we can define a continuous metric structure on $M'$ induced by one in $M^{*}/G_{c}$.

To discuss the relation between minimal submanifold on $M$ and $M'$, we introduce the \emph{volume function} $\mathcal{V}$ on $M'$. Consider the following volume function $\mathcal{V}: M' \rightarrow \mathbb{R}$:
\begin{equation*}
    \mathcal{V} = \begin{cases} Vol(\pi^{-1}(x')) \text{ if } x' \in M^{*}/G_{c},
    \\ m \cdot Vol(\pi^{-1}(x')) \text{ if } x' \text{ is in an exceptional orbit}, \\ 0 \text{ otherwise}
    \end{cases}
\end{equation*}
for $x' \in M'$ where $m = |H'/H|$ when $H'$ is an isotropy group of an exceptional orbit of $x$ and $H$ is that of a principal orbit. $\mathcal{V}$ is a continuous function on $M'$ and a differentiable on $M^{*}/G_{c}$. We define the metric $g'$ on $M'$ by
\begin{equation} \label{volumefunction}
    g' = \mathcal{V}^{\frac{2}{k}} \overline{g} = \mathcal{V} \overline{g}.
\end{equation}
Note that $g'$ is a Riemannian metric over $M^{*}/G_{c}$ and continuous metric degenerates to zero on the boundary $\partial M'$. The following theorem states the relation between minimal surfaces on $M'$ and $G_{c}$-equivariant minimal hypersurfaces on $M$. Also, we call $\Sigma' \subset M'$ to be \emph{properly embedded} if $\partial \Sigma' \subset \partial M'$ and $\Sigma'$ is embedded in $M'$.  We denote the canonical deformation induced by (\ref{volumefunction}) by $\mathcal{D}: (M',\overline{g}) \rightarrow (M',g')$. Note that $\mathcal{D}$ is a conformal deformation.
\begin{thm} [\cite{hsiang1971minimal}, Theorem 2] \label{deformationminimal}
$\Sigma$ is a $G_{c}$-invariant  minimal submanifold on $(M,g)$ if and only if the restriction
of $\Sigma'$ to the principal stratum is a minimal submanifold in $(M',g')$.
\end{thm}
\begin{rmk}
At a smooth boundary point $x'\in\partial M'$, the quotient of the local slice representation is a half-space. In this slice coordinates, a $G_{c}$-invariant hypersurface is smooth across the singular orbit if and only if its image in $M'$ intersects $\partial M'$ orthogonally with respect to $\bar g$. After restricting to the normal slice, this is precisely the usual smooth condition for a rotationally invariant hypersurface across the collapsed limit. The corresponding cohomogeneity $1$ case is discussed in \cite[Chapter III.2]{hsiang1971minimal}, and an explicit formulation and proof in the standard rotational setting is given in \cite[Remark 3.3 and Theorem 3.6]{bettiol2021nonplanar}. Thus the same local normal slice criterion applies at the smooth boundary points considered below.
\end{rmk}
\begin{rmk}
    On the other hand, at the nonsmooth point of the boundary, the intersection angle does not need to be a right angle. For instance, in Section $5$ of Lawson \cite{lawson1972equivariant}, he described examples of minimal submanifolds which are preimages of a geodesic passing the origin with an angle $\pi/2p$ in $2$-dimensional wedge. While his examples are on cohomogeneity $1$ setting in our notation, Taking the product of this local model with a trivial one-dimensional factor gives the corresponding local example in a three-dimensional wedge orbit space.. See the example (5.2) of \cite{lawson1972equivariant} for further details.
\end{rmk}
Also we enumerate basic properties on the relation between stationarity and stability and those on equivariant settings. We call $V$ to be a $G_{c}$-varifold if $g_{\sharp}V=V$ for all $g \in G_{c}$ and $X$ to be a $G_{c}$-equivariant vector field if $g_{\sharp}X=X$. For a $G_{c}$-varifold $V$, we say the varifold $V$ is $G_{c}$-stationary if
\begin{equation*}
    \partial V(X) =0
\end{equation*}
for any $G_{c}$-equivariant vector field $X$. Moreover, we define $V$ to be a $G_{c}$-stable varifold if the area functional is stable under an $G_{c}$-equivariant normal variation supported on the regular part of the varifold. The equivalence between $G_{c}$-stationarity and stationarity is proven by Liu \cite{liu2021existence}.

\begin{lem}[\cite{liu2021existence}, Lemma 2.2] \label{lem:stationarity}
A $G_{c}$-varifold $V$ is $G_{c}$-stationary in $(M^{n+1},g)$ if and only if it is stationary in $(M^{n+1},g)$.
\end{lem}
Although the equivalence of $G_{c}$-stability and stability is proven for smooth embedded $2$-sided hypersurfaces by Wang \cite{wang2022min}, since stability operator is supported only on the regular part and we will only consider $2$-sided hypersurfaces, we can apply the lemma without modification later (see also Definition 6.1 in \cite{wang2024equivariant} on $G$-equivariant Morse index of varifolds).

\begin{lem}[\cite{wang2022min}, Lemma 2.8] \label{equivstablity} Let $V$ be a compact $G_{c}$-equivariant stationary $n$-dimensional varifold whose support is $2$-sided in $M$. Then $V$ is $G_{c}$-stable if and only if it is stable.
\end{lem}
We collect notions about the local slice structure. For $x \in M$, we denote $G_{c}(x)$ and $(G_{c})_{x}$ to be the orbit of $x$ in $M$ and the isotropy group of $x$, respectively. Also, let $B_{G_{c}}^{\rho}(x)$ be a tubular neighborhood of $G_{c} (x)$. We call the set $S \subset M$ to be a slice at $x$ if $(G_{c})_{x}(S)=S$ and $G_{c} \times_{(G_{c})_{x}} S$ can be mapped diffeomorphically onto a tubular neighborhood $B^{\rho}_{G_{c}}(x)$ by the exponential map. We will use the following slice theorem later which is a fundamental fact on homogeneous Riemannian manifolds. 
\begin{thm} [Slice Theorem, \cite{hsiang1971minimal}, Chapter I Section 2] \label{slicethm} The normal bundle $\nu(G_{c}(x))$ of $G_{c}(x)$ is associated with the canonical $(G_{c})_{x}$-bundle $(G_{c})_{x} \rightarrow G_{c} \rightarrow G_{c}(x)$ with the action of $(G_{c})_{x}$ naturally acting on the space of normal vectors to $G_{c}(x)$ at $x$. Moreover, the exponential map sends a small disk bundle of $\nu(G_{c}(x))$ equivariantly onto a small tubular neighborhood of $G_{c}(x)$. 
\end{thm}
\subsection{Discrete equivariant setup on $3$-dimensional orbit manifolds} Let $G_{f}$ be an orientation-preserving discrete isometry acting on $M'$. We enumerate basic ingredients to develop the equivariant min-max construction to produce minimal hypersurfaces with controlled topology on $(M,g)$ with an isometry group $G = G_{c} \times G_{f}$, in particular when $G_{f}$ acts on a $3$-manifold with boundary $M'$. Let us define the quotient map $\pi_{f}: M' \rightarrow M'/G_{f}$ and take $\pi_{f}(\tilde{\mathcal{S}}) = \mathcal{S}$.

For $x' \in M'$, let $(G_{f})_{x'}$ denote its isotropy group. We define the \emph{singular locus} to be a set of points whose isotropy group is nontrivial:
\begin{equation*}
    \tilde{\mathcal{S}} = \{ x' \in M' \, | \, (G_{f})_{x'} \neq \{ e \} \}
\end{equation*}
and denote its image in the quotient by
\[
\mathcal S:=\pi_f(\widetilde{\mathcal S})\subset M'/G_f.
\]
At a smooth point $x'\in M'$, the isotropy representation identifies $(G_f)_{x'}$ with a finite subgroup of $SO(3)$. Thus the local structure of $\widetilde{\mathcal S}$ is determined by the standard classification of finite subgroups of $SO(3)$.

We call a connected component of the one-dimensional stratum of $\widetilde{\mathcal S}$ a singular axis. If the isotropy group along its interior is isomorphic to $\mathbb Z_m$, we call it a $\mathbb Z_m$-singular axis. Points at which the isotropy group is not cyclic will be called singular vertices.

Away from the corners of $M'$, the image $\mathcal S$ has the following standard local structure. At an interior point of a $\mathbb Z_m$-singular axis, $\mathcal S$ is locally a single geodesic segment. At the image of a singular vertex, three such segments meet. After passing to a slice, the action is locally modeled on a finite subgroup of $SO(3)$; for a non-cyclic finite subgroup, the quotient of the unit sphere has three singular points, and taking the cone gives the corresponding trivalent local model. Thus $\mathcal S\cap\operatorname{Int}(M'/G_f)$ is locally a trivalent graph, with edges labeled by their cyclic isotropy groups.

Throughout the paper we impose the following assumption on the interaction between the finite-group singular locus and the boundary:
\[
\widetilde{\mathcal S}\cap\partial M' \quad\text{is contained in the smooth part of }\partial M'. \tag{$\star$} \label{singularlocusstar}
\]
Adding this condition is enough to construct minimal hypersurfaces in our examples. Notice that $\mathbb{D}_{n}$-type singular locus point can only be located in $\operatorname{Int}(M')$. In particular, no singular axis considered below meets a corner of $M'$. Singular axes may terminate on the smooth part of $\partial M'$, so the local trivalent graph structure only occurs at interior vertices. Our construction contains the case where $G_{f}$ being a trivial group, i.e. $G_{f}= \{ e \}$.

By Lemma 3.3 in \cite{ketover2016equivariant}, we can restrict to the cases that our $G_{f}$-equivariant surfaces intersect singular loci with $\mathbb{Z}_{n}$ or $\mathbb{D}_{n}$-symmetry. Lemma 3.4 and Lemma 3.5 in \cite{ketover2016equivariant} give the local structure of $\mathbb{Z}_{n}$ and $\mathbb{D}_{n}$-equivariant surface in $3$-dimensional manifolds, and we will apply these later.
\section{Isotopy minimization problem associated with quotient space by continuous Lie group} We formulate and prove the problem of isotopy minimization of the area of $n$-dimensional equivariant hypersurfaces over $3$-dimensional quotient spaces by a connected Lie group. By the metric deformation (\ref{volumefunction}) and Theorem \ref{deformationminimal}, we consider the quotient space whose boundary is degenerate and does not achieve mean convexity, where we solve an isotopy minimization problem on $3$-manifold with boundary. Hence we cannot directly apply Meeks-Simon-Yau regularity theory \cite{meeks1982embedded} and its boundary regularity analog in De Lellis-Pellandini \cite{de2010genus} and Li \cite{li2015general} (See also Jost \cite{jost1986existence}). We develop the regularity theory of the minimizer over $G_{c}$-equivariant smooth isotopy, so that the corresponding (continuous) isotopy on the quotient space $M'$. We develop arguments based on Meeks-Simon-Yau \cite{meeks1982embedded} (See Almgren-Simon \cite{almgren1979existence} for regularity of minimizers of Plateau problem, Grüter-Jost \cite{gruter1986embedded} for nondegenerate boundary case).

Let $U' \subset M'$ be an open set with boundary in $(M',\overline{g})$ satisfying
\begin{enumerate}
    \item $U'$ is homeomorphic to a unit $3$-ball in $\mathbb{R}^{3}$;
  \item if $U:=\pi^{-1}(U')$, then $\pi^{-1}\bigl(\partial_{\mathrm{rel}}U'\bigr)$ is mean convex in $M$, where $\partial_{\mathrm{rel}}U':=\overline{\partial U'\cap\operatorname{Int}(M')}$;
    \item $\partial_{\mathrm{rel}}U'$ meets $\partial M'$ transversely and $U'\cap\partial M'$ is homeomorphic to the unit $2$-disk.
\end{enumerate} 
We call an open set $U'$ satisfying (1)-(3) by \emph{an admissible open set} in $M'$.

Let $\mathcal I_{G_c}(U)$ denote the collection of smooth $G_c$-equivariant isotopies $\Phi:[0,1]\times M\rightarrow M$ such that $\Phi_0=\operatorname{id}$ and $\Phi_t=\operatorname{id}$ on $M\setminus U$ for every $t\in[0,1]$.  Every $\Phi\in\mathcal I_{G_c}(U)$ preserves $G_c$-orbits and hence induces an isotopy $\phi:[0,1]\times M'\longrightarrow M'$ satisfying
$\pi\circ\Phi_t=\phi_t\circ\pi$.

We denote by $\mathcal I(U')$ the collection of quotient isotopies obtained in this way. Thus, throughout this section, an isotopy in $\mathcal I(U')$ is understood to be induced by a smooth $G_c$-equivariant isotopy on $M$. 

Throughout this section, all isotopies, local surgeries, replacements, and topological modifications performed on $M'$ are chosen in local slice coordinates so that their preimages under $\pi$ are smooth embedded $G_c$-equivariant hypersurfaces in $M$ (after small deformation if necessary). Whenever an ambient isotopy of $M'$ is used, we take a convention that an isotopy belongs to $\mathcal I(U')$.

For an admissible open set $U'$, We now define the following minimization problem $(\Sigma', \mathcal{I}(U'))$: 
\[
    \alpha(\Sigma',U')
    :=
    \inf_{\phi\in\mathcal I(U')}
    \mathcal H^n
    \left(
        \pi^{-1}\bigl(\phi_1(\Sigma')\bigr)
    \right).
\]
A sequence of surfaces $\{(\Sigma')^i\}_{i\in\mathbb N}$ is called a \emph{minimizing sequence} for $(\Sigma',\mathcal I(U'))$ if
\[
    (\Sigma')^i=\phi^i_1(\Sigma') \qquad\text{for some }\phi^i\in\mathcal I(U')
\]
and
\[
    \lim_{i\to\infty} \mathcal H^n \left( \pi^{-1}\bigl((\Sigma')^i\bigr) \right) = \alpha(\Sigma',U').
\]
We write
\[
    \Sigma_i:=\pi^{-1}\bigl((\Sigma')^i\bigr).
\]
Notice that, on a relatively compact open subset of $M^{*} = M^{pr}/G_{c}$, every compactly supported isotopy lifts to a $G_c$-equivariant isotopy on $M$. Exceptional orbits are considered in Remark \ref{rmk:exceptional} Consequently, on the interior of $M'$ the minimization problem above is locally the usual three-dimensional isotopy minimization problem.  

We devote this section to prove the following regularity result of area minimizer of the minimization problem.
\begin{thm} \label{regularitycontiLie} Let $U'\subset M'$ be an admissible open set, and let $\{(\Sigma')^i\}$ be a minimizing sequence for a minimization problem $(\Sigma',\mathcal I(U'))$ so that $(\Sigma')^{i}$ intersects $\partial M'$ transversally for each $i$. Suppose that $\Sigma_i:=\pi^{-1}((\Sigma')^i)$ converges as varifolds to $V\in IV_n(M)$.  Then there exists a compact smooth embedded $G_c$-equivariant minimal hypersurface $\Gamma$ (possibly with multiplicity) in $U=\pi^{-1}(U')$ such that
\begin{enumerate}
    \item $V\mres U = |\Gamma|\mres U$;
    \item $V\mres(M\setminus U) =|\Sigma|\mres(M\setminus U)$;
    \item $\Gamma$ is $G_c$-stable with respect to compactly supported
    $G_c$-equivariant variations in $U$.
\end{enumerate}
Moreover,
\[
    \partial\Gamma \subset \pi^{-1}(\partial_{\mathrm{rel}}U').
\]
\end{thm}
\subsection{Interior regularity à la Meeks-Simon-Yau and Ketover} We first consider the regularity of the interior of $M'$. By the slice theorem, the quotient map $\pi$ has a local bundle structure. Hence, for a small open set $U'$, a smooth isotopy compactly supported in $U'$ lifts to a smooth $G_{c}$-equivariant isotopy compactly supported in $U = \pi^{-1}(U')$. Consequently, a minimizing sequence for $(\Sigma', \mathcal{I}(U'))$ is locally an isotopy minimizing sequence on $\operatorname{Int}(M')$. Considering a finite symmetry $G_f$, we choose the local lift $G_f$-equivariantly. Thus, near a point with nontrivial finite isotropy, the projected minimizing sequence is a $(G_f)_{x'}$-equivariant isotopy minimizing sequence in M'.

At a point $x'\in\operatorname{Int}(M')$ with trivial $G_f$-isotropy, we can apply the well-known regularity result of Meeks-Simon-Yau \cite{meeks1982embedded}. If $(G_f)_{x'}$ is nontrivial, we instead apply discrete equivariant orientation-preserving isotopy version of Ketover \cite{ketover2016equivariant} to the induced finite group action on a sufficiently small $(G_f)_{x'}$-invariant ball centered at $x'$.

\begin{thm}[Meeks-Simon-Yau \cite{meeks1982embedded}, Ketover \cite{ketover2016equivariant}] \label{intregular} 
Let $U'\Subset\operatorname{Int}(M')$ be a sufficiently small ball, and suppose that $\{(\Sigma')^i\}$ is a minimizing sequence for the isotopy minimization problem in $U'$. If a finite isotropy group $H\leq G_f$ acts on $U'$, assume that the sequence and the minimizing problem are $H$-equivariant. If
\[
    \Sigma^i=\pi^{-1}((\Sigma')^i)\rightarrow V
\]
as varifolds, then $V$ is induced in $\pi^{-1}(U')$ by a smooth embedded $(G_c\times H)$-equivariant minimal hypersurface.
\end{thm}
\begin{rmk} \label{rmk:exceptional}
    We briefly explain how to deal with exceptional $G_c$-orbits. Let $x\in M$ lie on an exceptional orbit. Since an exceptional orbit has the same dimension as a principal orbit, the slice representation factors through a finite group. By the slice theorem, a neighborhood of $x'=\pi(x)$ is therefore locally modeled by a finite-group action on a three-dimensional slice. In particular, the slice action is orientation preserving. Otherwise the quotient space obtains a boundary or fails to be a $3$-manifold at $x'$. Thus, $F_{x} \subset SO(3)$. Hence the isotopy minimization problem near $x'$ reduces to a finite-group equivariant isotopy minimization problem in dimension three, and the previous theorem applies.
\end{rmk}
\begin{rmk}
Theorem \ref{intregular} reduces us to prove the regularity on singular orbits. The boundary regularity of the nondegenerate boundary case is discussed in \cite{gruter1986embedded}, \cite{jost1986existence}, \cite{de2010genus} and \cite{li2015general} before. However, we cannot apply their argument directly due to the degeneracy at boundary points in the weighted metric. We will rely on the interior regularity to prove the regularity on singular orbits.
\end{rmk}
\subsection{$\gamma$-reductions on quotient space: topological deformations on orbit space} The aim of remaining parts of this section is to prove regularity on singular orbits. We reduce our local minimization problem to the case of genus zero surfaces on $M'$ by the adaptation of $\gamma$-reduction in Meeks-Simon-Yau \cite{meeks1982embedded} (De Lellis-Pellandini \cite{de2010genus} and Li \cite{li2015general} developed $\gamma$-reduction arguments in the mean convex boundary setting). Here we develop the quotient space version of $\gamma$-reduction.

We have the following version of the thin tube isoperimetric inequality of Lemma 1 in \cite{meeks1982embedded} and Lemma 4.2 of \cite{jost1986existence} for the boundary version (See also \cite{li2015general}) as a singular orbit version.

\begin{lem}  [Tube isoperimetric inequality] \label{tubeisopineq} There exists $r_{1}(M,G_{c}), \delta_{1}(M,G_{c}) \in (0,1)$ with the property that if $\Sigma$ is a $G_{c}$-equivariant closed hypersurface in $M$ and $\partial \Sigma' \subseteq \partial M'$ for $\Sigma'=\pi(\Sigma)$ and
\begin{equation*}
    \mathcal{H}^{n}(\Sigma \cap B^{r_{1}}_{G_{c}}(x)) < \delta_{1}^{n- \dim (G_{c} \cdot x)} r_{1}^{n- \dim (G_{c} \cdot x)} \text{ for each } x \in M
\end{equation*}
then there exists a unique compact set $K\subset M$ such that
\begin{enumerate}
    \item $\partial K = \Sigma$ i.e. $K$ is bounded by $\Sigma$; 
    \item $\mathcal{H}^{n+1}(K \cap B^{r_{1}}_{G_{c}}(x)) \le \delta_{1}^{n- \dim (G_{c} \cdot x)} r_{1}^{n+1- \dim (G_{c} \cdot x)}$ for each $x \in M$;
    \item $\mathcal{H}^{n+1} (K) \le c_{1} (\mathcal{H}^{n} (\Sigma))^{\frac{n+1}{n}}$, where $c_{1}$ depends only on $(M,g)$.
\end{enumerate}
\begin{proof}
We apply the slice lemma together with the lower semi-continuity of orbit type property. Hence, we can take a radius $r_{1}$ such that $(G_{c})_{x}$ achieves a maximal isotropy group in $B^{r_{1}}_{G_{c}}(x)$. We can also consider local triangulation on $M'$ which satisfies `homogeneous regularity' property in a slice sense as in Meeks-Simon-Yau \cite{meeks1982embedded}. Applying the isoperimetric inequality arguments in \cite{meeks1982embedded} on each cell and summing those up, we obtain (1) and (2). Since $\Sigma'$ splits $M'$ into two separated regions, we have a uniqueness of $K$.

Since the quotient space $M'$ is compact, we can consider a global triangulation on $M'$ and apply the arguments in the proof of Lemma 1 in \cite{meeks1982embedded} in the same way. We obtain (3) by taking the minimum exponent among possible isoperimetric inequalities in each tubular neighborhood, which is a preimage of a cell, which enables us to take a uniform radius over $M'$.
\end{proof}
\end{lem}
\begin{rmk}
We apply (1) and (2) of Lemma \ref{tubeisopineq} in the sense that a preimage $\Sigma$ of a projected embedded local surface $\Sigma'$ with small area bounds a unique compact set $K$ which has a small uniform volume, while (3) comes from a standard isoperimetric inequality and global patching argument of Meeks-Simon-Yau in $M$. We proceed the proof of $\gamma$-reduction without the scaling as $r_{1}=1$ in \cite{li2015general} and \cite{meeks1982embedded}.
\end{rmk}
From now on, we fix $\delta_{1}>0$ and $r_{1}>0$ in Lemma \ref{tubeisopineq}. Suppose $0<\gamma<\delta_{1}^{n} r_{1}^{n} /9$. We adapt the definition and the statements from \cite{meeks1982embedded} and \cite{jost1986existence}. Let $\Sigma'_{1}$ and $\Sigma'_{2}$ be embedded surfaces with boundary in $M'$ whose boundary is on $\partial M'$. Also denote $U'$ to be an open subset of $M'$ whose slice of preimage is strictly convex. We define $\Sigma_{2}'$ to be a $(\gamma,U')$\emph{-reduction} of $\Sigma_{1}'$ if $\Sigma_{1}'$ and $\Sigma_{2}'$ satisfy the following:
\begin{enumerate}
    \item $\Sigma_{2}'$ is obtained by the surgery from $\Sigma_{1}'$ in $U'$ as one of the following cases: 
    
    \begin{enumerate}
        \item $\Sigma_{2}'$ is obtained by cutting away a neck from $\Sigma_{1}'$. That is $\Sigma_{1}' \setminus \Sigma_{2}'$ is homeomorphic to $\mathbb{S}^{1} \times (0,1)$ and $\Sigma_{2}' \setminus \Sigma_{1}'$ is homeomorphic to the disjoint union of two open disks. 
        \item  $\Sigma_{2}'$ is obtained by cutting away a half-neck from $\Sigma_{1}'$. That is $\Sigma_{1}' \setminus \Sigma_{2}'$ is homeomorphic to $[0,1] \times (0,1)$ and $\Sigma_{2}' \setminus \Sigma_{1}'$ is homeomorphic to the disjoint union of two open half-disks. 
    \end{enumerate}
    \item There exists a compact set $K'$ embedded in $U'$ which is homeomorphic to the unit closed $3$-ball with $\partial K' = \overline{\Sigma_{1}' \triangle \Sigma_{2}'} \cup Y'$ for some compact set $Y' \subset \partial M'$ ($\partial K = \overline{\Sigma_{1}' \triangle \Sigma_{2}'}$ modulo $\partial M'$).
    \item $\mathcal{H}^{n}(\pi^{-1}(\Sigma_{1}' \triangle \Sigma_{2}')) < 2 \gamma$. Moreover, $\operatorname{int}K'\cap(\Sigma'_1\cup\Sigma'_2)=\varnothing$
    \item If $\Gamma'$ is a connected component of $\Sigma_{1}' \cap \overline{U'}$ containing $\Sigma_{1}' \setminus \Sigma_{2}'$, and $\Gamma' \setminus (\Sigma_{1}' \setminus \Sigma_{2}')$ is disconnected, then for each component of $\Gamma' \setminus  (\Sigma_{1}' \setminus \Sigma_{2}')$ satisfies either one of the following:
    \begin{enumerate}
        \item It is a genus zero surface contained in $U' \cap M'$ with area of preimage at least $\delta_{1}^{n}r_{1}^{n}/2$;
        \item It has a positive genus.
    \end{enumerate}
\end{enumerate}
We call $\Sigma'$ to be $(\gamma,U')$\emph{-irreducible} if there is no surface which can be obtained by a $(\gamma,U')$-reduction from $\Sigma'$. 

We also define a strong $(\gamma,U')$-reduction. $\Sigma_{1}'$ and $\Sigma_{2}'$ are properly embedded surfaces in $M'$. We call $\Sigma_{2}'$ to be a \emph{strong} $(\gamma,U')$\emph{-reduction} of $\Sigma_{1}'$ if there is an isotopy $\{ \psi_{s} \}_{s \in [0,1]} \in \mathcal{I}(U')$ such that
\begin{enumerate}
    \item $\Sigma_{2}'$ is a $(\gamma,U')$-reduction of $\psi_{1}(\Sigma_{1}')$; 
    \item $\Sigma_{2}' \cap (M' \setminus U') = \Sigma_{1}' \cap (M' \setminus U')$;
    \item $\mathcal{H}^{n}(\pi^{-1}(\psi_{1}(\Sigma_{1}') \triangle \Sigma_{1}') )<\gamma$.
\end{enumerate}
We also denote $\Sigma'$ to be \emph{strongly} $(\gamma,U')$\emph{-irreducible} if there does not exist a surface which is obtained by a strong $(\gamma,U')$-reduction from $\Sigma'$. By following the arguments in Remark 3.1 in \cite{meeks1982embedded}, we have the maximum number of the strong $(\gamma,U')$-reduction is finite:

\begin{prop} \label{gammareduction}
    Given any (possibly disconnected) properly embedded surface $\Sigma'$ on $M'$, and put $\Sigma_{0}'= \Sigma'$, then there exists a finite sequence of properly embedded surfaces $\{ \Sigma_{i}' \}_{i=0}^{k}$ such that
    \begin{enumerate}
        \item $\Sigma_{i+1}'$ is a strong $(\gamma,U')$-reduction of $\Sigma_{i}'$ for $i=0,..., k-1$;
        \item $\Sigma_{k}'$ is strongly $(\gamma,U')$-irreducible.
    \end{enumerate}
    Moreover, $\mathcal{H}^{n}(\pi^{-1}(\Sigma' \triangle \Sigma_{k}')) \le 3c \gamma.$
    \begin{proof}
    The proof follows from the arguments of \cite[Remark 3.1]{meeks1982embedded}.
    \end{proof}
\end{prop}

For any properly embedded surface $\Sigma'$, let us denote
\begin{align*}
    E(\Sigma') = \mathcal{H}^{n}(\pi^{-1}(\Sigma'))- \inf_{\tilde{\Sigma} \in \mathcal{J}_{U'}(\Sigma')} \mathcal{H}^{n}(\pi^{-1}(\tilde{\Sigma})),
\end{align*}
where $\mathcal{J}_{U'}(\Sigma') = \{ \psi_{1}(\Sigma') | \{ \psi_{s}\}_{s \in [0,1]} \in \mathcal{I}(U') \}$ denotes all surfaces isotopic to $\Sigma'$. Here we define $\Sigma_{0}'$ to be the union of all connected components $\Lambda' \subset \Sigma' \cap \overline{U'}$ such that there exists $K_{\Lambda'} \subset U'$ diffeomorphic to the unit $3$-ball satisfying $\Lambda' \subset K_{\Lambda'}$ and $\partial K_{\Lambda'} \cap \Sigma' \cap M' = \emptyset$ in $M'$.

We now show the genus zero property of surfaces for strongly $(\gamma,U')$-irreducible surfaces $\Sigma'$. It directly follows from the boundary version of arguments of Meeks-Simon-Yau.

\begin{thm} \label{genuszero} Let $U' \subset M'$ be an admissible open set, and $A' \subset U'$ be a compact subset diffeomorphic to the unit $3$-ball. Suppose $\partial M'$ intersects both $\partial U'$ and $\partial A'$ transversally. Suppose $\Sigma' \subset M'$ is a smooth, closed, and embedded surface such that
\begin{enumerate}
    \item $\Sigma'$ intersects both $\partial M'$ and $\partial A'$ transversally;
    \item $E(\Sigma') \le \gamma/4$ and $\Sigma'$ is strongly $(\gamma,U')$-irreducible;
    \item For each component $\Gamma'$ of $\Sigma' \cap \partial_{rel} A' \cap M'$, let $F_{\Gamma'}'$ be a component in $(\partial_{rel} A' \cap M') \setminus \Gamma'$ and $\mathcal{H}^{n}(\pi^{-1}(F_{\Gamma'}')) = \min \{\mathcal{H}^{n}(\pi^{-1}(F_{\Gamma'}')),\mathcal{H}^{n}(\pi^{-1}(\partial_{rel} A' \setminus F_{\Gamma'}'))\}$. We also suppose that $\sum_{j=1}^{q} \mathcal{H}^{n}(\pi^{-1}(F_{j}')) \le \gamma/8$, where $F_{j}' = F'_{\Gamma_{j}'}$ and $\Gamma_{1}', ..., \Gamma_{q}'$ is denoted to be the components of $\Sigma' \cap \partial_{rel} A' \cap M'$. $\Gamma'$ is either a closed Jordan curve in $M'$ or a Jordan arc with endpoints on $\partial M'$ and each $F'_{\Gamma_{j}'}$ is either a disk, a half-disk or an annulus. 
\end{enumerate}
Then $\mathcal{H}^{n}(\pi^{-1}(\Sigma_{0}')) \le E(\Sigma')$ and there exists pairwise disjoint connected closed genus zero surfaces $D_{1}', ..., D_{p}'$ with $D_{i}' \subset (\Sigma' \setminus \Sigma_{0}') \cap U'$, $\partial D_{i}' \setminus \partial M' \subset \partial A'$, and $(\cup_{i=1}^{p} D_{i}' ) \cap A' = (\Sigma' \setminus \Sigma_{0}') \cap A' $. Moreover,
\begin{equation*}
    \sum^{p}_{i=1} \mathcal{H}^{n}(\pi^{-1}(D_{i}')) \le \sum^{q}_{j=1} \mathcal{H}^{n}(\pi^{-1}(F_{j}')) + E(\Sigma').
\end{equation*}
\begin{proof}
    We apply Lemma \ref{tubeisopineq} in place of Lemma 1 in \cite{meeks1982embedded}. Other parts remain the same as the proof of Theorem 2 in \cite{meeks1982embedded} except we consider the boundary version. We note that the local replacement arguments can be approximated arbitrarily closely in area by isotopies in $\mathcal I(U')$.
\end{proof}
\end{thm}
\subsection{Area analysis} In this section, we introduce a $G_{c}$-equivariant version of the area comparison theorem, the replacement lemma and filigree type lemma. We now develop the area comparison argument which is a Lie group $G_{c}$-equivariant analog of Lemma 3 of \cite{meeks1982embedded} whose images on $M'$ are surfaces with boundary. First area comparison lemma follows from the local mean convexity of the tubular neighborhood of an orbit.

\begin{lem} [Area Comparison] \label{areacomparison}  Suppose $U \subset M$ is a $G_{c}$-equivariant open mean convex set whose image $U'$ is diffeomorphic to $3$-ball satisfying the following conditions:
\begin{enumerate}
    \item $\{ U_{t} \}_{t \in [0,\theta r_{1}] }$ is a $1$-parameter family of $G_{c}$-equivariant open sets which is defined by $U_{t} =\{ x \, | \, d( x,U)<t \text{ and } x \in M \}$ where $\partial U_{t}$ has positive mean curvature for $0 \le t \le \theta r_{1}$.
    \item  The following isoperimetric inequality holds: For $t \ge \theta r_{1}/2$, if $\Gamma'$ is a Jordan curve or a Jordan arc in $\partial U'_{t})$ and its preimage $\pi^{-1}(\Gamma')$ divides $\partial U_{t}$ into two components $E_{1}$ and $E_{2}$, then there exists $k \in \{3, ..., n \}$ such that
    \begin{equation} \label{sliceisopineq}
        (\min(\mathcal{H}^{n}(E_{1}),\mathcal{H}^{n}(E_{2})))^{k-1} \le \beta \mathcal{H}^{n-1}(\pi^{-1}(\Gamma'))^{k}
    \end{equation}
    for some $\beta>0$ depending on $(M,g)$, $U$ and $\theta$.
\end{enumerate}
Also let
\begin{equation*}
    \delta_{2} = \min \{\delta_{1} , (4(1+ 4c_{1}))^{-1} \theta, (2n)^{-1} \beta^{- \frac{k-1}{n}} \theta^{\frac{k}{n}} \min (1, r_{1}^{-1}) \}.
\end{equation*}
Let $\Sigma'$ be a smooth embedded disk intersecting $\partial U'$ transversally, $\partial \Sigma' \subset M' \setminus U'$, and $\Lambda'$ be a connected component of $\Sigma' \setminus U'$ with $\partial \Sigma' \cap \Lambda' = \emptyset$. Denote $\Sigma = \pi^{-1}(\Sigma')$ and $\Lambda = \pi^{-1}(\Lambda')$. Suppose 
\begin{equation} \label{smallareabound}
    \mathcal{H}^{n}(\Sigma)+\mathcal{H}^{n}(\partial U) \le \frac{1}{2^{n+2}} \delta_{2}^{n} r_{1}^{n}.
\end{equation}
Then there exists a $G_{c}$-equivariant hypersurface $F \subset \partial U$ whose image $F' := \pi(F) \subset \partial U'$ is a surface such that $\partial F' \cap M' = \Lambda' \cap \partial U' \cap M'$ and
\begin{equation} \label{areabound}
    \mathcal{H}^{n}(F) < \mathcal{H}^{n}(\Lambda \cap U_{\theta r_{1}}).
\end{equation}
Moreover, Then there exists unique compact set $K_{\Sigma}$ bounded by $F \cup \Lambda $ with
\begin{equation}\label{smallsetvolume}
     \mathcal{H}^{n+1}(K_{\Sigma}) < c_{1} \delta_{2}^{n+1} r_{1}^{n+1}.
\end{equation}
\begin{proof}
We can consider $F_{0}' \subset \partial U'$ such that $ \partial F_{0}' \setminus \partial M' = \partial \Lambda' \setminus \partial M'$ since $\partial U' \cap M'$ is simply connected. Let us take preimage $F_{0} = \pi^{-1} (F_{0}')$. Note that $F_{0}$ and $\Lambda$ are hypersurfaces with boundary on $M$. Then by (\ref{smallareabound}) and Lemma \ref{tubeisopineq}, we obtain a compact set $W$ with
\begin{equation*}
    \mathcal{H}^{n+1}(W) \le c_{1} \delta_{2}^{n+1} r_{1}^{n+1} \text{ and } \partial W = F_{0} \cup \Lambda.
\end{equation*}
Let us set $K_{\Sigma} = W \setminus U$m then we have \ref{smallsetvolume} by Lemma \ref{tubeisopineq} and (\ref{smallareabound}). For $t \ge 0$, we define
\begin{equation*}
    F_{t} = K_{\Sigma} \cap \partial U_{t} \text{ and } E_{t} = \Lambda \cap U_{t}.
\end{equation*}
By the mean convexity of $\partial U_{t}$ for $0 \le t \le \theta r_{1}$, we have $\Delta d(U,x) >0$ on $U_{\theta r_{1}}$. The divergence theorem implies the following for $0 \le t_{1} \le t_{2} \le \theta r_{1}$:
\begin{equation} \label{areadiff}
    \mathcal{H}^{n}(F_{t_{1}}) - \mathcal{H}^{n}(F_{t_{2}}) \le \int_{\overline{E_{t_{2}}}\setminus E_{t_{1}}} |\nabla d(U, \cdot) \cdot \nu| \le \mathcal{H}^{n}(E_{t_{2}}) - \mathcal{H}^{n}(E_{t_{1}}),
\end{equation}
where $\nu$ is the unit normal vector field of $\Lambda$ pointing out of $K_{\Sigma}$. We have (\ref{areadiff}) since $|\nabla d(U,\cdot)| \le 1$. Now we consider the case $|F_{\theta r_{1}}| \neq 0$. By following arguments of Lemma 3 in \cite{meeks1982embedded} with (\ref{smallsetvolume}) by coarea formula,
\begin{equation*}
\mathcal{H}^{n}(F_{t}) \le 4 c_{1} \theta^{-1} \delta_{2}^{n+1} r_{1}^{n}
\end{equation*}
for a subset of $[0,\theta r_{1}]$ whose Lebesgue measure is at least $3\theta r_{1}/4 $. Since we took $\delta_{2} < (16c_{1})^{-1} \theta$, we have the area bound of each slice from (\ref{smallareabound}) and (\ref{areadiff}) that
\begin{equation} \label{boundfromtheta}
    \mathcal{H}^{n}(F_{t}) \le \frac{1}{2} \delta_{2}^{n} r_{1}^{n}
\end{equation}
for all $t \in [0, 3\theta r_{1}/4]$. Since $U_{\frac{\theta r_{1}}{2}}$ contains a geodesic ball of radius $\theta r_{1}/2$, we have
\begin{equation} \label{containgeodball}
    \mathcal{H}^{n}(\partial U_{t}) \ge \mathcal{H}^{n+1}(U_{t})^{\frac{n}{n+1}} \ge \frac{1}{2^{n}} \theta^{n} r_{1}^{n}
\end{equation}
for $t \in [\theta r_{1}/2, \theta r_{1}]$ by isoperimetric inequality in $\mathbb{R}^{n+1}$. Since $\delta_{2} < \theta/4$, we have
\begin{equation} \label{comparisonFU}
    \mathcal{H}^{n}(F_{t}) \le \frac{1}{2} \mathcal{H}^{n}(\partial U_{t})
\end{equation}
for $t \in [\theta r_{1}/2, 3\theta r_{1}/4 ]$ by (\ref{boundfromtheta}) and (\ref{containgeodball}). We apply the isoperimetric inequality (\ref{sliceisopineq}) and coarea formula and obtain
\begin{equation}\label{sliceisopineqappl}
    \mathcal{H}^{n}(F_{t}) \le \beta^{\frac{k}{k-1}} \mathcal{H}^{n-1}(\partial F_{t})^{\frac{k}{k-1}} = \beta^{\frac{k}{k-1}}\mathcal{H}^{n-1}(\partial E_{t} \setminus \partial U')^{\frac{k}{k-1}} \le \beta^{\frac{k}{k-1}} \Big( \frac{d}{dt} \mathcal{H}^{n}(E_{t}) \Big)^{\frac{k}{k-1}}
\end{equation}
for $t \in [\theta r_{1}/2, 3\theta r_{1}/4 ]$ almost everywhere. By (\ref{areadiff}) and (\ref{sliceisopineqappl}), we have
\begin{equation} \label{areaode}
\mathcal{H}^{n}(F) - \mathcal{H}^{n}(E_{t}) \le \mathcal{H}^{n}(F_{t}) \le \beta^{\frac{k}{k-1}} \Big( \frac{d}{dt}(\mathcal{H}^{n}(F) - \mathcal{H}^{n}(E_{t})) \Big)^{\frac{k}{k-1}}
\end{equation}
for $t \in [\theta r_{1}/2, 3\theta r_{1}/4 ]$ almost everywhere. By integrating (\ref{areaode}) over $[\theta r_{1}/2, 3\theta r_{1}/4 ]$, we obtain
\begin{equation} \label{areaintegral}
    ((\mathcal{H}^{n}(F) - \mathcal{H}^{n}(E_{\theta r_{1}/2})))^{\frac{1}{k}} -  ((\mathcal{H}^{n}(F) - \mathcal{H}^{n}(E_{3 \theta r_{1}/4})))^{\frac{1}{k}} \ge \frac{1}{4k} \beta^{-\frac{1}{k}} \theta r_{1}.
\end{equation}
Suppose $\mathcal{H}^{n}(F) \ge \mathcal{H}^{n}(E_{3 \theta r_{1}/4})$. Then by (\ref{areaintegral}), we have
\begin{equation}
    (\mathcal{H}^{n}(F))^{\frac{1}{k}} \ge  ((\mathcal{H}^{n}(F) - \mathcal{H}^{n}(E_{\theta r_{1}/2}))^{\frac{1}{k}} \ge \frac{1}{4k} \beta^{-\frac{1}{k}} \theta r_{1}.
\end{equation}
However, this contradicts our choice of $\delta_{2}$ because of the area bound (\ref{smallareabound}). Hence we have $\mathcal{H}^{n}(F) < \mathcal{H}^{n}(E_{3 \theta r_{1}/4}) < \mathcal{H}^{n}(\Lambda \cap U_{\theta r_{1}})$.
\end{proof}
\end{lem}
\begin{rmk}
We will consider $U$ to be a small tubular neighborhood of $G_{c} \cdot x$ and determine $k+1$ to be a dimension of slice of $G_{c}(x)$ so $k = n- \dim (G_{c}(x))$. Then (\ref{sliceisopineq}) is an isoperimetric inequality with symmetry from the slice theorem (Theorem \ref{slicethm}).
\end{rmk}
We now obtain the replacement lemma in our setting and the idea of Theorem 2 in \cite{almgren1979existence} applies here (See Lemma 4.4 in \cite{jost1986existence} for the statement for the boundary version). Note that the replacement is performed on $M'$ and we focus on the corresponding area in $M$ along the replacement. Let $\mathcal{M}(0)$ be the set of all embedded genus zero surfaces $D'$ (which may not be connected) contained in $M'$ whose preimage is a smooth embedded $G_{c}$-equivariant hypersurfaces in $M$. We also denote $\mathcal{M}(0,m)$ as the set of surfaces whose number of boundary components is less than equal to $m$ among surfaces in $\mathcal{M}(0)$.
\begin{lem}  [Replacement Lemma] \label{replacement} Suppose $U$ and $\Sigma$ satisfies the hypotheses in Lemma \ref{areacomparison}. Let $\theta>0$ and $\Sigma' \in \mathcal{M}(0,m)$ intersects $U'$ transversally. Here we define $U_{\theta} =\{ x \, | \, d( x,\partial U) \ge \theta \text{ and } x \in U \}$ and $U'_{\theta} = \pi(U_{\theta})$. Moreover, suppose $\partial \Sigma' \cap \operatorname{Int}(M')$ is not contained in any set $\pi(K_{\Sigma})$. Then there exists $\tilde{\Sigma}' \in \mathcal{M}(0,m')$ with $m' \le m$ and with $\tilde{\Sigma} = \pi^{-1}(\tilde{\Sigma}')$, such that
\begin{enumerate}
    \item $\partial \tilde{\Sigma}' \cap \operatorname{Int} (M') = \partial \Sigma' \cap \operatorname{Int}(M')$, $ \tilde{\Sigma}' \setminus U' \subseteq \Sigma' \setminus U'$ and $\tilde{\Sigma}' \cap U'_{\theta} \subseteq \Sigma' \cap U'_{\theta}$;
    \item $\tilde{\Sigma}'$ intersects $U'$ transversally;
    \item $\mathcal{H}^{n}(\tilde{\Sigma}) + \mathcal{H}^{n}((\Sigma \setminus \tilde{\Sigma}) \cap U_{\theta}) \le \mathcal{H}^{n}(\Sigma)$;
    \item $\tilde{\Sigma}' \cap \overline{U'} = \cup_{j=1}^{l} N_{j}' $ where $N_{j}' \in \mathcal{M}(0)$.
\end{enumerate}
Moreover, if
\begin{equation} \label{genuszeroareabound}
    \mathcal{H}^{n}(\Sigma) \le \mathcal{H}^{n}(P)+ \epsilon 
\end{equation}
for every $P$ with $P':= \pi(P) \in \mathcal{M}(0,m)$ and $ \Sigma' \cap \partial_{rel}U' =  P' \cap \partial_{rel}U'$, then there exists $\epsilon_{1}, ..., \epsilon_{l}>0$ with $\sum_{j=1}^{l} \epsilon_{j} \le \epsilon$ and $N_{j} := \pi^{-1} (N_{j}')$ satisfies 
\begin{equation} \label{genuszerodecomp} 
    \mathcal{H}^{n}(N_{j}) \le \mathcal{H}^{n}(P_{j}) + \epsilon_{j}
\end{equation}
for any $P_{j}$ such that $P_{j}' := \pi(P_{j}) \in \mathcal{M}(0)$ with $\partial N_{j}' \cap \operatorname{Int}(M') = \partial P_{j}' \cap \operatorname{Int}(M')$ and $P_{j}'$ having at most as many boundary components on $\partial M'$ as $N_{j}'$. If (\ref{genuszeroareabound}) holds for every $P' \in \mathcal{M}(0)$, then we allow $P_{j}'$ to have arbitrary number of boundary components on $\partial M'$ in (\ref{genuszerodecomp}).
\begin{proof} The replacement argument in \cite[Theorem 1]{almgren1979existence} and \cite[Lemma 4]{jost1986existence} directly works here. Moreover, since the number of boundary components in $\partial M'$ does not increase along the replacement procedure in \cite[Theorem 1]{almgren1979existence} by \cite[Lemma 4]{jost1986existence}.
\end{proof}
\end{lem}
Now we have a slice filigree lemma which is a Lie group equivariant version of Lemma 3 in \cite{almgren1979existence}. We denote a disk bundle in Theorem \ref{slicethm} by $E_{r}^{G_{c}(x)}$ generated on $G_{c}(x)$ with a radius $r$. Denote a disk on $E_{r}^{G_{c}(x)}$ by $D_{r}^{y}$ which is a fiber at $y \in G_{c}(x)$ and its image slice by $S_{r}^{y} := \exp_{y}(D_{r}^{y}) \subseteq M \cap B_{r}^{G_{c}}(x)$. We prove the area bound on the slice since we will use this bound to prove the rectifiability later.
\begin{lem}[Slice filigree Lemma] \label{slicefiligree} For $x \in M$ and small $\rho>0$, suppose $S_{\rho}^{x} \subset M^{n+1}$ is a $(k+1)$-dimensional slice. Suppose $\{ Y_{t} \}_{t \in [0,1]} \subset B_{\rho}^{G_{c}}(x)$ is a family of $G_{c}$-equivariant open mean convex set satisfying the conditions in Lemma \ref{areacomparison} with $Y_{t} = \{ x \in M \, | \, f(x)<t \}$ for a nonnegative $G_{c}$-equivariant $C^{1}$-function $f$ on $M$ and suppose $\sup_{\overline{Y_{1}} \setminus \overline{Y_{0}}} |Df| \le c_{1}$ for some constant $c_{1}>0$.

Suppose also that there is a constant $c_{2} < \infty$ such that whenever $\pi(\Gamma_{1})$ is a closed Jordan curve or a Jordan arc with endpoints on $\partial M' \cap \pi(\partial Y_{t})$ then there is a region $E' \subset (M' \cap \pi(\partial Y_{t}))$ with $\partial E' \setminus \partial M' = \pi(\Gamma_{1})$ such that
    \begin{equation} \label{sliceisopineq2}
        \mathcal{H}^{n}(E) \le c_{2} \mathcal{H}^{n-1}(\Gamma_{1})^{\frac{k}{k-1}},
    \end{equation}
where $E := \pi^{-1}E'$. Suppose $\Sigma^{n} \subseteq M^{n+1}$ is a $G_{c}$-equivariant hypersurface and $\Sigma':= \pi(\Sigma) \in \mathcal{M}(0,m)$ whose relative boundary is $\Gamma' = \Sigma' \cap \pi(\partial B_{r}^{\overline{g}} (x'))$ is a union of closed Jordan curves or Jordan arcs with endpoints on $\partial M'$, and suppose $\partial \Sigma' \cap Int(M')$ is on outside of $\pi(Y_{t})$ for $t \in [0,1]$. We now assume that for some $\epsilon>0$,
\begin{equation} \label{areatightening}
    \mathcal{H}^{n}(\Sigma) \le \mathcal{H}^{n}(\pi^{-1}(\Sigma'')) +\epsilon \text{ for all } \Sigma'' \in \mathcal{M}(0,m')
\end{equation}
for $m' \le m$ with $\partial \Sigma'  \setminus \partial M' = \partial \Sigma'' \setminus \partial M'$. Then
\begin{equation} \label{filigreeestimate}
    \mathcal{H}^{k}(\Sigma \cap (S^{x}_{\rho} \cap Y_{t})) \le 4 \mathcal{H}^{n-k}(G_{c}(x))^{-1} \epsilon
\end{equation}
for $t \in [0,1-2^{\frac{k-1}{k}}k c_{1} c_{2}^{\frac{k-1}{k}} \mathcal{H}^{n-k}(G_{c}(x))^{\frac{1}{k}}\mathcal{H}^{k}(\Sigma \cap (S^{x}_{\rho} \cap Y_{1}))^{\frac{1}{k}}]$.
\begin{proof}

 By Sard's theorem, $\Sigma \cap S_{\rho}^{x}$ intersects $Y_{t} \cap S_{\rho}^{x}$ transversally for almost all $t \in [0,1]$. Notice that for almost all $t \in [0,1]$, $\Sigma' \cap \pi(\partial Y_{t})$ on $M'$ consists of Jordan curves and Jordan arcs connecting two points on $\partial M'$, and the preimage of each curve separates $\partial (Y_{t} \cap S_{\rho}^{x})$ to two connected regions and we can apply an isoperimetric inequality (\ref{sliceisopineq2}).

We apply Lemma \ref{replacement} with $Y_{t}':=\pi(Y_{t})$ in place of $U'$ and obtain $\tilde{\Sigma}'$ with its preimage $\tilde{\Sigma}= \pi^{-1}(\tilde{\Sigma}')$ such that $\mathcal{H}^{n}(\tilde{\Sigma}) \le \mathcal{H}^{n}(\Sigma)$,  $\partial \tilde{\Sigma}' \cap Int(M') =  \partial \Sigma' \cap Int(M')$, and $\tilde{\Sigma}' \cap \overline{Y_{t}'} = \cup_{j=1}^{l} N_{j}' $ where $N_{j}' \in \mathcal{M}(0)$ and $N_{j} := \pi^{-1} (N_{j}')$ satisfies 
\begin{equation} \label{pieceineq}
    \mathcal{H}^{n}(N_{j}) \le \mathcal{H}^{n}(P_{j}) + \epsilon_{j}
\end{equation}
for any $P_{j}$ such that $P_{j}' := \pi(P_{j}) \in \mathcal{M}(0)$ with $\partial N_{j}' \cap  Int(M') = \partial P_{j}' \cap  Int(M')$ and $P_{j}'$ having at most as many boundary components on $\partial M'$ as $N_{j}'$, where $\epsilon_{1}, ..., \epsilon_{l}>0$ with $\sum_{j=1}^{l} \epsilon_{j} \le \epsilon$.

Let $E^{j}_{t} = N_{j} \cap S^{x}_{\rho}$ and $\Gamma^{j}_{t} = \pi^{-1}(\partial N_{j}' \setminus \partial M') \cap S^{x}_{\rho}$. Then by (\ref{sliceisopineq2}), (\ref{pieceineq}), Theorem \ref{slicethm} and considering that $B_{\rho}^{G}(x)$ is a small $G_{c}$-tubular neighborhood of $x$, we have
\begin{equation} \label{orbitineq1}
    \mathcal{H}^{k}(E^{j}_{t}) \le 2 c_{2} \mathcal{H}^{n-k}(G_{c}(x))^{\frac{1}{k-1}} \mathcal{H}^{k-1}(\Gamma^{j}_{t})^{\frac{k}{k-1}} +   \mathcal{H}^{n-k}(G_{c}(x))^{-1} \epsilon_{j}.
\end{equation}
By summing over $j$, we obtain
\begin{equation} \label{orbitineq2}
    \mathcal{H}^{k}(\tilde{\Sigma} \cap (Y_{t} \cap S^{x}_{\rho})) \le  2 c_{2}\mathcal{H}^{n-k}(G_{c}(x))^{\frac{1}{k-1}} \mathcal{H}^{k-1}(\tilde{\Sigma} \cap  (\partial Y_{t} \cap S^{x}_{\rho}))^{\frac{k}{k-1}} + \mathcal{H}^{n-k}(G_{c}(x))^{-1} \epsilon.
\end{equation}
Moreover, by combining (\ref{areatightening}) and (\ref{orbitineq2}), we have
\begin{align}
   \nonumber \mathcal{H}^{k}(\Sigma \cap (Y_{t} \cap S^{x}_{\rho})) &\le 2 c_{2}\mathcal{H}^{n-k}(G_{c}(x))^{\frac{1}{k-1}} \mathcal{H}^{k-1}(\tilde{\Sigma} \cap  (\partial Y_{t} \cap S^{x}_{\rho}))^{\frac{k}{k-1}} + 2 \mathcal{H}^{n-k}(G_{c}(x))^{-1} \epsilon \\ \label{orbitineq3}
   &\le  2 c_{2}\mathcal{H}^{n-k}(G_{c}(x))^{\frac{1}{k-1}} \mathcal{H}^{k-1}(\Sigma \cap  (\partial Y_{t} \cap S^{x}_{\rho}))^{\frac{k}{k-1}} + 2 \mathcal{H}^{n-k}(G_{c}(x))^{-1} \epsilon,
\end{align}
where we obtain (\ref{orbitineq3}) by $\partial \tilde{\Sigma}' \setminus Y_{t}' \subseteq  \partial \Sigma' \setminus Y_{t}'$. Take $g(t) = \mathcal{H}^{k}(\Sigma \cap (Y_{t} \cap S^{x}_{\rho})) - 4 \mathcal{H}^{n-k}(G_{c}(x))^{-1} \epsilon$ and $t_{0} := \inf \{t \, | \, g(t) \ge 0 \}$ then by co-area formula we have
\begin{align}
    \nonumber g(t) &\le 2 c_{2} \mathcal{H}^{n-k}(G_{c}(x))^{\frac{1}{k-1}} \mathcal{H}^{k-1}(\Sigma \cap  (\partial Y_{t} \cap S^{x}_{\rho}))^{\frac{k}{k-1}} \\ \label{sliceode}
    &\le 2 c_{1}^{\frac{k}{k-1}} c_{2} \mathcal{H}^{n-k}(G_{c}(x))^{\frac{1}{k-1}} g'(t)^{\frac{k}{k-1}}.
\end{align} By integrating (\ref{sliceode}) from $t_{0}$ to $1$, we have
\begin{equation*}
    \frac{1-t_{0}}{2^{\frac{k-1}{k}}k c_{1} c_{2}^{\frac{k-1}{k}} \mathcal{H}^{n-k}(G_{c}(x))^{\frac{1}{k}}} \le \sqrt[k]{g(1)}-\sqrt[k]{g(t_{0})}
\end{equation*}and conclude with (\ref{filigreeestimate}).
\end{proof}
\end{lem}
\begin{rmk}
$c_{2}$ in Lemma \ref{slicefiligree} depends on the orbit size $\mathcal{H}^{n-k}(G_{c}(x))$ in our applications. In spirit of the slice theorem and the isoperimetric inequality on the round sphere, for instance, we can take $c_{2}$ to satisfy $c_{2} (\mathcal{H}^{n-k}(G_{c}(x)))^{\frac{1}{k-1}} \le 2 IS_{k}$ for small $\rho$ where $IS_{k}$ is a constant of an isoperimetric inequality on $\mathbb{S}^{k}$.
\end{rmk}
\subsection{Regularity on singular orbits} We now discuss the regularity of an isotopy minimizer at 
 the preimage of interior and boundary points in $M$ which correspond to singular orbits on $M$. We obtain a minimizing sequence $\{ \Sigma_{k}' \}$ of the minimization problem $(\Sigma', \mathcal{I}(U'))$ which of $\Sigma_{k}'$ intersects $\partial M'$ transversally for each $k$. By virtue of Theorem \ref{genuszero}, it is enough to consider minimizing sequences of genus $0$ surfaces on $M'$. We first prove the rectifiability of limit varifolds.

\subsubsection{The limit hypersurface does not contain singular points in the preimage of interior points of $M'$} We here prove that the support of the limit varifold does not include singular points in the preimage of interior points of $M'$. We first prove that each connected component of singular orbits in the quotient space $M'$ is an isolated point. Let $\mathcal S$ denote the intersection of union of the interior singular $G_c$-orbits and $\pi^{-1}(\operatorname{Int}(M'))$.
\begin{lem} \label{isolatedinterior}
    Suppose $S'$ is a connected component $\pi(\mathcal{S})$. Then $S'$ consists of a single point.
\end{lem}
\begin{proof}
    Fix $x\in \mathcal{S}$, for a small $\rho>0$, and let us consider a slice $S^{x}_{\rho}$. Then the $(G_{c})_{x}$-action on $S^{x}_{\rho}$ has cohomogeneity two. We define $K_{x}$ to be an orbit on $S^{x}_{\rho}$ containing $x$. We claim that this action has no singular orbits. Indeed, if it had a singular orbit, \cite[Proposition IV.8.3]{bredon1972introduction} imply that $S^{x}_{\rho}/ K_{x}$ has nonempty boundary and homeomorphic to $D^{2}$. Note that, $(G_{c})_{x}/K_{x}$ is a finite group. Then for small $\rho>0$, we have that $\pi^{-1}( \partial (S^{x}_{\rho}/ (G_{c})_{x}))$ lies on $\partial M'$. This contradicts $\pi(x)$ is an interior point in $M'$.

    Hence, every orbit in $S^{x}_{\rho}$ is principal or exceptional for the slice
    action, $G_{c} \cdot x$ is the only singular orbit in the neighborhood and isolated.
\end{proof}

Now we prove that the support of the limit varifold does not pass through interior singular orbits. We apply the classification of stable cone by Simons \cite{simons1968minimal}.
\begin{prop} \label{prop:nointeriorsingularorbit}
    Consider a minimization problem $(\Sigma', \mathcal{I}(U'))$ and assume $\Sigma^i=\pi^{-1}((\Sigma')^i)\rightarrow V$ as a varifold. Then $\operatorname{spt} V \cap \mathcal{S} = \emptyset$.
\end{prop}
\begin{proof}
    By Lemma \ref{isolatedinterior}, all singular orbits corresponding to the interior point of $M'$ are isolated. Suppose there is a point $x\in \operatorname{spt}\|V\|\cap \mathcal S$ and let us prove by contradiction. Then, for small $0<\rho_{2}<\rho_{1}$, by the regularity in the principal and exceptional orbits Theorem \ref{intregular}, $\pi((B^{\rho_{1}}_{G_{c}}(x)/B^{\rho_{2}}_{G_{c}}(x)) \cap \operatorname{spt} V)$ is a properly embedded and hence separates $\pi((B^{\rho_{1}}_{G_{c}}(x)/B^{\rho_{2}}_{G_{c}}(x)))$ since it is simply connected.

    Then $(\Sigma', \mathcal{I}(U'))$ is a minimizing problem and $V$ is a $G_{c}$-stable minimal hypersurface. Since $V$ separates $B^{\rho}_{G_{c}}(x)$ to two parts, we can apply Lemma \ref{equivstablity} and have a stability of $V \mres B^{\rho}_{G_{c}}(x)$. By slice theorem, the tangent cone at $V$ splits into $T_{x}(\operatorname{spt}V^{S}) \times \mathbb{R}^{\dim(G_{c} \cdot x)}$, where $\operatorname{spt}(V)$ has at most one singularity point. Since $G_{c} \cdot x$ has codimension at least $4$, we apply Simons' classification of stable cone \cite{simons1968minimal} and $T_{x} V$ is a hyperplane (possibly with multiplicity).

    Since $\operatorname{spt}(V)$ is $G_c$-invariant and contains $x$, we have $T_x (G_{c} \cdot x)\subset T_x (\operatorname{spt}(V))$. Hence a unit normal $\nu$ to $\operatorname{spt}(V)$ at $x$ and is orthogonal to $G_{c} \cdot x$. Moreover, $K_{x}$ preserves the normal line and it acts trivially on this line. This contradicts $G_{c} \cdot x$ is isolated in the slice since there does not exist such nontrivial normal vector $\nu$. 
\end{proof}
\subsubsection{Regularity on singular orbits on boundary}

\begin{lem}[Rectifiability] \label{rectifiability} Let $D_{l}' \in \mathcal{M}(0)$ be a minimizing sequence and suppose the pullback sequence of $n$-varifolds $\{V_{l} \}_{l \in \mathbb{N}}$ where $\pi_{\sharp}V_{l} = D_{l}'$ converges to the varifold limit $V$ in $V_{n}(M)$. Then $V$ is rectifiable.
\begin{proof}
It is enough to consider a small neighborhood of each point on $M$. The rectifiability on small neighborhoods of points on principal and exceptional orbits directly comes out from arguments of Theorem 2 in \cite{almgren1979existence}. Also, it is enough to consider singular orbits in $\pi^{-1}(\partial M')$ by Proposition \ref{prop:nointeriorsingularorbit}.

We now fix $x \in M$ which is on a singular orbit and prove the rectifiability on a small neighborhood $B_{\rho}^{G_{c}}(x) \subseteq M$ and consider an associated slice $S$ where $B_{\rho}^{G_{c}}(x) = G_{c} \times _{(G_{c})_{x}} S$. Since $V_{l}$ is a $G_{c}$-equivariant $n$-varifold, we take an induced $k$-varifolds $V_{l}^{S}$ where $spt(V_{l}^{S}) \subset S$ and $spt(V_{l}^{S}) \rightarrow spt(V_{l}) \rightarrow G_{c}(x)$ is a subbundle of $B_{\rho}^{G_{c}}(x)$ as a disk bundle in Theorem \ref{slicethm}. Moreover, $k$-Grassmannian measure of $V_{l}^{S}$ is induced by a projection of $n$-Grassmannian measure from $T_{y}M$ to $T_{y}(S)$ for $y \in spt (V_{l}^{S})$. Also take $V^{S}$ correspondingly as a varifold whose support is a projection of the support of the limit varifold $V$.

We will show that there exists a constant $c>0$ such that 
\begin{equation} \label{densitybound}
\theta^{k}(V^{S},y) \ge c.    
\end{equation}
Suppose $B^{G_{c}}_{s}(x)$ satisfies the condition (\ref{areatightening}) in Lemma \ref{slicefiligree} for $0<s<\rho$, and $|V_{l}| \le |V| + \epsilon_{l}$ and $\epsilon_{l} \rightarrow 0$ as $l$ goes to infinity. By the area formula, we have
\begin{equation}
    |V_{l}^{S}| \le |V^{S}| + 2\mathcal{H}^{n-k}(G_{c}(x))^{-1} \epsilon_{l}.
\end{equation}
Since (\ref{areatightening}) follows from the area formula on $G_{c}$-orbit, we apply the slice filigree lemma (Lemma \ref{slicefiligree}) with $f(y) = d_{S}(y,x) / \rho$ where $y \in S$ and $d_{S}$ is a distance function on $S$, $c_{1} = 1/\rho$, and $Y_{t} = B_{x}^{G_{c}}(t \rho)$ and there exists corresponding $c_{2}$ for cylindrical bundle over $G_{c}(x)$. Let us put $c' = 2^{\frac{k-1}{k}}k  c_{2}^{\frac{k-1}{k}} \mathcal{H}^{n-k}(G_{c}(x))^{\frac{1}{k}}$. Suppose there exists a subsequence $\{ l' \} \subset \{ l \}$ with $\mathcal{H}^{k}(V_{l'}^{S}) \le (1/2c')^{k}\rho^{k}$, then by applying Lemma \ref{slicefiligree} we have $|V^{S} \mres B_{\rho/2}^{G}(x)|=0$. This contradicts $x \in spt(V^{S}) \subset spt(V)$. Hence for all sufficiently large $l'$ we have
\begin{equation*}
|V_{l'}^{S}| \ge (1/2c')^{k}\rho^{k}
\end{equation*}
and by taking $c= k^{-1} \omega_{k}^{-1} (1/2c')^{k}$ we obtain (\ref{densitybound}). By applying Allard's rectifiability theorem in 5.5 of \cite{allard1972first} to $V$ with corresponding density bound and the stationarity of $V$, we have the rectifiability of $V$ on the neighborhood of $x$. Since we can apply this for finite number of strata and $M$ is compact, we have the global rectifiability of $V$.
\end{proof}
\end{lem}
We prove the regularity theorem assuming that every tangent cone on points on singular orbits is a hyperplane first.

\begin{thm} \label{firstcontireg}Suppose $\{N_{l}' \} \in \mathcal{M}(0)$ is a minimizing sequence and the pullback sequence of $n$-varifolds $\{V_{l} \}_{l \in \mathbb{N}}$ where $\pi_{\sharp}V_{l} = N_{l}'$ converges to the varifold limit $V$ in $V_{n}(M)$. Then $V$ is a stationary varifold with the property that if $x \in spt(V)$ and suppose $V$ has a tangent cone $C$ at $x$ with $spt(C) \subseteq H$, where $H$ is a hyperplane, then there is a $\rho >0$ such that $V \mres B^{G_{c}}_{\rho}(x)$ is an integer multiple of smooth $G_{c}$-equivariant minimal hypersurface containing $x$.
\begin{proof}
By Lemma \ref{rectifiability}, $V$ is a stationary and rectifiable manifold with a positive density lower bound for all $x \in spt(V)$. Let us consider $\pi_{\sharp} H$ as a blow-up of an image of $spt(V)$ by projection map $\pi$ at $x$. Since $H$ divides the neighborhood by two components projects to a cone on a conical orbifold, we get topological triviality of $\pi_{\sharp}H$. Hence $\pi_{\sharp}H$ is a topological half-plane whose boundary is on $\partial M'$, where $H$ is a $G_{c}$-equivariant hyperplane.

We denote $D_{\rho} = \pi_{\sharp}H \cap B_{\rho}(0)$ and $K_{\rho, \sigma}$ by $((D_{\rho} \setminus \partial_{rel} D_{\rho}) \times (-\sigma,\sigma)) \cap \pi_{\sharp} T_{x} M$, where the latter component of $D_{\rho} \times (-\sigma,\sigma)$ is given by a local coordinate on $D_{\rho}$. We denote the scaling $\mu_{r}$ on $T_{x}M$ given by $\mu_{r} = ry$ and the diffeomorphism induced by the exponential map $\exp_{x}$ by $\phi_{x}$. Then we have
\begin{equation*}
(\mu_{r_{k}} \circ \phi_{x}^{-1})_{\sharp} V \rightharpoonup mH,
\end{equation*}
where $m>0$ and $r_{k}$ goes to infinity. By the density lower bound (\ref{densitybound}) and the proof of Lemma \ref{rectifiability}, we have a global density lower bound and for any $\sigma_{0} \in (0,1)$ there exists $r>0$ such that
\begin{equation} \label{boundedspt}
    K_{1,1} \cap \pi(spt(( \mu_{r} \circ \phi_{x}^{-1})_{\sharp} V)) \subset K_{1, \sigma_{0}/2}.
\end{equation}
We also can choose $r$ such that
\begin{equation*}
    \mathcal{H}^{n} (spt(( \mu_{r} \circ \phi_{x}^{-1})_{\sharp} V) \cap \pi^{-1}(\partial_{rel}D_{\frac{1}{2}} \times \mathbb{R})) = 0.
\end{equation*}
Since $\{ N_{l}' \}$ is a minimizing sequence, for $N_{l} : = \pi^{-1}(N_{l}')$ and sufficiently large $l$, we have
\begin{equation} \label{dilatedarea}
    \mathcal{H}^{n}(( \mu_{r} \circ \phi_{x}^{-1})_{\sharp} N_{l}) \le \mathcal{H}^{n}(( \mu_{r} \circ \phi_{x}^{-1})_{\sharp} N)+ r^{n} \epsilon_{l}
\end{equation}
for all $N:= \pi^{-1}(N')$ with $\partial N' \setminus \partial M' = \partial N_{l}' \setminus \partial M'$ and $\epsilon_{l} \rightarrow 0$ as $l$ goes to infinity. By (\ref{boundedspt}) and coarea formula we get for almost all $\sigma \in (\sigma_{0}/2,1)$,
\begin{equation*}
    \mathcal{H}^{n-1}(( \mu_{r} \circ \phi_{x}^{-1})_{\sharp} N_{l} \cap \pi^{-1}(D_{1} \times (\{ -\sigma\} \cup \{\sigma\})))\rightarrow 0
\end{equation*}
as $l \rightarrow \infty$. Then for given any $\eta>0$, for sufficiently large $l$, there exists $\sigma_{l} \in (3\sigma_{0}/4,\sigma_{0})$ such that
\begin{equation} \label{dilatedseq}
     \mathcal{H}^{n-1}(( \mu_{r} \circ \phi_{x}^{-1})_{\sharp} N_{l} \cap \pi^{-1}(D_{1} \times (\{ -\sigma_{l}\} \cup \{\sigma_{l}\}))) < \eta.
\end{equation}

Now we explain the modifications we need in our setting which is described in the footnote of \cite[p.~465]{almgren1979existence}. By applying small topological modification on small region of $M'$ and considering its preimage on $M$, we separate curves of the intersection $( \pi((\mu_{r} \circ \phi_{x}^{-1}) N_{l})) \cap \partial K_{\rho_{l},\sigma_{l}}$ into Jordan curves on its disk parts $D_{\rho_{l}} \times \{ -\sigma_{l},\sigma_{l}\} \cap \pi( T_{x} M)$ and its cylinder part $\partial_{rel} D_{\rho_{l}} \times (-\sigma_{l}, \sigma_{l}) \cap \pi( T_{x} M) $ (this means we rule out the intersection on the edge $\partial_{rel} D_{\rho_{1}} \times \{ -\sigma_{l},\sigma_{l}\} \cap \pi( T_{x} M)$ by pushing slightly). Also by applying Sard's theorem, we can assume that $\pi((\mu_{r} \circ \phi_{x}^{-1}) N_{l})$ intersects $\partial K_{\rho_{l},\sigma_{l}}$ transversally. Moreover, we can arrange each part of intersection $\pi((\mu_{r} \circ \phi_{x}^{-1}) N_{l}) \cap (D_{\rho_{l}} \times \{ -\sigma_{l},\sigma_{l}\})$ and $\pi((\mu_{r} \circ \phi_{x}^{-1})_{\sharp} N_{l}) \cap (\partial_{rel} D_{\rho_{l}} \times (-\sigma_{l}, \sigma_{l}))$ intersects $\partial_{rel} D_{\rho_{l}} \times \{ -\sigma_{l}, \sigma_{l} \}$ transversally.

We apply the replacement lemma (Lemma \ref{replacement}) with $\pi((\mu_{r} \circ \phi_{x}^{-1}) N_{l})$ and $K_{\rho_{l}, \sigma_{l}}$ in place of $\Sigma'$ and $U'$, respectively. Notice that $\pi((\mu_{r} \circ \phi_{x}^{-1}) N_{l}) \cap \partial K_{\rho_{l},\sigma_{l}}$ consists of closed Jordan curves and Jordan arc with endpoints on $\partial \pi( T_{x} M)$. Then there exists $0<R_{l}^{1}\le R_{l}^{2}\le R_{l}^{3}\le R_{l}^{4}$ such that there exist genus zero surfaces (with possible boundary components on the boundary $\partial \pi( T_{x} M)$) $P_{l}^{1}, ..., P_{l}^{R_{l}^{4}} \in \mathcal{M}(0)$ with
\begin{align*}
    \partial P_{l}^{1}, ..., \partial P_{l}^{R_{l}^{1}} &\subset D_{\rho_{l}} \times \{ -\sigma_{l},\sigma_{l}\},  \\
    \partial P_{l}^{R_{l}^{1}+1}, ..., \partial P_{l}^{R_{l}^{3}} &\subset \partial D_{\rho_{l}} \times ( -\sigma_{l},\sigma_{l}), \\
    \partial P_{l}^{i} \cap (D_{\rho_{l}} \times \{ -\sigma_{l},\sigma_{l}\}) \neq \emptyset \text{ and } \partial P_{l}^{i} & \cap (\partial D_{\rho_{l}} \times (-\sigma_{l}, \sigma_{l})) \neq \emptyset \text{ for } R_{l}^{3}+1 \le i \le R_{l}^{4},
\end{align*}
where $\partial P_{l}^{R_{l}^{1}+1}, ..., \partial P_{l}^{R_{l}^{2}}$ bound disks in $\partial D_{\rho_{l}} \times (-\sigma_{l}, \sigma_{l})$ and $\partial P_{l}^{R_{l}^{2}+1}, ..., \partial P_{l}^{R_{l}^{3}}$ are homotopic to a homotopically nontrivial curve in $\partial D_{\rho_{l}} \times (-\sigma_{l}, \sigma_{l})$ i.e. Jordan arcs containing two endpoints on separate sides on $\partial \pi( T_{x} M)$. By Lemma \ref{replacement}, we have an area control of preimages of surfaces:
\begin{equation} \label{replaceareabd}
\mathcal{H}^{n}(\pi^{-1}(P_{l}^{i})) \le \mathcal{H}^{n}(\pi^{-1}(P)) + \epsilon_{l,i}
\end{equation}
for $P \in \mathcal{M}(0)$ with $\partial P \setminus \partial \pi( T_{x} M) = \partial P_{l}^{i} \setminus \partial \pi( T_{x} M)$ for $i = 1, ..., R_{l}^{4}$ and $\sum_{i=1}^{R_{l}^{4}} \epsilon_{l,i} \le r^{n} \epsilon_{l}$.
Moreover, we have a convergence toward limit varifold
\begin{equation}\label{convergencevfd}
    (\mu_{r} \circ \phi_{x}^{-1})_{\sharp} V \mres \pi^{-1}(K_{\frac{1}{2},1})= \lim_{l \rightarrow \infty} \sum_{i=1}^{R_{l}^{4}} \pi^{-1}(P_{l}^{i} \cap K_{\frac{1}{2},1}),
\end{equation}
where we took associated varifolds of preimage of surfaces on the right-hand side of (\ref{convergencevfd}).

First of all, we apply the slice filigree lemma (Lemma \ref{slicefiligree}) to discard $P_{l}^{1}, ..., P_{l}^{R_{l}^{1}}$. By an isoperimetric inequality on $\pi^{-1} (D_{\rho_{l}} \times \{ -\sigma_{l},\sigma_{l}\})$, $\partial P_{l}^{i}$ encloses $A_{l}^{i} \subset D_{\rho_{l}} \times \{ -\sigma_{l},\sigma_{l}\}$ such that $\mathcal{H}^{n}(\pi^{-1}(A_{l}^{i})) \le c_{n} \eta^{n/n-1}$ where $c_{n}$ is a constant for $n$-dimensional isoperimetric inequality by (\ref{dilatedseq}). Hence by (\ref{replaceareabd}) and for sufficiently large $l$, we have $\mathcal{H}^{n}(\pi^{-1}(P_{l}^{i})) \le 2 c_{n} \eta^{n/n-1}$. Take $Y_{t} = K_{\rho_{l}, t\sigma_{l}}$ and $c_{2}$ to be the same as Lemma \ref{rectifiability}, and apply Lemma \ref{slicefiligree} with sufficiently small $\eta>0$ together with (\ref{replaceareabd}) we obtain
\begin{equation} \label{filigreebound}
    \mathcal{H}^{n}(\pi^{-1}(P_{l}^{i} \cap K_{\frac{1}{2},\frac{\sigma_{0}}{2}})) \le 8\mathcal{H}^{n-k}(G_{c}(x))^{-1} \epsilon_{l,i}
\end{equation}
by applying (\ref{boundedspt}), (\ref{replaceareabd}) and obtaining area bound of the whole preimage from the slice for $i = 1, ..., R_{l}^{1}$. By sending $l$ to infinity, we obtain
\begin{equation} \label{discardfirst}
    (\mu_{r} \circ \phi_{x}^{-1})_{\sharp} V \mres \pi^{-1}(K_{\frac{1}{2},1})  = \lim_{l \rightarrow \infty} \sum_{i=R_{l}^{1}+1}^{R_{l}^{4}} \pi^{-1}(P_{l}^{i} \cap K_{\frac{1}{2},1}).
\end{equation}

Now we deform $P_{l}^{R_{l}^{3}+1}, ..., P_{l}^{R_{l}^{4}}$ into surfaces with boundary on cylindrical parts with small $n$-dimensional area change on preimages. We employ the fact that $(n-1)$-dimensional measure of the preimage of $P_{l}^{i} \cap D_{\rho_{l}} \times \{ -\sigma_{l},\sigma_{l}\}$ are small for $i = R_{l}^{3}+1, ..., R_{l}^{4}$ and push these pieces of boundary curves into cylinder parts by attaching small surface pieces on $K_{\rho_{l}, \sigma_{l}}$.

We explain this procedure more in detail. Note that we could disregard $P_{l}^{1}, ..., P_{l}^{R_{l}^{1}}$. We denote each piece of $\cup_{i=R_{l}^{3}+1}^{R_{l}^{4}} P_{l}^{i} \cap (D_{\rho_{l}} \times \{ -\sigma_{l},\sigma_{l}\})$ by $\gamma_{l}^{1},..., \gamma_{l}^{p_{l}}$. By relative isoperimetric inequality, each $\gamma_{l}^{j}$ and either $\partial D_{\rho_{l}} \times \{ -\sigma_{l}\}$ or $\partial D_{\rho_{l}} \times \{ \sigma_{l}\}$ bound a disk $\mathcal{D}_{l}^{j}$ such that
\begin{align} \label{isopineqsmallpiece}
    \mathcal{H}^{n} (\pi^{-1}(\mathcal{D}_{l}^{j})) &\le c_{n}' (\mathcal{H}^{n-1}(\pi^{-1}(\gamma_{l}^{j})))^{\frac{n}{n-1}} \\
   \nonumber \mathcal{D}_{l}^{j} \subset D_{\rho_{l}} \times \{ -\sigma_{l}\} &\text{ or } \mathcal{D}_{l}^{j} \subset D_{\rho_{l}} \times \{ \sigma_{l}\}.
\end{align}
Then we attach each $D_{l}^{j}$ along $\gamma_{l}^{j}$ to corresponding $P_{l}^{i}$ and slightly push the surfaces to the interior of $K_{\rho_{l},\sigma_{l}}$ while keeping the embeddedness. Furthermore, we also push pieces of new boundary curves on $\partial D_{\rho_{l}} \times \{ -\sigma_{l},\sigma_{l} \}$ into the interior $\partial D_{\rho_{l}} \times (-\sigma_{l},\sigma_{l})$ by deforming surfaces slightly (we also deform parts of $P_{l}^{i}$ as well, if necessary). We denote this new surfaces by $P_{l}^{\prime R_{l}^{3}+1}, ..., P_{l}^{\prime R_{l}^{4}}$ and $\partial P_{l}^{i} \subset \partial D_{\rho_{l}} \times \{ -\sigma_{l},\sigma_{l} \}$. We denote $N_{l}''$ as a union of $P_{l}^{R_{l}^{1}+1}, ..., P_{l}^{R_{l}^{3}}, P_{l}^{\prime R_{l}^{3}+1}, ..., P_{l}^{\prime R_{l}^{4}}$. Moreover, by considering the area bound (\ref{dilatedarea}), (\ref{dilatedseq}) and (\ref{isopineqsmallpiece}) and summing up all $\mathcal{D}_{l}^{j}$, we have
\begin{equation} \label{pushedareadiff}
    \mathcal{H}^{n}(\pi^{-1}(N_{l}'')) \le \mathcal{H}^{n}(\pi^{-1}(N))+ r^{n} \epsilon_{l} + 2c_{n}' \eta^{\frac{n}{n-1}}
\end{equation}
for all $N$ with $\partial N \setminus\partial \pi( T_{x} M) = \partial(\mu_{r} \circ \phi_{x}^{-1})_{\sharp} N_{l} \setminus\partial \pi( T_{x} M)$. We can divide $\partial P_{l}^{\prime R_{l}^{3}+1}, ..., \partial P_{l}^{\prime R_{l}^{4}}$ into nullhomotopic curves in $\partial D_{\rho_{l}} \times (-\sigma_{l}, \sigma_{l})$ and homotopically nontrivial boundaries as we did before. After reclassifying these, we relabel $N_{l}''$ as a union of $P_{l}^{R_{l}^{1}+1}, ..., P_{l}^{R_{l}^{3}}$ ($R_{l}^{2}$ and $R_{l}^{3}$ subject to be changed along this procedure).

Now we treat $P_{l}^{R_{l}^{1}+1}, ..., P_{l}^{R_{l}^{2}}$ whose boundaries are nullhomotopic in $\partial D_{\rho_{l}} \times (-\sigma_{l}, \sigma_{l})$ to be disregarded. We take $A^{i}_{l}$ to be the part of the cylinder $\partial D_{\rho_{l}} \times (-\sigma_{l}, \sigma_{l})$ filling the boundary $P_{l}^{i}$ for $i = R_{l}^{1}+1, ..., R_{l}^{2}$. By applying Theorem \ref{slicethm} and the area formula, we have
\begin{equation} \label{cylinderareabd}
    \mathcal{H}^{n}(\pi^{-1}(A^{i}_{l})) < \mathcal{H}^{n}(\pi^{-1}(\partial D_{\rho_{l}} \times (-\sigma_{l}, \sigma_{l})) < 4 \mathcal{H}^{n-k}(G_{c}(x)) w_{k-1} \rho_{l}^{k-1} \sigma_{l}, 
\end{equation}
where $w_{k-1}$ is an area of a unit $(k-1)$-sphere. Hence by (\ref{replaceareabd}), (\ref{isopineqsmallpiece}) and (\ref{cylinderareabd}), we obtain
\begin{equation*}
    \mathcal{H}^{n}(\pi^{-1}(P^{i}_{l})) < 4 \mathcal{H}^{n-k}(G_{c}(x)) w_{k-1} \rho_{l}^{k-1} \sigma_{l} + \epsilon_{l,i} + 2c_{n}' \eta^{\frac{n}{n-1}}.
\end{equation*}
We apply Lemma \ref{slicefiligree} again with $Y_{t} = \pi^{-1}(D_{t\rho_{l}} \times \mathbb{R})$, $c_{1} = \rho_{l}^{-1}$, and $c_{2}$ depending on $k$ and $G_{c}(x)$. Since $\sigma_{l}$ is sufficiently small and we take sufficiently small $\eta>0$, similarly with (\ref{filigreebound}) we obtain,
\begin{equation} \label{filigreebound2}
    \mathcal{H}^{n}(\pi^{-1}(P_{l}^{i} \cap K_{\frac{1}{2},\frac{\sigma_{0}}{2}})) \le 8\mathcal{H}^{n-k}(G_{c}(x))^{-1} \epsilon_{l,i}
\end{equation}
for $i = R_{l}^{1}+1, ..., R_{l}^{2}$. Hence we can discard $P_{l}^{R_{l}^{1}+1}, ..., P_{l}^{R_{l}^{2}}$ and have
\begin{equation} \label{discardsecond}
    (\mu_{r} \circ \phi_{x}^{-1})_{\sharp} V \mres \pi^{-1}(K_{\frac{1}{2},1}) = \lim_{l \rightarrow \infty} \sum_{i=R_{l}^{2}+1}^{R_{l}^{3}} \pi^{-1}(P_{l}^{i} \cap K_{\frac{1}{2},1}).
\end{equation}
from (\ref{discardfirst}).

Notice that $\partial P_{l}^{R_{l}^{2}+1}, ..., \partial P_{l}^{R_{l}^{3}}$ are homotopically nontrivial on $\partial D_{\rho_{l}} \times (-\sigma_{l}, \sigma_{l})$. Since each $\pi^{-1}(P_{l}^{i})$ is in the small $G_{c}$-equivariant cylindrical neighborhood of $\pi^{-1}(D_{\rho_{l}} \times \{ 0 \})$ which is a $k$-disk bundle on $G_{c}(x)$ and by the area formula we have
\begin{equation} \label{diskarealower}
    \mathcal{H}^{n}(\pi^{-1}(P_{l}^{i} \cap K_{t,\sigma_{l}})) \ge \mathcal{H}^{n-k}(G_{c}(x)) k^{-1} w_{k-1} t^{k},
\end{equation}
for $t \in [0, \rho_{l}]$ and $i = R_{l}^{2}+1, ..., R_{l}^{3}$. On the other hand, let $A_{l}^{i}$ be a component of $\partial K_{\rho_{l},\sigma_{l}} \setminus \partial P_{l}^{i}$ containing $D_{\rho_{l}} \times \{ -\sigma_{l} \}$. Then as in (\ref{cylinderareabd}) we have
\begin{align}
    \nonumber \mathcal{H}^{n}(\pi^{-1}(A^{i}_{l})) &< \mathcal{H}^{n}(\pi^{-1}(D_{\rho_{l}} \times \{ -\sigma_{l} \}) + \mathcal{H}^{n}(\pi^{-1}(\partial D_{\rho_{l}} \times (-\sigma_{l}, \sigma_{l})) \\ &< \mathcal{H}^{n-k}(G_{c}(x)) k^{-1} w_{k-1} \rho_{l}^{k} +4 \mathcal{H}^{n-k}(G_{c}(x)) w_{k-1} \rho_{l}^{k-1} \sigma_{l} \label{cylinderareabd2}
\end{align}
and (\ref{replaceareabd}), (\ref{isopineqsmallpiece}), and (\ref{cylinderareabd2}) deduces
\begin{equation*}
    \mathcal{H}^{n}(\pi^{-1}(P^{i}_{l})) < \mathcal{H}^{n-k}(G_{c}(x)) k^{-1} w_{k-1} \rho_{l}^{k} +4 \mathcal{H}^{n-k}(G_{c}(x)) w_{k-1} \rho_{l}^{k-1} \sigma_{l} + \epsilon_{l,i} + 2c_{n}' \eta^{\frac{n}{n-1}}.
\end{equation*}
By taking small $\eta>0$, for sufficiently large $l$, we have
\begin{equation} \label{sheetareabd}
    \mathcal{H}^{n}(\pi^{-1}(P^{i}_{l})) < \mathcal{H}^{n-k}(G_{c}(x)) k^{-1} w_{k-1} \rho_{l}^{k} +5 \mathcal{H}^{n-k}(G_{c}(x)) w_{k-1} \rho_{l}^{k-1} \sigma_{l}
\end{equation}
for $i = R_{l}^{2}+1, ..., R_{l}^{3}$.

Note that $\sum_{i=R_{l}^{2}+1}^{R_{l}^{3}} \mathcal{H}^{n}(\pi^{-1}(P_{l}^{i} \cap K_{\frac{1}{2},1}))$ is bounded then the cardinality of sheets $R_{l}^{3}- R_{l}^{2}+1$ is bounded independent of $l$. Now we rescale and take a limit of each sheet as follows by the area formula and there exists positive number $m$ such that:
\begin{enumerate}
    \item $\rho_{l} \rightarrow \rho_{0} \in [3/4,1]$ as a subsequential limit;
    \item for $i=1, ..., m$, $\mu_{\rho_{l}^{-1}}(\pi^{-1}(P_{l}^{i}))$ converges to a varifold $V_{i}$ satisfying for $\rho \in [0,1]$
    \begin{align} \label{sheetsboundbegin}
       \mathcal{H}^{n-k} (G_{c}(x)) k^{-1} w_{k-1} \rho^{k} &\le |V_{i}\mres\pi^{-1}(K_{\rho,1})| \\ \label{sheetsboundmid}
       |V_{i}\mres\pi^{-1}(K_{1,1})| &< \mathcal{H}^{n-k} (G_{c}(x)) (k^{-1} w_{k-1}  + 10  w_{k-1} \sigma_{0})
    \end{align}
     and
    \begin{equation} \label{sheetsboundend}
        spt(V_{i}) \subset \pi^{-1}(K_{1,\sigma_{0}}),
    \end{equation}
    and
    \begin{equation*}
        (\mu_{r \rho_{0}^{-1}} \circ \phi_{x}^{-1})_{\sharp} V \mres \pi^{-1}(K_{\frac{1}{2},1}) = \sum_{i=1}^{m} V_{i} \mres \pi^{-1}(K_{\frac{1}{2},1}).
    \end{equation*}
\end{enumerate}
By taking arbitrary small $\sigma_{0}$ and considering (\ref{sheetsboundbegin})-(\ref{sheetsboundend}), the slice theorem (Theorem \ref{slicethm}), and also the monotonicity of $G_{c}$-stationary varifold (in the slice sense), we have an integer density $\theta(V,x_{0}) =m$ and we proved that $V$ is a stationary integral varifold. Also by the area bound (\ref{sheetsboundbegin}), (\ref{sheetsboundmid}) and Allard regularity theorem on slices, we obtain
\begin{equation*}
    V_{i} \mres \pi^{-1}(K_{\frac{1}{2},1}) = M_{i}
\end{equation*}
where $M_{i}$ is a $G_{c}$-equivariant minimal hypersurface in $G_{c}$-tubular neighborhood of $x$ and whose slices by $G_{c}(x)$ are graphical. From the construction of $V_{i}$, $M_{i}$ is on one-side of $M_{j}$ for $i,j =1, ..., n$. By maximum principle since $\theta(V,x_{0})=m$, $(\mu_{r \rho_{0}^{-1}} \circ \phi_{x}^{-1})_{\sharp} V \mres \pi^{-1}(K_{\frac{1}{2},1}) = m M$ for some connected $G_{c}$-minimal hypersurface $M$ containing $x$.
\end{proof}
\end{thm}
\begin{rmk} In our application later, $H'= \pi(H)$ is a smooth half-plane, but our argument contains the case that $H'$ is contained in a wedge region and topological half-plane.
\end{rmk}
Now we prove that the tangent cone at any point is a hyperplane. Here we apply the interior regularity theory on principal orbits and prove the regularity on singular orbits with Proposition \ref{prop:nointeriorsingularorbit}.

\begin{thm} \label{secondcontireg} Suppose $\{N_{l}' \} \in \mathcal{M}(0)$ is a minimizing sequence and the pullback sequence of $n$-varifolds $\{V_{l} \}_{l \in \mathbb{N}}$ where $V_{l} = \pi^{\sharp} N_{l}'$ converges to the varifold limit $V$ in $V_{n}(M)$. Then at each point $x \in spt(V)$, there is a tangent cone $C$ of $V$ having the form $C = mH$ where $H$ is a hyperplane in $\mathbb{R}^{n+1}$ where $m$ is a positive integer. Moreover, if $x$ is on a top strata of singular orbits, the dihedral angle between $\pi(H)$ and $\partial M'$ is $\pi/2$.
\begin{proof}
For tangent cones on principal orbits, we apply Theorem \ref{intregular} and obtain the regularity on $M'$ and the tangent cone is a hyperplane if we pull back to the preimage. Therefore without loss of generality, we only consider the case that $x$ is on singular orbits.  

Take a tangent cone with $C = \lim_{l \rightarrow \infty} (\mu_{r_{l}} \circ \phi_{x}^{-1})_{\sharp} V$ where $r_{l} \rightarrow \infty$ as $l$ goes to infinity, where we can decompose $C = C_{S} \times T_{x}G_{c}(x)$. We denote $C_{G_{c},x}$ to be a $G_{c}$-invariant $C_{S}$-bundle whose base is $G_{c}(x)$ in $\mathbb{R}^{k+1} \times G_{c}(x)$ which satisfies $T_{x} C_{G_{c},x} = C$. Then we can find a sequence $\{ D_{l}' \} \in \mathcal{M}(0)$ and $D_{l} := \pi^{-1}(D_{l}')$ such that
\begin{equation} \label{tangentconverg}
    D_{l} = (\mu_{r_{l}} \circ \phi_{x}^{-1})_{\sharp} N_{l} \text{ and } D_{l} \rightharpoonup C_{G_{c},x}
\end{equation}
as $l \rightarrow \infty$ where $N_{l} = \pi^{-1}(N_{l}')$ and
\begin{equation} \label{coneareabd}
    \mathcal{H}^{n}(D_{l}) \le \mathcal{H}^{n}(\pi^{-1}(D')) + \epsilon_{l}'
\end{equation}
for any $D' \in \mathcal{M}(0)$ with $\partial D' \setminus \partial M' = \partial D_{l}' \setminus \partial M'$ where $\epsilon_{l}' \rightarrow 0$ as $l \rightarrow \infty$. By (\ref{coneareabd}), $C_{G_{c},x}$ is an area minimizer and $C$ is a $G_{c}$-stable cone.

We first consider the principal orbit part on $C$. Note that the principal or exceptional orbit as a preimage by $\pi$ of the interior part of the blow-up on $\partial M'$. Fix a point $x_{1} \in C \subset \mathbb{R}^{n+1}$ on a principal orbit in $\mathbb{R}^{n+1}$. There exists a $G_{c}$-neighborhood of $B^{G_{c}}_{\rho}(x_{1})$ such that all points are on the principal orbit. Then $\pi(B^{G_{c}}_{\rho}(x_{1})) = B_{\rho}(\pi(x_{1})) \in C'_{x'} M'$. For sufficiently large $r_{l}$, $B^{G_{c}}_{r_{l}^{-1} \rho}(\exp_{x}(r_{l}^{-1}x_{1})) \subset M$ are on principal or exceptional orbits. From (\ref{tangentconverg}) as $l$ goes to infinity, we have the following:
\begin{equation} \label{localconeconv}
    D_{l} \mres (\mu_{r_{l}} \circ \phi_{x}^{-1})(B^{G_{c}}_{r_{l}^{-1} \rho}(\exp_{x}(r_{l}^{-1}x_{1}))) \rightarrow C_{G_{c},x} \mres B^{G_{c}}_{\rho}(x_{1}).
\end{equation}
We consider the corresponding image by $\pi$. We first take the weighted pushforward metric $g'_{tan}$ on $\pi(T_{x}M)$ from $(T_{x}M,g_{Eucl})$ induced by its volume function of each orbit with the formula in (\ref{volumefunction}). Let us denote the corresponding varifold by $V_{l, \rho}'$ i.e. $V_{l, \rho}' :=(\mathcal{D} \circ \pi)_{\sharp} ( D_{l} \mres (\mu_{r_{l}} \circ \phi_{x}^{-1})(B^{G_{c}}_{r_{l}^{-1} \rho}(\exp_{x}(r_{l}^{-1}x_{1}))))$. Since $B^{G_{c}}_{r_{l}^{-1} \rho}(\exp_{x}(r_{l}^{-1}x_{1}))$ is on principal orbits, and we have $spt(V_{l,\rho}') \in \mathcal{M}(0)$ and area constraint (\ref{coneareabd}), we can apply Theorem 4 in \cite{almgren1979existence} and obtain $C'_{\rho}$ as a varifold limit of $V_{l,\rho}'$ so is a smooth minimal surface. By taking $\pi^{\sharp}(C'_{\rho})$ and applying Theorem \ref{slicethm}, we obtain that $C_{G_{c},x} \mres B^{G_{c}}_{\rho}(x_{1})$ to be a smooth minimal hypersurface in $\mathbb{R}^{n+1}$.

Now we analyze singular orbits on $C$. Note that $\pi(\partial B^{G_{c}}_{\rho}(x))$ is connected and has no nontrivial genus since it is an image of a sphere bundle over the projection map $\pi$, so is a topological disk whose boundary is on $\partial M'$. By this topological fact and the regularity on principal orbits, $\pi(C_{G_{c},x} \cap \partial B^{G_{c}}_{\rho}(x))$ consists of simple closed curves and simple curve segments whose boundary is on $\partial \pi(T_{x}M) \cap \pi(\partial B^{G_{c}}_{\rho}(x))$ which intersects $\partial \pi(T_{x}M) \cap \pi(\partial B^{G_{c}}_{\rho}(x))$ transversally.

Let us assume that there exists a simple closed curve $\gamma$ in $\pi(C_{G_{c},x} \cap \partial B^{G_{c}}_{\rho}(x))$. Assume that $\gamma$ intersects $\partial \pi(T_{x}M) \cap \pi(\partial B^{G_{c}}_{\rho}(x))$ tangentially and take $y \in \gamma\,\cap \partial \pi(T_{x}M) \cap \pi(\partial B^{G_{c}}_{\rho}(x))$. Then a density of tangent cone at $y$ collapses to $0$, so this cannot happen and $\gamma$ lies in $\operatorname{Int}(M')$. By maximum principle, there is at most one simple closed curve, and there does not exist any simple curve segments by Jordan curve theorem. Then $C$ is a stable cone and may have a singularity at $x$ whose codimension is at least $4$. Applying Simons' classification again, we obtain that the tangent cone at $x$ is a tangent plane. This gives that it is enough to consider the case when $\partial \pi(T_{x}M) \cap \pi(C_{G_{c},x} \cap B^{G_{c}}_{\rho}(x))$ are finite number of rays containing $x$ by its cone property and monotonicity formula by the minimality. 

We take $y \in C \setminus G_{c}(x)$ on singular orbits. Note that $\pi(y)$ is on one of the rays of $\partial \pi(T_{x}M) \cap \pi(C_{G_{c},x} \cap B^{G_{c}}_{\rho}(x))$. Consider the tangent cone of $C$ at $y$ and denote this by $C_{y}$. We denote the slice at $y$ by $S_{y}$ which is an orthogonal component of $T_{y}G_{c}(y)$ in $T_{y}M$, and $S_{y/ \{ry \}}$ as its $(k-1)$-dimensional equatorial disk which is perpendicular to the ray $\operatorname{span}\{ y \}$. $C_{y}$ splits to $C_{y} = C_{S_{y}/ \{ ry\}} \times T_{y}G_{c}(y) \times T_{y} (\{ ry \})$, where $C_{S_{y}/ \{ ry\}}$ is a restriction of $C_{y}$ on $S_{y/ \{ry \}}$. Denote $C_{G_{c},y}$ to be a $G_{c}$-invariant $C_{S_{y}}$-bundle whose base is $G_{c}(y)$ which satisfies $T_{y} C_{G_{c},y} = C_{y}$. We first prove that $\pi(C_{y})$ is a single sheet of half-plane containing the ray $\pi(\{ry\})$ (multiplicity may occur). 

By the regularity on principal orbits, for sufficiently small $\rho'>0$, $\pi(C_{G_{c},y} \cap \partial B^{G_{c}}_{\rho'}(y))$ consists of simple arcs whose boundary endpoints are two points $ \pi(\{ry\} \cap \partial B^{G_{c}}_{\rho'}(y))$ that intersect its boundary transversally.

Assume that there are two segments $\gamma_{1}$ and $\gamma_{2}$ that meet at two endpoints $ \pi(\{ry\} \cap \partial B^{G_{c}}_{\rho'}(y))$. Since two half-plane sheets containing $\gamma_{1}$ and $\gamma_{2}$ are quotients of a minimal cone in $S_{y}$, we have $C_{S_{y/ \{ ry\}}} \cap \pi^{-1}(\gamma_{1})$ and $C_{S_{y/ \{ ry\}}} \cap \pi^{-1}(\gamma_{2})$ as links which are disjoint smooth minimal hypersurfaces of $S_{y/ \{ ry\}} \cap \partial B^{G_{c}}_{\rho'}(y)$. Notice that the regularities come from the regularity on principal orbits. Since $S_{y/ \{ ry\}} \cap \partial B^{G_{c}}_{\rho'}(y)$ has strictly positive Ricci curvature for small $\rho$, any minimal hypersurfaces on $S_{y/ \{ ry\}} \cap \partial B^{G_{c}}_{\rho'}(y)$ intersect each other by Frankel property. This contradicts $\pi^{-1}(\gamma_{1}) \cap S_{y}$ and $\pi^{-1}(\gamma_{2}) \cap S_{y}$ being disjoint smooth minimal hypersurfaces and $\pi(C_{G_{c},y} \cap \partial B^{G_{c}}_{\rho'}(y))$ is a single simple curve and we denote this by $\gamma$. Hence, $\pi(C_{y})$ is a single sheet of topological half-plane.

Now we prove the regularity at $y \in C \setminus G_{c}(x)$. We consider the tangent cone $C_{y}$ and prove that it is a hyperplane and it suffices to prove that $C_{S_{y /\{ ry \}}}$ is a hyperplane in $S_{y/\{ry \}}$. By Corollary (Lower semicontinuity of orbit types) in \cite{hsiang1971minimal}, the link of $C_{S_{y/ \{ ry\}}}$ is a $(G_{c})_{y}/(G_{c})_{z} \subset S^{k-1}$ where $z \in C_{S_{y/ \{ ry\}}} \cap \pi^{-1}(\gamma \setminus \partial\pi (T_{y}M))$ so that is a smooth minimal hypersurface of $S_{y/ \{ ry\}}$ by the regularity on principal orbits. Moreover, $C_{S_{y/ \{ ry\}}}$ is a two-sided $G_{c}$-stable cone so is a two-sided stable cone by Lemma \ref{equivstablity}. By the classification of stable minimal hypercones of Simons \cite{simons1968minimal}, we obtain $C_{S_{y/ \{ ry\}}}$ is a smooth hyperplane.

The final step is to prove the regularity on $G_{c}(x)$. We proved the regularity on $C \setminus G_{c}(x)$ in the previous steps. We split into $C_{x} = C_{S_{x}} \times T_{x}G_{c}(x)$ and prove that $C_{S_{x}}$ is a hyperplane. By applying Lemma \ref{equivstablity} again, we have $C_{S_{x}}$ is a stable minimal hypercone and again apply Simons' classification \cite{simons1968minimal}, and obtain the regularity at $x$ and this extends to the regularity on $G_{c}(x)$.

Suppose now $x$ is on a top strata of singular orbits, then $T_{x'}\partial M'$ is a smooth plane. By the fact that the tangent cone is a hyperplane and previous arguments, $\pi(H)$ is a smooth half-plane. Denote the intersection line segment between $\partial \pi(T_{x}M)$ and $\pi(H)$ by $\gamma'$. Suppose $H_{0}'$ is a half-plane in $ \pi(T_{x} M)$ which contains $\gamma'$ and orthogonal to $\partial \pi(T_{x}M)$. By the Slice Theorem \ref{slicethm}, $\pi^{-1} (H_{0}')$ is a product of a hyperplane in $S_{x}$ and the orbit $G_{c} \cdot x$. Hence, by the uniqueness of such hyperplane, we obtain $H_{0}' = \pi(H)$ and we deduce that the dihedral angle is a right angle.
\end{proof}
\end{thm}
We obtain the following regularity result of isotopy minimizing problem of genus $0$ surfaces on $M/G$.
\begin{thm} \label{regularityminiconti}Suppose $\{N_{l}' \} \in \mathcal{M}(0)$ is a minimizing sequence and the pullback sequence of $n$-varifolds $\{V_{l} \}_{l \in \mathbb{N}}$ where $\pi_{\sharp}V_{l} = N_{l}'$ converges to the varifold limit $V$ in $V_{n}(M)$. Then for every point $x \in spt(V)$, there is a $\rho >0$ such that $V \mres B^{G}_{\rho}(x)$ is an integer multiple of smooth $G_{c}$-equivariant minimal hypersurface containing $x$.
\begin{proof} The proof directly follows from Theorem \ref{firstcontireg} and \ref{secondcontireg}.
\end{proof}
\end{thm}

\subsection{Main Regularity} We now prove the main regularity theorem in general cases.
\begin{proof} [Proof of Theorem \ref{regularitycontiLie}] By $\gamma$-reduction (Proposition \ref{gammareduction}), it suffices to consider when $\{(\Sigma')^{i}\}$ is strongly $(\gamma, U')$-irreducible.

Take $x \in spt(V)$ and choose small $\rho_{0}>0$ that $B_{\rho_{0}}^{G_{c}}(x)$ to be a $G_{c}$-neighborhood and $\pi(B_{\rho_{0}}^{G}(x))$ to be a topological $3$-ball. By the coarea formula, we have
\begin{equation*}
    \int^{\rho}_{\rho-\sigma} \mathcal{H}^{n-1}( \Sigma^{l} \cap \partial B^{G_{c}}_{s}(x)) ds \le \mathcal{H}^{n}( \Sigma^{l} \cap (B^{G_{c}}_{\rho}(x) \setminus B^{G_{c}}_{\rho-\sigma}(x)))
\end{equation*}
for almost every $\rho \in (0,\rho_{0})$ and $\sigma \in (0,\rho)$. By the monotonicity formula, the slice theorem (Theorem \ref{slicethm}), the area formula, and by taking $\sigma = \rho/2$, we have
\begin{equation*}
    \int^{\rho}_{\rho/2} \mathcal{H}^{n-1}(\Sigma^{l} \cap \partial B^{G_{c}}_{s}(x)) ds \le c \mathcal{H}^{n-k} (G(x_{0})) \rho^{k},
\end{equation*}
where $c$ depends on $|V \mres B^{G_{c}}_{\rho_{0}}(x)|$. Together with Sard's theorem, we obtain a sequence $\rho_{l} \in (3\rho/4, \rho)$ which $\Sigma^{l}$ intersects $\partial B^{G_{c}}_{\rho_{l}}(x)$ transversally and satisfy
\begin{equation} \label{areaboundrepl}
    \mathcal{H}^{n-1}( \Sigma^{l} \cap \partial B^{G_{c}}_{\rho_{l}}(x)) \le 2c \mathcal{H}^{n-k} (G_{c}(x)) \rho^{k-1} \le c \mathcal{H}^{n-k} (G_{c}(x)) \eta \rho_{0}^{k-1}
\end{equation}
for sufficiently large $l$ where we can take arbitrary small $\eta>0$ by taking small $\rho \in (0,\rho_{0})$. Note that the image $(\Sigma^{l})'$ transverses $\pi(\partial B^{G_{c}}_{\rho_{l}}(x))$.

Hence, we can apply Theorem \ref{genuszero} with $U_{l}' = \pi(B^{G}_{\rho_{l}}(x_{0}))$ to $\{ \Sigma_{l}' \}$ by the isoperimetric inequality on slices and take $A_{l}'$ to be a slightly smaller ball. There exist pairwise disjoint connected closed genus zero surfaces $D_{1}', ..., D_{p_{l}}'$ with $D_{i}' \subset (\Sigma^{l})' \cap U' \cap M'$, $\partial D_{i}' \setminus \partial M' \subset \partial A_{l}'$, and $(\cup_{i=1}^{p_{l}} D_{i}' ) \cap A_{l}' \cap M' = (\Sigma^{l})' \cap A_{l}' \cap M'$ for $i=1, ..., p_{l}$. Moreover, from (\ref{areaboundrepl}) and the isoperimetric inequality on slices, $D_{i}:=\pi^{-1}(D_{i}')$ satisfies
\begin{equation*}
    \sum^{p_{l}}_{i=1} \mathcal{H}^{n}(D_{i}) \le  c' \mathcal{H}^{n-k} (G_{c}(x)) \rho^{k} \le c' \mathcal{H}^{n-k} (G_{c}(x)) \eta \rho_{0}^{k}
\end{equation*}
for some constant $c'$ which only depends on $c$.
Now we can apply the replacement lemma (Lemma \ref{replacement}) and obtain pairwise disjoint connected closed genus zero surfaces $\tilde{D}_{1}', ..., \tilde{D}_{p_{l}}'$ such that
\begin{align*}
\partial \tilde{D}_{i}' \cap \operatorname{Int} M' = \partial D_{i}' \cap \operatorname{Int} M' \text{ and } \tilde{D}_{i}' \subset A' \setminus \partial A'
\end{align*}
and $\tilde{D}_{i}:= \pi^{-1}(\tilde{D}_{i}')$ satisfies
\begin{equation*}
    \sum^{p_{l}}_{i=1} \mathcal{H}^{n}(\tilde D_{i}) \le \sum^{p_{l}}_{i=1} \mathcal{H}^{n}(D_{i}).
\end{equation*}
We take $\rho_{\infty} = \lim_{l \rightarrow \infty} \rho_{l}$, and for any arbitrary small $\sigma>0$ then for sufficiently large $l$, we can take  $(\tilde{\Sigma}^{l})'$ to satisfy
\begin{enumerate}
    \item $(\tilde{\Sigma}^{l})' \cap \pi(B_{\rho_{\infty}-\sigma}^{G_{c}}(x)) = \cup_{i=1}^{p_{l}} \tilde D_{i}' \cap \pi(B_{\rho_{\infty}-\sigma}^{G_{c}}(x))$;
    \item $\partial_{rel} (\tilde{\Sigma}^{l})' \cap \pi(\overline{B}_{\rho_{\infty}-\sigma}^{G_{c}}(x)) = \emptyset$;
\end{enumerate}
and we can connect surfaces between $\partial \pi(B_{\rho_{\infty}-\sigma}^{G_{c}}(x))$ and $\partial \pi(B_{\rho_{\infty}}^{G}(x))$ by adding small surfaces and obtain a minimizing sequence of genus zero surfaces on the orbit space. $\{ (\tilde{\Sigma}^{l})' \}$ satisfies the conditions of Theorem \ref{regularityminiconti} and obtain the desired conclusion by applying Theorem \ref{regularityminiconti}.
\end{proof}
\section{Discrete isometry actions}
In this section, we discuss the regularity of isotopy minimization problem when the isometry group $G$ is a product of the connected Lie group $G_{c}$ and the finite isometry group $G_{f}$. Recall that we proved the regularity of $G_{c}$-equivariant isotopy minimization in the last section. We assumed that $\tilde{\mathcal{S}} \cap \partial M'$ is contained in the smooth point of $\partial M'$ in the assumption (\ref{singularlocusstar}). As in the discrete isometry equivariant theory of Ketover (\cite{ketover2016equivariant}, \cite{ketover2016free}), we focus on the regularity on singular orbits by cyclic $\mathbb{Z}_{n}$-actions. The structure of this section will consist of necessary changes in each step in Section $3$.
\subsection{Regularity near the interior $G_{f}$-singular locus}
After recalling that we proved the regularity on $\pi^{-1}(M' \setminus \tilde{\mathcal{S}})$ in the previous section, we take a point $x \in \pi^{-1}(\tilde{\mathcal{S}} \cap \operatorname{Int}(M'))$ first. Then it suffices to consider when $x$ is on exceptional orbits in $M$ with the isometry group $G$ by Proposition \ref{prop:nointeriorsingularorbit}. We can directly apply Theorem \ref{intregular} for these exceptional orbits and obtain the following theorem. The stability from $(G_{c} \times G_{f})$-stablity follows from Lemma \ref{equivstablity} and Proposition 4.6 in \cite{ketover2016equivariant}.
\begin{thm} \label{discreteexcep} Suppose $G_{f}$ is an orientation preserving isometry and $\{ \Sigma^{\prime}_{l} \}$ is a minimizing sequence of surfaces for the $G_{f}$-equivariant isotopy minimization problem in the $3$-dimensional open ball $U' \Subset M'$ on the area of $\{ \Sigma_{l} \} : = \{ \pi^{-1} (\Sigma^{\prime}_{l}) \}$ in $U:= \pi^{-1}(U')$ which converges to a varifold $V$. Then there exists a $(G_{c} \times G_{f})$-stable embedded minimal hypersurface $\Gamma$ with $\overline{\Gamma} \setminus \Gamma \subset \partial U$ and $V= \Gamma$ in $U$. Moreover, if no finite isotropy group induced by $G_{f}$ is $\mathbb{Z}_{2}$ in $U'$ or $\pi(spt(V))$ does not contain a segment of $\tilde{\mathcal{S}}$, $\Gamma$ is a stable minimal hypersurface.
\end{thm}
\subsection{Regularity on singular orbits by $\mathbb{Z}_{n}$-actions} We consider the regularity on points on singular orbits. That being said, we focus on the regularity on points in $\pi^{-1}( \tilde{\mathcal{S}} \cap  \partial M')$. Note that we assumed that the intersection points of $\tilde{\mathcal{S}}$ and $\partial M'$ are on smooth points i.e. $T_{x'}\partial M'$ is a smooth plane on $x' \in \tilde{\mathcal{S}} \cap \partial M'$.
\begin{thm} \label{znaction} Suppose $\{ \Sigma^{\prime}_{l} \}$ is a minimizing sequence of genus $0$ surfaces for the $\mathbb{Z}_{n}$-equivariant isotopy minimization problem in a $3$-dimensional open half-ball $U'= B_{r}(x') \cap M'$ where $x' \in \partial M'$, on the $n$-dimensional area of $\{ \Sigma_{l} \} : = \{ \pi^{-1} (\Sigma^{\prime}_{l}) \}$ in $U:= \pi^{-1}(U')$, which converges to a varifold $V$. Also assume that the relative boundary curves of $\partial \Sigma^{\prime}_{l} \setminus \partial M' = \cup_{j=1}^{m} \gamma'_{j}$ for all $l$ where each $\gamma'_{j}$ is a Jordan curve on $\partial U'$. Suppose also that $\mathbb{Z}_{n}$ acts freely on $\cup_{j=1}^{m} \gamma_{j}'$. Then the points in the preimage of boundary intersection of singular locus $\pi^{-1}(U' \cap \partial M' \cap \tilde{\mathcal{S}})$ is not in $spt(V)$. Hence, $V$ is a smooth embedded $(G_{c} \times \mathbb{Z}_{n})$-equivariant minimal hypersurface and the relative boundary of $\Sigma'_{\infty} = \pi(\operatorname{spt}V)$ is $\cup_{j=1}^{m} \gamma'_{j}$ and genus of $\Sigma'_{\infty}$ is $0$. The multiplicity of convergence is $1$.
\begin{proof}
Suppose that there exists $x \in \pi^{-1}(U' \cap \partial M' \cap \tilde{\mathcal{S}}) \cap spt(V)$. Notice that $x'= \pi(x)$ is an isolated point in the discrete set $U' \cap \partial M' \cap \tilde{\mathcal{S}}$. Take small $\rho>0$ to be $B_{\rho}(x') \cap \partial M' \cap \tilde{\mathcal{S}} = x'$ in $\overline{g}$ metric. By Theorem \ref{regularitycontiLie} and Theorem \ref{discreteexcep}, we have a regularity of $V$ away from $G(x)$. 

By \cite[Lemma 3.4 and Proposition 4.14]{ketover2016equivariant}, we have that $\pi(\operatorname{spt}(V))$ cannot contain any segment of $\tilde{S}$, since $\mathbb{Z}_{n}$ acts freely on the prescribed boundary curves. We can choose $(G_{c} \times \mathbb{Z}_{n})$-invariant normal. We now apply the stability proof of Lemma 7 in \cite[Section 2]{wang2022min} to $\operatorname{spt}(V) \setminus \pi^{-1}(x')$ and we obtain the stability of $V$ in $\operatorname{spt}(V) \setminus \pi^{-1}(x')$ from $(G_{c} \times \mathbb{Z}_{n})$-stability. 

Since $x$ is in a singular orbit which has dimension at most $n-3$, we can apply Schoen-Simon regularity for stable minimal hypersurfaces \cite{schoen1981regularity} of stable embedded minimal hypersurfaces. We obtain that $V$ is smooth at $x$. Since the quotient of a smooth invariant hypersurface intersects $\partial M'$ orthogonally, $\pi(spt(V))$ contains $\tilde{S}$ and this gives contradiction.

Since the points of the preimage of singular locus segment $\pi^{-1}(U' \cap \partial M' \cap \tilde{\mathcal{S}})$ are not on the support of the limit varifold $V$, the regularity, genus $0$ property, and multiplicity $1$ property is obtained directly from Proposition 4.14 of \cite{ketover2016equivariant} and Theorem \ref{regularitycontiLie}.
\end{proof}
\end{thm}
Now we prove the case where $\mathbb{Z}_{n}$ acts non-freely. By Lemma 3.4 in \cite{ketover2016equivariant}, we only need to consider the case that $\mathbb{Z}_{2}$ acts on $\cup_{j=1}^{m} \gamma_{j}'$.
\begin{thm} \label{z2action}
Suppose $\mathbb{Z}_{2}$ acts on a topological 3-ball $U' =  B_{r}(x') \cap M'$ for $x' \in \tilde{\mathcal{S}} \cap \partial M'$ which has a rotation isometry with angle $\pi$ about the line $\tilde{\mathcal{S}}$. Suppose $\{ \Sigma^{\prime}_{l} \}$ is a minimizing sequence of genus $0$ surfaces for the $\mathbb{Z}_{2}$-equivariant isotopy minimization problem in $U'$, on the area of $\{ \Sigma_{l} \} : = \{ \pi^{-1} (\Sigma^{\prime}_{l}) \}$ in $U:= \pi^{-1}(U')$, which converges to a varifold $V$. Also assume that the relative boundary curves of $\partial \Sigma^{\prime}_{l} \setminus \partial M' = \cup_{j=1}^{m} \gamma'_{j}$ for all $l$ where each $\gamma'_{j}$ is a Jordan curve or a Jordan arc on $\partial_{rel} U'$ so that we put $\mathbb{Z}_{2}$ acts freely on the curves $\{ \gamma'_{j} \}_{j=2}^{m}$ but non-freely on $\gamma'_{1}$. Then $V$ is a smooth embedded $(G_{c} \times \mathbb{Z}_{2})$-equivariant minimal hypersurface and the relative boundary of $\Sigma'_{\infty} = \pi(\operatorname{spt}V)$ is $\cup_{j=1}^{m} \gamma'_{j}$ and genus of $spt(\pi_{\sharp}(V))$ is $0$. The multiplicity of convergence is $1$.
\begin{proof}
By Theorem \ref{regularitycontiLie}, Theorem \ref{discreteexcep}, and Proposition 7.3 in \cite{ketover2016free}, it suffices to prove the regularity on $spt(V) \cap \pi^{-1}(\tilde{\mathcal{S}} \cap \partial M')$. Take $x' \in \pi(spt(V)) \cap \tilde{\mathcal{S}} \cap \partial M'$ and pick one point $x \in \pi^{-1}(x')$. We take a tangent cone $C$ of $V$ at $x$. Then by Theorem \ref{slicethm}, $T_{x}C = C_{\mathcal{S}} \times T_{x} (G(x))_{x}$.

We consider a pushforward cone on half-space $C':= \pi_{\sharp}C \subset T_{x'} M'$. Since we proved regularity on points except for $x'$ and $x'$ is on the top strata, $C'$ has a half-cone structure. Note that $C'$ contains the dilation of $\tilde{\mathcal{S}}$ and we consider the link $C' \cap S^{2}(1)$. Then since we have the regularity on $C' \setminus x'$ and we know $C'$ intersects $T_{x'} \partial M'$ orthogonally on $C' \setminus x'$ by the orthogonality statement of Theorem \ref{secondcontireg}, we have $C' \cap S^{2}(1)$ is a union of embedded curve segments which intersects boundary circle $S^{2}(1) \cap T_{x'} \partial M'$ orthogonally. By the same argument with the Frankel property argument in the proof of Theorem \ref{secondcontireg}, we obtain that $C' \cap S^{2}(1)$ has a single embedded curve with a $\mathbb{Z}_{2}$-symmetry with angle $\pi$ rotation.

Notice that the preimage of a half-plane cone whose link is a half-circle is a hyperplane in $T_{x}M$ which is minimal since $x$ is on a top strata. Denote $C' \cap S^{2}(1)= \gamma$ and $C' \cap S^{2}(1) \cap T_{x'} \partial M'= \{ p_{1}, p_{2} \}$. Since $\gamma$ intersects $S^{2}(1) \cap T_{x'} \partial M'$ orthogonally and the lift of an orthogonal half-plane is the unique invariant minimal hyperplane, a half-circle is the only configuration which makes $C$ to be a minimal cone. Hence, the tangent cone $C$ at $x$ is a hyperplane and we proved the regularity at $x$. The remaining statements follow from Proposition 7.3 in \cite{ketover2016free}.
\end{proof} 
\end{thm}
Finally, we obtain the following regularity of $G$-equivariant isotopy minimization problem.
\begin{thm}
\label{regularitydiscLie} Let $U'$ be an open set satisfying (1)-(3) in Section 3 and have a $\mathbb{Z}_{n}$-symmetry. Suppose $\{(\Sigma')^{i}\}$ is a minimizing sequence for the $\mathbb{Z}_{n}$-equivariant minimization problem $(\Sigma', \mathcal{I}(U'))$ so that $(\Sigma')^{i}$ intersects $\partial M'$ transversally for each $i$, and $\Sigma^{i} : = \pi^{-1}((\Sigma')^{i})$ converges to a limit (in a varifold sense) to $V \in IV_{n}(M)$. Then the following holds:
\begin{enumerate}
    \item $V = \Gamma$, where $\Gamma$ is a compact smooth embedded $(G_{c} \times \mathbb{Z}_{n})$-equivariant minimal hypersurface in $\pi^{-1}(U') \cap M$ whose boundary is contained in $\pi^{-1}(\partial_{rel} U')$;
    \item $V \mres (M \setminus U) =  \pi^{-1}(\Sigma') \mres (M \setminus U)$;
    \item $\Gamma$ is $(G_{c} \times \mathbb{Z}_{n})$-stable with respect to the pullback isotopy of $\mathcal{I}(U')$. ($(G_{c} \times \mathbb{Z}_{n})$-stability)
\end{enumerate}
\begin{proof} We apply Proposition \ref{gammareduction} with a $G_{f}$-symmetry, and we can apply Theorem \ref{genuszero} and Lemma \ref{replacement} with a $G_{f}$-symmetry as in the proof of Theorem \ref{regularitycontiLie}. By Lemma 3.3 in \cite{ketover2016equivariant}, we obtain the regularity statement by applying the regularity theorems (Theorem \ref{discreteexcep}, Theorem \ref{znaction}, and Theorem \ref{z2action}) in place of Theorem \ref{regularityminiconti} in the proof of Theorem \ref{regularitycontiLie}.
\end{proof}
\end{thm}
\section{Min-max theory}
In this section, we formulate the equivariant min-max setting on orbit spaces. While our min-max set up works for higher parameter min-max theories as in other (equivariant) min-max settings, we develop the theory based on the one-parameter min-max for simplicity. A $G_{f}$-sweepout on $M'$ is a $1$-parameter family of closed sets $\{ \Sigma_{t} \}_{t \in [0,1]}$ if the family satisfies the following: 
\begin{enumerate}
\item $\Sigma_{t}$ varies continuously in the Hausdorff topology (we allow degeneration along $\partial M'$, in the sense of permitting convergence to $\Sigma_t\cup K_t$, where $K_t\subset\partial M'$ is compact so that $\mathcal H^n(\pi^{-1}(K_t))=0$);
\item $\Sigma_{t}$ is a $G_{f}$-equivariant properly embedded surface of $M'$ for $0<t<1$ which varies smoothly; 
\item $\Sigma_{0}$ and $\Sigma_{1}$ are unions of smooth properly embedded surfaces and finitely many $1$-dimensional arcs.
\end{enumerate}
Let us call an isotopy $\{ \psi_{t} \}_{t \in [0,1]}$ to be a $G_{f}$-isotopy if each $\psi_{t}$ is a diffeomorphism which commutes with the $G_{f}$-action, and denote the subset consisting of $G_{f}$-equivariant isotopies of $\mathcal{I}(U')$ by $\mathcal{I}_{G_{f}}(U')$. Given a $G_{f}$-sweepout $\{ \Lambda_{t} \}_{t \in [0,1]}$ of $M'$, we define its $G_{f}$-equivariant saturation $\Pi$ by
\begin{equation*}
    \Pi : = \{ \psi_{t}(\Lambda_{t}) \, | \, \{ \psi_{t} \}_{t \in [0,1]} \text{ is an } G_{f}\text{-isotopy with } \psi_{t} = \text{id for } t=0, 1 \}.
\end{equation*}
Now we define the min-max width as follows:
\begin{equation} \label{widthdef}
    W^{G_{f}}_{\Pi} = \inf_{\{ \Lambda_{t} \} \in \Pi} \sup_{t \in [0,1]} \mathcal{H}^{n}(\pi^{-1}(\Lambda_{t}))
\end{equation}
with the nontriviality of the width:
\[
W^{G_{f}}_{\Pi} > \max \{ \mathcal{H}^{n}(\pi^{-1} \Lambda_{0}), \mathcal{H}^{n}(\pi^{-1} \Lambda_{1})\}.
\]
We call a sequence of $G_{f}$-sweepouts $\{\Lambda^{i} \}$ to be a \emph{minimizing sequence} if $\lim_{j \rightarrow \infty} \sup_{t \in [0,1]} \mathcal{H}^{n} (\pi^{-1}(\Lambda^{j}_{t})) = W^{G_{f}}_{\Pi}$. We define a \emph{min-max sequence} to be a sequence of properly embedded surfaces $\{\Lambda^{i} _{t_{i}}\}$ if $\mathcal{H}^{n} (\pi^{-1}(\Lambda^{i}_{t_{i}}))$ converges to $W^{G_{f}}_{\Pi}$, where $t_{i} \in [0,1]$ and $\{\Lambda^{i} \}$ is a minimizing sequence.

Let $V_{n}^{G}(M)$ be a space of $G$-equivariant varifolds and $IV_{n}^{G}(M)$ be a space of integral $G$-equivariant varifolds on $(M,g)$ and endow an $F$-metric on the space of varifolds.
\subsection{Pull-tight procedure}
Now we discuss the pull-tight procedure by directly following the arguments in \cite{ketover2016equivariant} and \cite{wang2022min} for equivariant versions (see Proposition 4.1 of \cite{colding2003min} for the general version). We have the following version of the pull-tight procedure.
\begin{prop} \label{pulltight}
    There exists a minimizing sequence of $G_{f}$-sweepouts of $M'$ so that any min-max sequence obtained from it converges to a stationary varifold.
    \begin{proof}
    By Theorem \ref{deformationminimal}, $G_f$-equivariant variations of a surface in $M'$ lift to $G$-equivariant variations of its preimage in $M$, with corresponding variations. We can conclude by following the remaining arguments in \cite{ketover2016equivariant} and \cite{wang2022min}, and obtain the stationarity from $G$-stationarity, with for instance Lemma \ref{lem:stationarity} and \cite[Proposition 4.7]{wang2022min}.
    \end{proof}
\end{prop}
\subsection{$G$-almost minimizing varifolds} We adapt the arguments in \cite{ketover2016equivariant} and \cite{wang2022min} to prove that there exists a min-max sequence that has the $G$-almost-minimizing property relative to $G_{f}$-equivariant isotopies on $M'$. The argument is originally developed by Almgren \cite{almgren1965mimeo} and Pitts \cite{pitts2014existence}.

Let us follow the convention that $G_{f}$-isotopy on $M'$ is induced by the quotient isotopy of $G$-isotopy in Section 3. We take a smooth $G$-equivariant $\Psi_{t}$ to satisfy $\pi \circ \Psi_{t} = \psi_{t} \circ \pi$. For $G$-invariant open set $U$, we define a $G$-equivariant varifold $V$ on $M$ to be \emph{$(G,\delta,\epsilon)$-almost minimizing} in $U$ if there does not exist $G_{f}$-isotopy $\psi_{t} : \pi(U) \rightarrow \pi(U)$ on $M'$ such that
\begin{enumerate}
\item $|(\Psi_{1})_{\sharp}(V))| \le |V|- \epsilon$;
\item $|(\Psi_{t})_{\sharp}(V))| \le |V|+ \delta$ for $t \in [0,1]$.
\end{enumerate}
We define $V$ to be a \emph{$(G,\epsilon)$-almost minimizing} if $V$ is $(G,\epsilon/8,\epsilon)$-almost minimizing varifold. Moreover, we define that $V$ is a \emph{$(G,\epsilon)$-almost minimizing varifold in $G$-annuli} if for each $x \in M$, there exists $r = r(G \cdot x) >0$ such that $V$ is $(G,\epsilon)$-almost minimizing in $An^{G_{f}}(x',s,t)$ for all $G_{f}$-equivariant annuli whose outer radius is at most $r$. Note that $\pi(An^{G}(x,s,t)) = An^{G_{f}}(x',s,t)$ for small $s$ and $t$. We can apply the arguments in Proposition 5.1 in \cite{colding2003min} and Proposition 4.12 in \cite{ketover2016equivariant} to obtain the following almost-minimizing property. Almost-minimizing property for high parameter family follows from Appendix of \cite{colding2018classification} (See also \cite{wang2022min} for equivariant Almgren-Pitts setting). We denote $\tilde{\mathcal{S}}_{0}$ to be the set of vertices of $\tilde{S}$. We take a distance function in a $\bar{g}$ metric.
\begin{prop} \label{almostminimizing} There exists a $G_{f}$-equivariant function $r : M' \rightarrow \mathbb{R}_{+}$ and a min-max sequence $\{\Lambda^{i} _{t_{i}}\}$ such that 
\begin{enumerate}
\item For $x' \notin \tilde{\mathcal{S}}$, $r < \operatorname{dist}(x', \tilde{\mathcal{S}})$;
\item For $x' \in \tilde{\mathcal{S}} \setminus \tilde{\mathcal{S}}_{0}$, $r < \operatorname{dist}(x', \tilde{\mathcal{S}}_{0})$;
\item The sequence $\{ \pi^{-1}\Lambda^{i} _{t_{i}} \}$ is $1/i$-almost minimizing in $An^{G}(x,s,t)$ where $ 0 < s < t \le r( \pi(x))$ for sufficient large $i$;
\item $\{\pi^{-1}\Lambda^{i} _{t_{i}}\}$ converges to $G$-stationary varifold $V$ as $i \rightarrow \infty$.
\end{enumerate}
\begin{proof} The statements directly follow by applying arguments to Proposition 5.1 in \cite{colding2003min} and Proposition 4.12 in \cite{ketover2016equivariant} on the tightened minimizing sequence from Proposition \ref{pulltight}.
\end{proof}
\end{prop}
\subsection{$G$-equivariant replacements} In this section, we construct $G$-equivariant replacements based on the regularity arguments in Section 3 and 4, and $G$-almost minimizing property in the previous section. Suppose $V$ is a limit varifold in Proposition \ref{almostminimizing} and $U$ is a $G$-equivariant set. We define $G$-stationary varifold $V' \in IV_{n}(M)$ to be a $G$-replacement for $G$-stationary varifold $V \in IV_{n}(M)$ if the following holds:
\begin{enumerate}
    \item $V' = V$ in $Gr(M \setminus U)$ and $|V'|= |V|$;
    \item $V'$ is a $G$-stable embedded minimal hypersurface $\Sigma$ in $U$ where $\overline{\Sigma} \setminus \Sigma \subset \partial U$.
\end{enumerate}
Moreover, we call $V$ has the \emph{good $G$-replacement property} if we can iterate the construction of $G$-replacement in a smaller $G$-annulus from the first construction with the radius function $r$ following conditions (1)-(3) in Definition 6.2 in \cite{colding2003min}.

We define the set of points where $\mathbb{Z}_{2}$-singular axis and $\partial M'$ intersects to be $\mathcal{S}_{0, \mathbb{Z}_{2}}$. Let us prove that if $V$ has a $G$-replacement, then the tangent cone at any point on the support of $V$ away from $\pi^{-1}(\mathcal{S}_{0, \mathbb{Z}_{2}})$ is a hyperplane.
\begin{lem}
Let $U$ be a $G$-equivariant open subset of $M$ and $V$ a $G$-equivariant stationary varifold in $U$. If $V$ has an $G$-replacement in $G$-annuli whose radius function is given by $r$, then $V$ is integer rectifiable and any tangent cone to $V$ at $x \in U \setminus \pi^{-1}(\mathcal{S}_{0, \mathbb{Z}_{2}})$ is a hyperplane (possibly with multiplicity).
\begin{proof}
    We obtain the rectifiability by applying the argument of Lemma 6.4 in \cite{colding2003min} to the slices $\mathcal{S}_{x}$ as in the proof of Lemma \ref{rectifiability}. We apply the compactness theorem of embedded $G$-stable minimal hypersurfaces (Lemma \ref{stablecptness}) in place of the compactness theorem of stable minimal hypersurfaces by Schoen-Simon \cite{schoen1981regularity} and obtain that the tangent cone at a point $x \in U \setminus \pi^{-1}(\mathcal{S}_{0, \mathbb{Z}_{2}})$ is a smooth hyperplane by directly following the arguments of Proposition 6.6 in \cite{wang2022min}.
\end{proof}
\end{lem}
Now we prove that the support of $V$ is a $G$-equivariant embedded minimal hypersurface if $V$ has a good $G$-replacement property.
\begin{prop} \label{goodimpliesregular} If $V$ has a good $G$-replacement property in an $G$-invariant open set $U$, then $V$ is a smooth minimal hypersurface in $U$.
\begin{proof}
We mainly point out the modification of the proof of Proposition 6.3 in \cite{colding2003min}. First we suppose $x \in U \setminus \pi^{-1}(\mathcal{S}_{0, \mathbb{Z}_{2}})$ take $r(G \cdot x)$ such that $B^{G}_{t}(x)$ is mean convex for $0 \le t \le r$ and $B^{G}_{t}(x) \cap \pi^{-1}(\mathcal{S}_{0, \mathbb{Z}_{2}}) = \emptyset$.

Then by the maximum principle, we can apply the unique continuation argument in \cite{colding2003min} up to $G \cdot x$, and obtain the local smooth gluing between two replacements. If $x$ is not on singular orbits, then we apply the regularity arguments in Step 4 of the proof of Proposition 6.3 in \cite{colding2003min} and remove the singularity point and take a preimage to obtain the regularity. Also, if $x$ is on a singular orbit but is not on $\pi^{-1}(\mathcal{S}_{0, \mathbb{Z}_{2}})$, then $spt(V) \cap B^{G}_{r}(x)$ is a $G$-stable integral stationary varifold so is a stable integral stationary varifold by Lemma \ref{equivstablity}. Since $x$ is an isolated singularity point in the slice $\mathcal{S}_{x}$, we can apply the classification of stable cone in Simons \cite{simons1968minimal} by taking the tangent cone $T_{x}C = C_{\mathcal{S}} \times T_{x} (G(x))_{x}$ at $x$, and obtain that the tangent cone at $x$ is a hyperplane.

Now suppose $x \in spt(V) \cap \pi^{-1}(\mathcal{S}_{0, \mathbb{Z}_{2}})$. By Theorem \ref{znaction} and \ref{z2action}, $x$ is on a top strata among singular orbits and $\pi(spt(V))$ contains $\tilde{\mathcal{S}}$ and the image of replacements by $\pi$ are orthogonal to $\partial M'$. Hence, as in the proof of Theorem \ref{z2action}, we apply the regularity of stable minimal hypersurfaces by Schoen-Simon \cite{schoen1981regularity} in the same way. Then we apply the uniqueness argument of a geodesic on the link since we have regularity away from $G \cdot x$ by the unique continuation argument. Thus we obtain that $V$ is a smooth minimal hypersurface in $U$.
\end{proof}
\end{prop}
We set a constrained $G_{f}$-equivariant isotopy class $\mathcal{I}_{G_{f}, 1/j}(U')$ as follows:
\begin{equation*}
\mathcal{I}_{G_{f},1/j}(U') = \{ \psi \in \mathcal{I}_{G_{f}}(U') \, | \, \mathcal{H}^{n} (\pi^{-1}(\psi(t,\Lambda^{j}))) \le \mathcal{H}^{n} (\pi^{-1}(\Lambda^{j})) + 1/8j  \text{ for } t \in [0,1]. \}
\end{equation*}
We consider the minimization problem $(\Lambda^{i}, \mathcal{I}_{G_{f},1/j}(U'))$ where $U'$ is an open set in $M'$. We first show the slice version of the squeezing lemma to construct a smooth $G$-equivariant replacement which is a Lie group equivariant analog of Lemma 7.6 in \cite{colding2003min}.
\begin{lem} [Slice squeezing Lemma]\label{squeezing} Suppose $\{\Lambda^{i} \}$ is a minimizing sequence of the minimization problem $(\Lambda, \mathcal{I}_{G_{f},1}(An))$. There exists $\epsilon>0$ such that for sufficiently large $k$, the following holds: for any $\psi \in \mathcal{I}_{G_{f}}(B_{\epsilon}(x') \cap U')$ with $\mathcal{H}^{n} (\pi^{-1}(\psi(1,\Lambda^{k}))) \le \mathcal{H}^{n} (\pi^{-1}(\Lambda^{k}))$, there exists isotopy $\Phi \in \mathcal{I}_{G_{f},1}(B_{\epsilon}(x') \cap U')$ such that $\Psi(1, \cdot) = \Phi(1, \cdot)$.
\begin{proof}
We consider the slice $\mathcal{S}^{r}_{x}$ at $x = \pi^{-1}(x')$. $\pi^{\sharp} \Lambda^{i}$ converges to a stationary varifold and we can apply the standard arguments as monotonicity formula, Sard-type theorem and co-area formula arguments on codimension $1$ stationary varifolds in small $G$-neighborhood $B^{G}_{\epsilon}(x)$. Considering the lower semicontinuity of orbit types and combining with the $G_{f}$-equivariant property of radial deformation in Lemma 4.13 in \cite{ketover2016equivariant}, we can directly adopt the radial deformation argument in the proof of Lemma 7.6 in \cite{colding2003min} on each slice $\mathcal{S}_{x}$ and we can keep using the same $\epsilon$ as arguments therein. The remaining parts of the proof are the same as in Lemma 7.6 of \cite{colding2003min}.
\end{proof}
\end{lem}
We prove that $V$ admits smooth $G$-equivariant replacement in the preimage of any $G_{f}$-annuli whose outer radius is smaller than $r$ in Proposition \ref{almostminimizing}. Here we apply the regularity theorem of minimizers of an isotopy minimization problem we discussed in Section $3$ and $4$.
\begin{prop}  [Construction of $G$-replacements] \label{constreplacement} Let $\{ \pi^{-1}\Lambda^{i} _{t_{i}} \}$ be a sequence of compact $G$-equivariant hypersurfaces in $M$ which converges to the stationary varifold $V$. For $x \in M$, if there exists $G_{f}$-equivariant function $r : M' \rightarrow \mathbb{R}_{+}$ such that in every $G$-annulus of $An^{G}(x,s,t)$ where $0 \le s, t \le r(\pi(x))$ and for sufficiently large $i$, $\{ \pi^{-1}\Lambda^{i} _{t_{i}} \}$ is $1/i$-almost minimizing, then $V$ is a smooth minimal hypersurface.
\begin{proof}
 We can basically follow the ideas in the proof of Theorem 7.1 and Lemma 7.4 in \cite{colding2003min}. Since we have a local genus bound over isotopy minimization in Theorem \ref{regularitydiscLie} from the $\gamma$-reduction and replacement procedures, we can apply Lemma \ref{stablecptness}. By applying Lemma \ref{squeezing}, Lemma \ref{stablecptness} and following the arguments in the proof of Lemma 7.4 in \cite{colding2003min}, we obtain that minimizers of constrained $G$-equivariant minimizing problem are $G$-stable minimal hypersurfaces where the radius of annuli is given by the radius in Lemma \ref{squeezing} for good $G$-replacement property. We also follow the diagonal argument in the proof of Theorem 7.1 in \cite{colding2003min} and obtain a good $G$-replacement property of $V$. The regularity of $V$ is achieved by the good $G$-replacement property and Proposition \ref{goodimpliesregular}.
\end{proof}
\end{prop}
\subsection{Regularity of min-max $G$-equivariant minimal hypersurfaces}
In this section, we prove the following min-max construction of $G$-equivariant minimal hypersurfaces.
\begin{thm}[Regularity of cohomogeneity $2$ min-max $G$-equivariant minimal hypersurfaces]  \label{minmaxreg} On $M^{n+1}$ $(4 \le n+1 \le 7)$, suppose $G$ is a product of smooth Lie group $G_{c}$ whose orbits are connected and a finite group $G_{f}$, which acts on $M$ as an isometry whose principal orbit has a codimension $3$ where $M'=M/G_{c}$ is a $3$-manifold with boundary. Moreover, a singular axis $\tilde{\mathcal{S}}$ of a finite group action $G_{f}$ intersects $\partial M'$ only on smooth points and $G_{f}$ is a orientation preserving isometry on $M'$ and such that $M'/G_{f}$ is an orbifold (with boundary). Let $\{ \Lambda_{t} \}_{t=0}^{1}$ be a sweepout of $M'$ and corresponding $G_{f}$-equivariant saturation by $\Pi$.

Then there exists a min-max sequence $\Lambda^{j}_{t_{j}}$ where $\Sigma_{j}:= \pi^{-1}(\Lambda^{j}_{t_{j}})$ converges as a varifold $V$ such that
\begin{enumerate}
    \item $V = \sum_{i=1}^{k} n_{i} |\Gamma_{i}|$, where each $\Gamma_{i}$ is a smooth, connected, embedded pairwise disjoint minimal hypersurface each other;
    \item $||V||(M) = W^{G_{f}}_{\Pi}$.
\end{enumerate}
          
\begin{proof}
    By applying Proposition \ref{almostminimizing}, we obtain a $G$-almost minimizing min-max sequence $\{\Lambda^{i} _{t_{i}}\}$. Then we apply Proposition \ref{constreplacement} and obtain the regularity of $V$.
\end{proof}
\end{thm}
\begin{rmk} We expect that generic multiplicity one result as in Wang-Zhou \cite{wang2023existence} on Simon-Smith min-max construction on dimension $3$ also holds in this setting (See multiplicity one result for Almgren-Pitts min-max setting in higher dimension in Zhou \cite{zhou2020multiplicity}). However, we do not focus on it since we can rule out higher multiplicity with area arguments in our applications.
\end{rmk}
\section{Topology of min-max $G$-equivariant minimal hypersurfaces}
In this section, we discuss the topology, mainly the genus of the image by the projection map $\pi$, of min-max $G$-equivariant minimal hypersurfaces. More specifically, we obtain the upper bound of the genus of the projection of min-max $G$-equivariant minimal hypersurfaces based on the previous genus bound results in De Lellis-Pellandini \cite{de2010genus} and its free boundary adaptation in Li \cite{li2015general} (See effective version for Ketover \cite{ketover2019genus} and recent progress on the Betti number bound in free boundary setting by Franz-Schulz \cite{franz2023topological}).
\subsection{Genus control of the projection of min-max $G$-equivariant minimal hypersurfaces} In this section, we prove an upper bound of the genus of the projection of min-max $G$-equivariant minimal hypersurfaces based on previous works on the genus bound whose key idea arises from Simon's lifting lemma (Proposition 2.1 in \cite{de2010genus}) and its applications in other settings.

Since we are trying to apply the lifting arguments in the metric $g'$ while avoiding regions where the metric degenerates, we introduce an exhaustion of $M'$ and approximate the whole $M'$ by this exhaustion. We denote $\Omega_{\epsilon}:= \{ x' \in M' \, | \, d_{M'}(x',\partial M') > \epsilon\}$ where $\overline{g}$ in endowed on $M'$.
\begin{thm} \label{genusbound} Suppose $\{ \Lambda_{t} \}_{t=0}^{1}$ be a sweepout of $M'$ to be a sweepout whose genus is bounded above by $g$ and $\Gamma = \sum_{i=1}^{k} n_{i} \Gamma_{i}$ to be a min-max $G$-equivariant minimal hypersurfaces obtained in Theorem \ref{minmaxreg}. Then $\sum_{i=1}^{k} \operatorname{genus}(\pi(\Gamma_{i})) \le g$.
\begin{proof}
By Theorem \ref{deformationminimal} and Theorem \ref{minmaxreg}, Note that we can regard min-max minimal hypersurfaces we obtained in Theorem \ref{minmaxreg} as a preimage of minimal surfaces produced by Simon-Smith min-max theory in $(M',g')$ in the interior part. Hence, here we adapt the lifting argument and obtain a genus bound on an open subset of $\operatorname{Int}(M')$ (cf. Ketover-Zhou \cite{ketover2018entropy} in a self-shrinker setting).

Consider an exhaustion $\Omega_{\epsilon_{j}}$ of $M'$ where $\epsilon_{j}>0$ converges to $0$ as $j \rightarrow \infty$. We allow smoothening of $\Omega_{\epsilon_{j}}$ (in the $G_{f}$-equivariant way) if necessary and take $\epsilon_{j}$ to satisfy that $\pi(\Gamma)$ is transverse to $\partial \Omega_{\epsilon_{j}}$. Then we can find $2g_{j}$ closed curves $\gamma_{1}, ..., \gamma_{2 g_{j}}$ such that $(\pi(\Gamma) \cap \Omega_{\epsilon_{j}}) \setminus \cup_{j=1}^{2 g_{j}} \gamma_{j}$ is a genus $0$ domain with ends on $\partial \Omega_{\epsilon_{j}}$, more precisely, we take a curve which is a basis of $\operatorname{im} \bigl(H_1(\pi(\Gamma) \cap \Omega_{\epsilon_{j}};\mathbb Q) \rightarrow H_1(\pi(\Gamma) \cap \Omega_{\epsilon_{j}},\partial (\pi(\Gamma) \cap \Omega_{\epsilon_{j}});\mathbb Q))$, where $\dim_{\mathbb Q}\operatorname{im}(H_1(\pi(\Gamma) \cap \Omega_{\epsilon_{j}};\mathbb Q)\rightarrow H_1(\pi(\Gamma) \cap \Omega_{\epsilon_{i}},\partial (\pi(\Gamma) \cap \Omega_{\epsilon_{i}});\mathbb Q))=2\operatorname{genus}(\pi(\Gamma) \cap \Omega_{\epsilon_{i}})$. We can apply the lifting arguments in \cite{de2010genus} and obtain $g_{j} \le g$. Since this holds for all $j$, by taking the limit, we have $\operatorname{genus}(\pi(\Gamma)) =  \sum_{i=1}^{k} \operatorname{genus}(\pi(\Gamma_{i})) \le g$.
\end{proof}
\end{thm}
\begin{rmk} We only focus on the genus bound without multiplicity since it is sufficient for our applications.
\end{rmk}
\begin{rmk}
In three-dimensional convex boundary cases in Gr{\"u}ter-Jost \cite{gruter1986embedded}, if $\Sigma'$ does not intersect $\partial M'$, then area minimizers of isotopy minimization problems do not intersect $\partial M'$. However, in our case, boundary components on $M'$ may be produced along the limiting procedure of minimizing sequence, due to the absence of convexity by degeneracy of the metric on the boundary, which is similar to Theorem 5.1 in Jost \cite{jost1986existence} where the limiting minimal surface also achieves boundary orthogonality in the $3$-dimensional case. This increment of a boundary component may bring additional Betti number to the min-max minimal hypersurface on $M$ that we can obtain.
\end{rmk}
\section{Minimal hypersurfaces on $\mathbb{S}^{n+1}$ with arbitrarily large first Betti number}
In this section, we construct sequences of minimal hypersurfaces with arbitrarily large first Betti number on high dimensional round spheres by applying $G$-equivariant min-max construction in previous sections. We modify the setup in Section 5 of Pitts-Rubinstein \cite{pitts1987applications}.
\subsection{Sweepouts}
We first construct the sweepout. Pitts-Rubinstein described the sweepout in \cite{pitts1987applications} in their monograph and we focus on providing a rigorous construction. Our formulation on the orbit space also resembles Ketover \cite{ketover2016free} in $3$-dimensional free boundary setting (see also Buzano-Nguyen-Schulz \cite{buzano2024noncompact} for the self-shrinker construction of mean curvature flow).

Let us consider $(\mathbb{S}^{n+1},g_{round})$ for $4 \le n+1 \le 7$. Denote $x = (x_{1}, ..., x_{n+2}) \in \mathbb{R}^{n+2}$ so that $\mathbb{S}^{n+1} = \{ x \, | \, |x| =1 \}$. We also denote $k$-dimensional volume of $\mathbb{S}^{k}$ by $\omega_{k}$. We simplify notations of known (Clifford-type) minimal hypersurfaces on $\mathbb{S}^{n+1}$ by $S^{n} = \mathbb{S}^{n} \times \{ 0 \}$, $S^{n-1} \times S^{1} = \sqrt{n-1/n} \, \mathbb{S}^{n-1} \times  \sqrt{1/n} \, \mathbb{S}^{1}$, and $S^{n-2} \times S^{2} = \sqrt{n-2/n} \, \mathbb{S}^{n-2} \times  \sqrt{2/n} \, \mathbb{S}^{2}$. Note that the intersection between $S^{n-1} \times S^{1}$, $S^{n-2} \times S^{2}$ and $S^{n}$ are $\sqrt{n-1/n} S^{n-2} \times \sqrt{1/n} S^{1}$ and $\sqrt{n-2/n} S^{n-2} \times \sqrt{2/n} S^{1}$, respectively. 

Now we take $G_{c} = SO(n-1)$ and $G_{f}= \mathbb{D}_{g+1}$ where $\mathbb{D}_{m}$ is a dihedral group with $2m$ elements, where $G_{c}$ acts on the first factor of decomposition on $S^{n-1} \times S^{1}$ and $S^{n-2} \times S^{2}$ whose orbits are $n-2$ dimensional spheres $S^{n-2}(\sqrt{1-x_{n}^{2}-x_{n+1}^{2}-x_{n+2}^{2}}) \times (x_{n},x_{n+1},x_{n+2}) \in \mathbb{S}^{n+1}$. Hence, the isometry action $G_{c}$ on $\mathbb{S}^{n+1}$ induces the orbit space $\mathbb{B}^{3} = \{ x'=(x_{1}', x_{2}', x_{3}') \, | \, x_{1}'^{2}+x_{2}'^{2}+x_{3}'^{2} \le 1\} $. Note that $S^{n}/ G_{c} =  \{(x_{1}', x_{2}', x_{3}') \, | \, x_{3}' = 0 \text{ and } x' \in \mathbb{B}^{3} \}$, $(S^{n-1} \times S^{1})/G_{c} = \{(x_{1}', x_{2}', x_{3}') \, | \, x_{1}'^{2}+x_{2}'^{2} = 1/n \text{ and } x' \in \mathbb{B}^{3} \}$, and $(S^{n-2} \times S^{2})/G_{c} =\{(x_{1}', x_{2}', x_{3}') \, | \, x_{1}'^{2}+x_{2}'^{2}+x_{3}'^{2} = 2/n \}$. Here $S^{1}$ lies on the $(x_{1}', x_{2}')$-plane. The induced metric on the orbit space $\mathbb{S}^{n+1}/SO(n-1)$ is written by
\[
ds_{g'}^{2} = \omega_{n-2}\sin^{n-2}\theta(d\theta^{2} + \cos^{2}\theta\,d\psi^{2} + \cos^{2}\theta\,\sin^{2}\psi\,d\phi^{2}), \quad
\]
where $x_{1}=\cos\theta\,\sin\psi\,\cos\phi,\  x_{2}=\cos\theta\,\sin\psi\,\sin\phi,\ x_{3}=\cos\theta\,\cos\psi$ with coordinate ranges $\theta \in [0,\tfrac{\pi}{2}],  \psi \in [0,\pi],  \phi \in [0,2\pi)$ by (\ref{volumefunction}). Hence, the boundary $\mathbb{B}^{3} = \{ x'=(x_{1}', x_{2}', x_{3}') \, | \, x_{1}'^{2}+x_{2}'^{2}+x_{3}'^{2} = 1\}$ collapses into a point and $(B^{3},g')$ looks like a balloon with a collapsed tip geometrically.

The dihedral group $\mathbb{D}_{g+1}$ acts on $\mathbb{B}^{3}$ by rotations of $2\pi /(g+1)$ about the $x_{3}'$ axis, and also rotations of $\pi$ about the $(g+1)$ horizontal axes $\{ L_{i} \}_{i=0}^{g}$ defined by
\[
L_{i} : = \Big\{ \Big(r \cos \Big( \frac{i}{g+1} \pi \Big), r \sin \Big( \frac{i}{g+1} \pi \Big), 0  \Big) \, | \, -1 \le r \le 1 \Big\}.
\]
Also we denote $D = \mathbb{B}^{3} \cap \{ x_{3}'=0\}$, which corresponds to $S^{n}/ G_{c}$. Now we construct a sweepout. There exists one-parameter continuous family of $\mathbb{D}_{g+1}$-invariant singular surfaces $\{ \Sigma_{t}' \}_{t \in [0,1]}$ with the following properties:
\begin{enumerate}
\item $\Sigma_{0}' = D$;
\item $\Sigma_{1}' = D$, where $\Sigma_{t}' \setminus D$ converges to $\partial \mathbb{B}^{3}$ as $t \rightarrow 1-$;
\item $\Sigma_{t}' = D \cup S_{t}$ where $S_{t}$'s are smooth sphere defined by $S_{t} = \{(x_{1}', x_{2}', x_{3}') \, | \, x_{1}'^{2}+x_{2}'^{2}+x_{3}'^{2} = t\} \cap \mathbb{B}^{3}$.
\end{enumerate}
Now we consider $\mathbb{D}_{g+1}$-equivariant desingularization to define a sweepout.  Along very small neighborhoods of an intersection circle $D \cap S_{t}$, we perform a $\mathbb{D}_{g+1}$-equivariant desingularization to create genus $g$ as in Ketover \cite{ketover2016free}, which rounds off two pieces of surfaces after pairing adjacent pieces while keeping the $\mathbb{D}_{g+1}$-symmetry.

More precisely, we run the following procedure. For each $0<t<1$, let $C_t:=D\cap S_t$. The points $C_t\cap\bigcup_{i=0}^{g}L_i$ divide $C_t$ into $2(g+1)$ consecutive arcs. Following the desingularization construction in \cite[Section 4.1]{ketover2016free}, we resolute the intersection along these arcs by alternating between the two possible pairings of the adjacent sheets of $D$ and $S_t$. At the endpoints of the arcs, we join these resolutions by adding smooth necks that are invariant under the corresponding symmetry and preserve the radial axes $L_i$. Topologically, the construction joins $D$ and $S_t$ through $g+1$ necks. We choose the desingularizations to depend smoothly on
$t\in(0,1)$. 

Note that a desingularization can be operated to decrease an area in $g'$ metric and so to decrease $n$-dimensional area of the preimage by Theorem \ref{deformationminimal}. Hence we obtain the following existence of $\mathbb{D}_{g+1}$-equivariant sweepout.
\begin{lem}\label{sweepoutconditions}For $g\ge1$, there exists a $\mathbb{D}_{g+1}$-equivariant sweepout $\{ \tilde{\Sigma}_{t}' \}_{t \in [0,1]}$ in $\mathbb{B}^{3}$ such that
\begin{enumerate}
    \item $\mathcal{H}^{n}(\pi^{-1}(\tilde{\Sigma}_{t}'))< \mathcal{H}^{n}(S^{n}) + \mathcal{H}^{n}(S^{n-2} \times S^{2})$ for $0 \le t \le 1$;
    \item $\tilde{\Sigma}_{0}' = \tilde{\Sigma}_{1}' = D$;
    \item $\tilde{\Sigma}_{t}'$ has a genus $g$ for $0<t<1$;
    \item $\tilde{\Sigma}_{t}'$ contains $\{ L_{i} \}_{i=0}^{g}$ and contains an origin for all $t$;
\end{enumerate}
\begin{proof}
    We take a desingularized surfaces by the previous description of the desingularization of $\{ \Sigma_{t}' \}_{t \in [0,1]}$. Then it suffices to consider (1) and (4). $\pi^{-1}(S_{t}) = \sqrt{1-t} \, \mathbb{S}^{n-2} \times \sqrt{t}\, \mathbb{S}^{2}$ and it achieves maximum value at $t=2/n$ and since desingularization also decreases $n$-dimensional area and we obtain (1). By construction, $\Sigma_{t}'$ contains $\{ L_{i} \}_{i=0}^{g}$ and the origin, and we can take a desingularization which does not impede our sweepout from avoiding axes $\{ L_{i} \}_{i=0}^{g}$, and we proved (4).
\end{proof}
\end{lem}
We prove a strict lower bound of the width by recalling the classical isoperimetric inequality arguments (see Gromov's survey article \cite{gromov1980paul}), as in \cite{ketover2016free} and \cite{buzano2024noncompact} for free boundary and shrinker version, respectively.
\begin{thm} [\cite{gromov1980paul}] \label{isoperimetrysphere}  The isoperimetric hypersurfaces in $\mathbb{S}^{n+1}$ are geodesic hyperspheres. Hence if a hypersurface $\Sigma^{n}$ on $\mathbb{S}^{n+1}$ divides $\mathbb{S}^{n+1}$ into two components with equal volume, then $\mathcal{H}^{n}(\Sigma) \ge \mathcal{H}^{n}(\mathbb{S}^{n})$, and the equality holds if and only if $\Sigma$ is an equator of $\mathbb{S}^{n+1}$.
\end{thm}
Let us consider $\mathbb{D}_{g+1}$-equivariant saturation $\Pi$ and corresponding width $W_{n+1}$. As in \cite{ketover2016free}, denote the upper and lower half of $\mathbb{B}^{3} \setminus D$ by $C_{1}$ and $C_{2}$, respectively. Then we have the half-volume and swapping properties as follows, which ensure $\{ \tilde{\Sigma}_{t}' \}_{t \in [0,1]}$ and its $\mathbb{D}_{g+1}$-equivariant isotopy deformation to be a sweepout.
\begin{enumerate}
    \item For each $t$, $\tilde{\Sigma}_{t}'$ separates $\mathbb{B}^{3}$ into two components $A(t)$ and $B(t)$, which are continuous at $t \in [0,1]$, where $\mathcal{H}^{n+1}(\pi^{-1}(A(t))) = \mathcal{H}^{n+1}(\pi^{-1}(B(t)))$;
    \item $A(0) = C_{1}$ and $B(0)=C_{2}$;
    \item $A(1) = C_{2}$ and $B(1)=C_{1}$.
\end{enumerate}
\begin{lem} \label{nontrivalwidth} We have the following width estimates:
\begin{equation} \label{strictineq}
    W_{n+1} >\mathcal{H}^{n}(\mathbb{S}^{n}).
\end{equation}
\begin{proof}
    Let us assume a contradiction and we have $W_{n+1} =\mathcal{H}^{n}(\mathbb{S}^{n})$ by Theorem \ref{isoperimetrysphere}. Then by the definition of the width, there exists a minimizing sequence $\{ \tilde{\Sigma}_{i,t}' \}_{t \in [0,1]}$ and in the $\mathbb{D}_{g+1}$-equivariant saturation $\Pi$ such that
    \begin{equation*}
        \sup_{t \in [0,1]} \mathcal{H}^{n}(\pi^{-1}(\tilde{\Sigma}_{i,t}')) < \mathcal{H}^{n}(\mathbb{S}^{n}) + \frac{1}{i}.
    \end{equation*}
For each $i$, we take $t_{i}$ such that 
\begin{equation} \label{halfvolumeprop}
\mathcal{H}^{n+1}(\pi^{-1}(A_{i}(t_{i}) \cap C_{1})) = \mathcal{H}^{n+1}(\pi^{-1}(A_{i}(t_{i}) \cap C_{2})) = \mathcal{H}^{n+1}(\mathbb{S}^{n+1})/4.
\end{equation}
Then by the half-volume property (1) in the preamble, $\mathcal{H}^{n}(\pi^{-1}(\tilde{\Sigma}_{i,t}')) \ge \mathcal{H}^{n}(\mathbb{S}^{n})$ and $\{ \tilde{\Sigma}_{i,t_{i}}' \}$ is a min-max sequence. We denote $\tilde{\Sigma}_{i,t_{i}} := \pi^{-1}(\tilde{\Sigma}_{i,t_{i}}')$. Then by compactness argument,  $\tilde{\Sigma}_{i,t_{i}}$ converges to $n$-dimensional totally geodesic hypersphere $S$ by Theorem \ref{isoperimetrysphere}. Since $\pi(S)$ needs to be $\mathbb{D}_{g+1}$-equivariant and changes the two sides along $\pi$-rotation, $\pi(S) = D$. However, this contradicts (\ref{halfvolumeprop}) and we obtain the strict inequality (\ref{strictineq}). 
\end{proof}
\end{lem}
\subsection{Determining the genus} By the genus control in Theorem \ref{genusbound}, we have the upper bound of the genus of image of min-max minimal hypersurfaces by the projection map. We first rule out the case of genus zero, by developing analogous topological ideas of Mcgrath-Zou \cite{mcgrath2024areas} in our setting, based on the two-piece property of minimal hypersurfaces on $\mathbb{S}^{n+1}$ by Choe-Soret \cite{choe2009first}. We denote a $(SO(n-1) \times \mathbb{D}_{g+1})$-equivariant min-max minimal hypersurface we can obtain by $\Sigma^{n}_{g}$.
\begin{lem} \label{nonzerogenus} $\pi(\Sigma^{n}_{g})$ has nonzero genus in $(\mathbb{B}^{3},g')$.
\begin{proof}
    Suppose $\pi(\Sigma^{n}_{g})$ has genus $0$ and we prove by contradiction. Firstly, $\pi(\Sigma^{n}_{g})$ consists of one connected component by Frankel property, Since $\pi(\Sigma^{n}_{g})$ contains $\cup_{i=0}^{g} L_{i}$, $\pi(\Sigma^{n}_{g})$ has at least one boundary component. Divide the cases into two cases when $\pi(\Sigma^{n}_{g})$ has one boundary component and multiple boundary components. We can assume the boundary components of $\pi(\Sigma^{n}_{g})$ consist of simple closed curves by the regularity (Theorem \ref{secondcontireg}).

    First, assume $\pi(\Sigma^{n}_{g})$ consists of two or more boundary curves and genus $0$. By the two-piece property of minimal hypersurfaces on $S^{n+1}$ by Theorem 1.2 in Choe-Soret \cite{choe2009first}, every equatorial disk, whose preimage is an equatorial $\mathbb{S}^{n}$, separates $\pi(\Sigma^{n}_{g})$ into two connected components. Also note that this two-piece property is a topological property, so it also holds on the degenerate metric $(\mathbb{B}^{3},g')$. Hence, by the same argument as Lemma 2.1 of Mcgrath-Zou \cite{mcgrath2024areas}, we can deduce that $\pi(\Sigma^{n}_{g})$ intersects equatorial disks transversally and we obtain $0 \notin \pi(\Sigma^{n}_{g})$. However, this contradicts our $\mathbb{D}_{g+1}$-equivariant construction based on the condition (4) in Lemma \ref{sweepoutconditions} which gives $0 \in \pi(\Sigma^{n}_{g})$.
     
    Now suppose $\pi(\Sigma^{n}_{g})$ has only one boundary component and genus $0$. By the condition (4) in Lemma \ref{sweepoutconditions} again, our limit surface contains an origin and axes and by considering $\mathbb{D}_{g+1}$-symmetry, and since the multiplicity is odd from \cite[Theorem 3.2]{ketover2016free}. This contradicts the two-piece property if it is not a disk. We can rule out the case of (multiple of) disk by Lemma \ref{nontrivalwidth} since $\pi^{-1}(D) = S^{n}$ and $|S^{n}|<|S^{n}|+|S^{n-2} \times S^{2}|<3|S^{n}|$. Thus, $\pi(\Sigma^{n}_{g})$ has nonzero genus.
\end{proof}
\end{lem}
Now we prove that the genus of $\pi(\Sigma^{n}_{g})$ is $g$. We employ the Riemann-Hurwitz formula arguments as in Lemma 2.10 of Buzano-Nguyen-Schulz \cite{buzano2024noncompact}.
\begin{lem} \label{genusg} The genus of $\pi(\Sigma^{n}_{g})$ is $g$.
\begin{proof} Consider a sequence of exhaustions $\Omega_{\epsilon_{j}} := \{(x_{1}', x_{2}', x_{3}') \, | \, x_{1}'^{2}+x_{2}'^{2}+x_{3}'^{2} \le 1-\epsilon_{j}\} \subsetneq \mathbb{B}^{3}$ where $\epsilon_{j} \rightarrow 0$ as $j$ goes to infinity. Take $\Sigma^{n}_{\epsilon_{j},g}:= \Omega_{\epsilon_{j}} \cap \pi(\Sigma^{n}_{g})$. Then for large $j$, by Theorem \ref{genusbound} and Lemma \ref{nonzerogenus}, $\operatorname{genus}(\Sigma^{n}_{\epsilon_{j},g}) \in \{1,2, ..., g \}$ and we can apply Corollary 2.10 in \cite{buzano2024noncompact} and obtain $\operatorname{genus}(\Sigma^{n}_{\epsilon_{j},g})=g$. By sending $j$ to infinity, we obtain $\text{genus}(\pi(\Sigma^{n}_{g})) = g$.
\end{proof}
\end{lem}
\subsection{Existence of minimal hypersurfaces with large Betti numbers} In this section, we construct minimal hypersurfaces with the arbitrarily large first Betti number on $\mathbb{S}^{n+1}$ for $4 \le n+1 \le 7$.
\begin{thm} \label{minhypsphere} For dimension $4 \le n+1 \le 7$, there exists a sequence of minimal hypersurfaces $\Sigma_{g}$ on $\mathbb{S}^{n+1}$ such that $\Sigma_{g}$ has a symmetry $SO(n-1) \times \mathbb{D}_{g+1}$. For sufficiently large $g$, $\Sigma_{g}$ has a diffeomorphism type of connected sums $\#^{2g}(S^{1} \times S^{n-1})$ or $\#^{2g+2}(S^{1} \times S^{n-1})$ for $n+1=4, 7$, and $\#^{2g}(S^{1} \times S^{n-1})$ for $n+1=5, 6$. Hence, the first Betti number is $2g$ or $2g+2$ for $n+1=4,7$, and $2g$ for $n+1=5,6$. Moreover, $\Sigma_{g}$ (after passing to a subsequence) converges to $S^{3} \cup (S^{2} \times S^{1})$ or $S^{3} \cup (S^{1} \times S^{2})$ for $n+1=4$, and in particular, $\Sigma_{g}$ converges to $S^{n} \cup (S^{n-2} \times S^{2})$ for $n+1=5, 6$ as $g \rightarrow \infty$.
\begin{proof} Let us send $g$ to infinity and we denote a (subsequential) limit of $\Sigma_{g}$ by $\Sigma_{\infty}$. Then $\Sigma_{\infty}$ has a $(SO(n-1) \times S^{1} \times \mathbb{Z}_{2})$-symmetry. Since each $\pi(\Sigma_{g})$ contains $\{L_{i} \}_{i=0}^{g}$, $\pi(\Sigma_{\infty})$ contains $D$ as a limit. We apply the local arguments in Section 4.3 in \cite{ketover2016free} to the interior part of $\pi(\Sigma_{g})$ as a minimal surface in $(\mathbb{B}^{3},g')$. We first define the blow-up set by
\[
\mathcal{B} := \{ x \in B \, | \, \inf_{r>0} \liminf_{g \rightarrow \infty} \int_{B_{r}(x)} |A|^{2}_{\pi(\Sigma_{g})} d \mathcal{H}^{2} \ge \epsilon_{0}\},
\]
where $\epsilon_{0}$ is a constant of $\epsilon$-regularity theorem of Choi-Schoen \cite{choi1985space}. Notice the area comparison $|S^{n-2} \times S^{2}| < 2 |S^{n}|$ and that the integrated Gauss-Bonnet argument in \cite{ilmanen1998lectures} in \cite{ketover2016free} applies directly to our setting. With the genus bound of min-max minimal surface in Theorem \ref{genusbound}, and by applying the integrated Gauss-Bonnet arguments for each genus $g$ and on each $\mathbb{D}_{g+1}$-orbit as in \cite{ketover2016free}, we obtain that $\mathcal{B} \cap \text{Int}(\mathbb{B}^{3})$ is a single circle in $D$. Hence, we obtain $\pi(\Sigma_{\infty}) = D \cup \Sigma'$ where $\Sigma'$ is a $(S^{1} \times \mathbb{Z}_{2})$-equivariant embedded free boundary minimal surface in $(\mathbb{B}^{3},g')$ which intersects $D$ transversally and $\partial \mathbb{B}^{3}$ orthogonally. Also, $\pi(\Sigma_{g})$ has genus $g$ by Lemma \ref{genusg}.

Now we compute the diffeomorphism type of $\Sigma_{g}$. We consider $\Sigma_{\infty}/(SO(n-1) \times S^{1}) = (D \cup \Sigma')/S^{1}$, which is a union of two embedded arcs in a topological disk. Note that $\Sigma'/S^{1}$ intersects the boundary orthogonally by the regularity. Hence, for sufficient large $g$, by Allard regularity theorem \cite{allard1972first}, we have the graphical convergence $\Sigma_{g} \rightarrow \Sigma_{\infty}$, and local graphical convergence of $\pi(\Sigma_{g})$ to $\pi(\Sigma_{\infty})$ by $SO(n-1)$-symmetry away from $\pi^{-1}(D \cap \Sigma')$. On $\Sigma_{g} \cap \pi^{-1}(\partial \mathbb{B}^{3})$, we can apply Allard regularity upstairs and obtain the graphical convergence. By projecting back to $B^{3}$, we obtain a sheet of surface which is orthogonal to the boundary. Hence, $\pi(\Sigma_{g})$ has genus $g$ from Lemma \ref{genusg}, and one boundary component or three boundary components.

We compute the topology of the preimage of a properly embedded surface $\Sigma_{g,b}'$ in $\mathbb{B}^{3}$, where $\Sigma_{g,b}'$ has genus $g$ and $b$ boundary components. $\Sigma_{g,b+1}'$ can be obtained by $(\Sigma_{g,b}' \setminus D^{2}) \cup_{S^{1} \times \{ 0\}}  (S^{1} \times I)$ where $D^{2},  S^{1} \times \{ 0\} \subset Int(B^{3})$ and $S^{1} \times \{1 \} \subset \partial B^{3}$. Hence, 
\[
\pi^{-1} (\Sigma_{g,b+1}') = (\pi^{-1}(\Sigma_{g,b}') \setminus (S^{n-2} \times D^{2})) \cup_{S^{n-2} \times  S^{1}}  (D^{n-1} \times S^{1}),
\]
and this is a codimension $2$ surgery and equivalent to connecting one piece of $S^{n-1} \times S^{1}$ in our setting. Similarly, adding genus by one, which can be realized by removing two disks and attaching a cylinder over the boundary circles in $B^{3}$, appears in the preimage as \[
\pi^{-1} (\Sigma_{g+1,b}') = (\pi^{-1}(\Sigma_{g,b}') \setminus ((S^{n-2} \times D^{2}) \cup (S^{n-2} \times D^{2})) \cup_{\partial (S^{n-2} \times S^{1} \times I)}  (S^{n-2} \times S^{1} \times I),
\]
and this is equivalent to attaching two $1$-handles so it corresponds to connecting two pieces of $S^{n-1} \times S^{1}$. Hence, we obtain that $\pi^{-1}(\Sigma_{g,b}')$ is diffeomorphic to $\#^{(b-1+2g)} (S^{n-1} \times S^{1})$. Since $\pi(\Sigma_{g})$ has genus $g$, and one boundary component or three boundary components, $\Sigma_{g}$ achieves the topological type $\#^{2g} (S^{n-1} \times S^{1})$ or $\#^{2g+2} (S^{n-1} \times S^{1})$.

By Theorem \ref{deformationminimal}, $\pi^{-1}(\Sigma')$ is a $(SO(n-1) \times S^{1} \times \mathbb{Z}_{2})$-equivariant minimal hypersurface. Moreover, by Lemma \ref{sweepoutconditions}, $\mathcal{H}^{n}(\pi^{-1}(\Sigma')) \le \mathcal{H}^{n}(S^{n-2} \times S^{2})$. By Willmore type inequality on antipodal hypersurfaces of Viana \cite[Theorem 1.4]{viana2023isoperimetry} on minimal hypersurfaces on $\mathbb{S}^{n+1}$, we have the lower bound $\mathcal{H}^{n}(\pi^{-1}(\Sigma')) \ge |\sqrt{\lfloor n/2 \rfloor / n} \mathbb{S}^{\lfloor n/2 \rfloor} \times \sqrt{\lceil n/2 \rceil / n} \mathbb{S}^{\lceil n/2 \rceil}|$. This gives $\pi^{-1}(\Sigma')$ is isometric to $S^{2} \times S^{n-2}$ when $n+1=4, 5, 6$. 

Note that there is a unique configuration of $S^{2} \times S^{n-2}$ with $(SO(n-1) \times S^{1} \times \mathbb{Z}_{2})$-symmetry when $n+1= 5,6$, while there are two configurations, which are $\pi^{-1}(\{ x_{1}'^{2}+x_{2}'^{2}+x_{3}'^{2} = 2/3 \} )$ and $\pi^{-1}(\{ x_{1}'^{2}+x_{2}'^{2} = 1/3 \})$, in dimension $4$. This proves that $\Sigma_{\infty}$ is $S^{3} \cup (S^{2} \times S^{1})$ or $S^{3} \cup (S^{1} \times S^{2})$ for $n+1=4$, and is $S^{n} \cup (S^{n-2} \times S^{2})$ for $n+1=5, 6$. When $n+1=5,6$, for sufficiently large $g$, we obtain that $\pi(\Sigma_{g})$ has one boundary component, and $\Sigma_{g}$ achieves the topological type $\#^{2g} (S^{n-1} \times S^{1})$ by Allard regularity theorem and graphical convergence as in above. Moreover, if $n+1=4,7$ and for sufficiently large $g$, $\pi(\Sigma_{g})$ has at most three boundary components by Allard regularity theorem and the graphical convergence and we have the topological type of $\Sigma_{g}$ to be $\#^{2g} (S^{n-1} \times S^{1})$ or $\#^{2g+2} (S^{n-1} \times S^{1})$.
\end{proof} 
\end{thm}
 \subsection{Discussion on the topological type of minimal hypersurfaces on dimension $7$ and higher} In the proof of Theorem \ref{minhypsphere}, we proved that $\Sigma_{g}$ achieves a diffeomorphism type of connected sums $\#^{2g}(S^{1} \times S^{n-1})$ for large $g$, by proving that $\Sigma_{g}$ converges to $S^{n} \cup (S^{n-2} \times S^{2})$ for $n+1=5,6$. This follows from the Willmore type inequality statement of Viana \cite{viana2023isoperimetry} for hypersurfaces separating $\mathbb{S}^{n+1}$ in two connected antipodal invariant regions that $S^{n-2} \times S^{2}$ is a $(SO(n-1) \times SO(2) \times \mathbb{Z}_{2})$-equivariant minimal hypersurface which achieves a minimum area except for the equator $S^{n}$ for $n+1 \le 6$. While the inequality gives the lower bound of the area $|\sqrt{\lfloor n/2 \rfloor / n} \mathbb{S}^{\lfloor n/2 \rfloor} \times \sqrt{\lceil n/2 \rceil / n} \mathbb{S}^{\lceil n/2 \rceil}|$ for general minimal hypersurfaces which are not equatorial spheres, this does not give the (larger) lower bound of areas of $(SO(n-1) \times SO(2) \times \mathbb{Z}_{2})$-equivariant minimal hypersurfaces in dimension $7$ and higher. 

Since we have run one parameter min-max construction, one possible approach is to obtain the classification of $(SO(n-1) \times SO(2) \times \mathbb{Z}_{2})$-equivariant minimal hypersurfaces whose $(SO(n-1) \times SO(2) \times \mathbb{Z}_{2})$-Morse index is smaller than or equal to $1$. By Do Carmo-Ritoré-Ros \cite{do2000compact} on the classification of index one two-sided minimal hypersurfaces in $\mathbb{RP}^{n+1}$, we know that only minimal hypersurfaces on $\mathbb{S}^{n+1}$ with antipodal symmetry with $\mathbb{Z}_{2}$-Morse index $1$ are Clifford hypersurfaces. Their arguments prove the existence of two $\mathbb{Z}_{2}$-equivariant eigenfunctions for the stability operator with negative eigenvalue for minimal hypersurfaces which are not Clifford hypersurfaces or geodesic spheres, which do not guarantee the existence of two $(SO(n-1) \times SO(2) \times \mathbb{Z}_{2})$-equivariant eigenfunctions with negative eigenvalue, and hence it may require more delicate analysis in our setting.
\subsection{Morse index bound} In this section, we apply Savo's Morse index lower bound for minimal hypersurfaces of the sphere by a linear function of its first Betti number in \cite{savo2010index} (See Ambrozio-Carlotto-Sharp \cite{ambrozio2018comparing} for the bound for more general manifolds).
\begin{thm}[\cite{savo2010index}, Corollary 1.2] \label{indexlowerbound} For $n+1 \ge 3$, let $\Sigma^{n} \subset \mathbb{S}^{n+1}$ be a closed and orientable minimal hypersurface on $\mathbb{S}^{n+1}$. Suppose the first Betti number $b_{1}(\Sigma) \ge 1$. Then Morse index of $\Sigma$ has a lower bound as follows.
\[
\operatorname{Index}(\Sigma) \ge \frac{2 b_{1}(\Sigma)}{(n+1)(n+2)} + n+2. 
\]
\end{thm}
By combining Theorem \ref{minhypsphere} and Theorem \ref{indexlowerbound}, we obtain the following Morse index lower bound of min-max minimal hypersurfaces we obtained.
\begin{cor} \label{equivmorse} For dimension $4 \le n+1 \le 7$, there exists a sequence of minimal hypersurfaces $\Sigma_{g}$ on $\mathbb{S}^{n+1}$ such that $\Sigma_{g}$ is a $(SO(n-1) \times \mathbb{D}_{g+1})$-equivariant minimal hypersurface. Moreover, $\Sigma_{g}$ has the following Morse index lower bound.
\begin{align} \label{lowerbdformula}
\operatorname{Index}(\Sigma_{g}) \ge \frac{4g}{(n+1)(n+2)} + n+2. 
\end{align}
    \begin{proof}
       By Lemma \ref{genusg} and Theorem \ref{minhypsphere}, we have the existence of a sequence of $(SO(n-1) \times \mathbb{D}_{g+1})$-equivariant minimal hypersurfaces whose lower bound of the first Betti number is $2g$. By plugging $b_{1}(\Sigma_{g}) \ge 2g$ into \ref{indexlowerbound}, we obtain the lower bound of Morse Index of $\Sigma_{g}$ in (\ref{lowerbdformula}).
    \end{proof}
\end{cor}
\begin{rmk}
We used only $b_{1}(\Sigma_{g}) \ge 2g$ to prove the Morse index lower bound in (\ref{lowerbdformula}). Since $\pi(\Sigma_{g})$ always has at least one boundary component, the lower bound of the first Betti number holds for every $g \ge 1$.
\end{rmk}
\appendix \section{Compactness theorem of $G$-stable embedded minimal hypersurfaces}
In this section, we prove the following compactness result of $G$-stable embedded minimal hypersurfaces that replaces the compactness theorem based on the curvature estimate of stable minimal surfaces of Schoen \cite{schoen1983estimates} in dimension $3$ in \cite{colding2003min}. We assume the assumptions in Section 2.2 along this section. Also, we denote the set of points where $\mathbb{Z}_{2}$-singular axis and $\partial M'$ intersects by $\mathcal{S}_{0, \mathbb{Z}_{2}}$ as in Section $5$.
\begin{rmk} It is tempting to apply Schoen-Simon's sheeting theorem in \cite{schoen1981regularity} for the compactness of stable embedded minimal hypersurfaces (see the recent direct proof of Bellettini \cite{bellettini2023extensions}) in our construction. However, as in Theorem \ref{z2action}, The $G$-stability of replacements does not necessarily imply the stability. Hence, if $spt(\pi_{\sharp}V)$ contains a $\mathbb{Z}_{2}$-singular axis $\gamma \in \tilde{\mathcal{S}}$, then we apply regularity argument of Section 7 in Ketover \cite{ketover2016free}, which is based on the compactness of minimal surfaces with bounded area and genus by Choi-Schoen in \cite{choi1985space}, which also proves the regularity at $x \in \pi^{-1}(\gamma \setminus \partial M')$. 
\end{rmk}

\begin{lem} \label{stablecptness}
    Suppose $G$-invariant open set $U \subset M$ does not intersect $\pi^{-1}(\mathcal{S}_{0, \mathbb{Z}_{2}})$. If $\{ \Sigma_{k} \}$ is a sequence of properly embedded $G$-stable minimal hypersurface in $U$ with uniform area bound, and uniform genus bound of $\pi(\Sigma_{k})$,  then up to the subsequence, $\Sigma_{k}$ converges to a properly embedded $G$-stable minimal hypersurface $\Sigma$.
\begin{proof}
    $G$-stationarity and $G$-stability in the varifold sense directly follow from \cite{allard1972first} and \cite{wang2024equivariant}. It suffices to prove the embeddedness of $\Sigma$. We divide the cases into two cases when $x \in \pi^{-1}(\gamma) \cap \Sigma$, where $\gamma$ is a $\mathbb{Z}_{2}$-singular axis in $\tilde{\mathcal{S}}$, and otherwise.

    Suppose $x$ is not on a $\mathbb{Z}_{2}$-singular axis first. Then let us consider $B^{G}_{\rho} (x) \subset U$ which does not intersect any $\mathbb{Z}_{2}$-singular axis. Then $\Sigma_{k} \cap B^{G}_{\rho} (x)$ is a stable embedded minimal hypersurface by Lemma \ref{equivstablity} and \cite[Proposition 4.6]{ketover2016equivariant}. By Schoen-Simon compactness theorem, we obtain that $\Sigma \cap B^{G}_{\rho} (x)$ is a $G$-equivariant stable minimal hypersurface.

    Now let us consider $x \in \pi^{-1}(\gamma) \cap \Sigma$. By our assumption $\pi(x) \notin \mathcal{S}_{0, \mathbb{Z}_{2}}$, $x$ is not on a singular orbit. Then we consider $B^{G}_{\rho}(x) \subset U$ which does not intersect any singular orbit and $\Sigma_{k} \cap B^{G}_{\rho} (x)$ is a $\mathbb{Z}_{2}$-stable minimal hypersurface by applying the balancing argument of Lemma \ref{equivstablity}. By Theorem \ref{deformationminimal}, $\pi(\Sigma_{k} \cap B_{\rho}^{G} (x))$ is a $\mathbb{Z}_{2}$-stable minimal surface in $(M',g')$. As in \cite[Section 7.2]{ketover2016free}, we apply the compactness theorem of embedded minimal surfaces with bounded area and genus of \cite{choi1985space} and obtain the compactness, where $L^{2}$-bound of the second fundamental form comes out from Ilmanen's integrated Gauss-Bonnet argument in \cite{ilmanen1998lectures}. We obtain the embeddedness of $\pi(\Sigma_{k} \cap B_{\rho}^{G} (x))$ and its preimage accordingly.
\end{proof}
\end{lem}
\bibliographystyle{plain}
\bibliography{references}
\end{document}